\documentclass[%
11pt,
onecolumn,
tightenlines,
superscriptaddress,
preprintnumbers,
nofootinbib,
amsmath,amssymb,amsthm,
physrev,
eqsecnum,
]{revtex4-2}
\pdfoutput=1
\usepackage{isomath}
\usepackage{amsmath,amsthm}
\usepackage{amsbsy}
\usepackage{amssymb}
\usepackage{amscd}
\usepackage{amsfonts}
\usepackage{stmaryrd}
\usepackage{siunitx}
\usepackage{euscript}
\usepackage[utf8]{inputenc}
\usepackage[T1]{fontenc}
\usepackage{newtxtext} 
\everymath{\displaystyle}
\usepackage{exscale}

\usepackage{graphicx}
\usepackage{boxedminipage}
\usepackage{calc}
\usepackage[usenames,dvipsnames]{xcolor}
\graphicspath{ {media/} }
\usepackage[caption=false,justification=centerlast]{subfig}

\usepackage{setspace}
\usepackage{enumitem}
\setitemize{noitemsep,topsep=0pt,parsep=0pt,partopsep=0pt}
\setenumerate{noitemsep,topsep=0pt,parsep=0pt,partopsep=0pt}
\setdescription{noitemsep,topsep=0pt,parsep=0pt,partopsep=0pt}

\usepackage{hyperref}

\usepackage[small]{titlesec}

\titlespacing*{\section}{0pt}{12pt plus 4pt minus 2pt}{2pt plus 2pt minus 2pt}
\titlespacing*{\subsection}{0pt}{12pt plus 4pt minus 2pt}{2pt plus 2pt minus 2pt}
\titlespacing*\subsubsection{0pt}{12pt plus 4pt minus 2pt}{2pt plus 2pt minus 2pt}
\titlespacing*\paragraph{0pt}{12pt plus 4pt minus 2pt}{2pt plus 2pt minus 2pt}

\makeatletter
    \renewcommand*{\thesection}{\arabic{section}}
    \renewcommand*{\thesubsection}{\thesection.\Alph{subsection}}
    \renewcommand*{\p@subsection}{}
    \renewcommand*{\thesubsubsection}{\thesubsection.\arabic{subsubsection}}
    \renewcommand*{\p@subsubsection}{}
\makeatother

\usepackage{isomath}
\usepackage{amsmath}
\usepackage{amssymb}
\usepackage{amscd}
\usepackage{amsfonts}

\newcommand{\R}{\mathbb R}
\newcommand{\N}{\mathbb N}

\newcommand{\eps}{{\varepsilon}}

\newcommand{\half}{\frac{1}{2}}

\newtheorem{theorem}{Theorem}[section]
\newtheorem{lemma}{Lemma}[section]
\newtheorem{proposition}{Proposition}[section]
\newtheorem{remark}{Remark}[section]

\newcommand{\bfsigma}{\mathbold {\sigma}}

\DeclareMathOperator{\trace}{tr}
\DeclareMathOperator*{\argmin}{argmin}

\newcommand{\parderiv}[2]{\frac{\partial #1}{\partial #2}}
\newcommand{\dm}{\ \mathrm{d}}

\newcommand{\bfa}{{\mathbold a}}
\newcommand{\bfb}{{\mathbold b}}

\newcommand{\bfd}{{\mathbold d}}
\newcommand{\bfe}{{\mathbold e}}
\newcommand{\bff}{{\mathbold f}}

\newcommand{\bfn}{{\mathbold n}}

\newcommand{\bft}{{\mathbold t}}

\newcommand{\bfx}{{\mathbold x}}
\newcommand{\bfy}{{\mathbold y}}

\newcommand{\bfA}{{\mathbold A}}

\newcommand{\bfC}{{\mathbold C}}

\newcommand{\bfF}{{\mathbold F}}
\newcommand{\bfG}{{\mathbold G}}

\newcommand{\bfI}{{\mathbold I}}

\newcommand{\bfQ}{{\mathbold Q}}
\newcommand{\bfR}{{\mathbold R}}

\newcommand{\bfT}{{\mathbold T}}
\newcommand{\bfU}{{\mathbold U}}

\usepackage{enumitem}
\usepackage{siunitx}

\newcommand{\e}{\epsilon}
\newcommand{\vareps}{\varepsilon}
\newcommand{\bfvareps}{\mathbold {\varepsilon}}
\DeclareMathOperator{\tr}{tr}
\DeclareMathOperator{\cof}{cof}

\DeclareMathOperator{\spn}{span}
\newtheorem{definition}[theorem]{Definition}

\newcommand{\Wd}{W_{\mathrm{d}}} 
\newcommand{\Wdlin}{W_{\mathrm{d},\mathrm{lin}}} 
\newcommand{\C}{\boldsymbol{\mathsf{C}}} 

\begin{document}


\preprint{Accepted to appear in Journal of the Mechanics and Physics of Solids (DOI: \url{https://doi.org/10.1016/j.jmps.2022.104994})}

\title{\Large{Phase-Field Finite Deformation Fracture with an \\ Effective Energy for Regularized Crack Face Contact}}

\author{Maryam Hakimzadeh}
    \email{mhakimza@andrew.cmu.edu}
    \affiliation{Department of Civil and Environmental Engineering, Carnegie Mellon University}

\author{Vaibhav Agrawal}
    \altaffiliation{Currently at Apple Corporation.}
    \affiliation{Department of Civil and Environmental Engineering, Carnegie Mellon University}

\author{Kaushik Dayal}
    \affiliation{Department of Civil and Environmental Engineering, Carnegie Mellon University}
    \affiliation{Center for Nonlinear Analysis, Department of Mathematical Sciences, Carnegie Mellon University}
    \affiliation{Department of Mechanical Engineering, Carnegie Mellon University}

\author{Carlos Mora-Corral}
    \affiliation{Departamento de Matem\'aticas, Universidad Auton\'oma de Madrid, 28049 Madrid (Spain)}

\date{\today}


\begin{abstract}
	Phase-field models are a leading approach for realistic fracture problems.
	They treat the crack as a second phase and use gradient terms to smear out the crack faces, enabling the use of standard numerical methods for simulations. 
	This regularization causes cracks to occupy a finite volume in the reference, and leads to the inability to appropriately model the closing or contacting -- without healing -- of crack faces.
	Specifically, the classical idealized crack face tractions are that the shear component is zero, and that the normal component is zero when the crack opens and identical to the intact material when the crack closes.
    Phase-field fracture models do not replicate this behavior.

	This work addresses this shortcoming by introducing an effective crack energy density that endows the regularized (finite volume) phase-field crack with the effective properties of an idealized sharp crack. 
	The approach is based on applying the QR (upper triangular) decomposition of the deformation gradient tensor in the basis of the crack, enabling a transparent identification of the crack deformation modes. 
	By then relaxing over those modes that do not cost energy, an effective energy is obtained that has the intact response when the crack faces close and zero energy when the crack faces are open.
	The effective energy is often, but not always, consistent with the classical crack-face tractions; it is shown here that there generally does not exist a stored energy that is consistent both with the classical crack-face tractions and with reproducing the intact response when the crack closes.
	
	A highlight of this approach is that it lies completely in the setting of finite deformation, enabling potential application to soft materials and other settings with large deformation or rotations.
	The model is applied to numerically study representative complex loadings, including (1) cyclic loading on a cavity in a soft solid that shows the growth and closing of cracks in complex stress states; and (2) cyclic shear that shows a complex pattern of crack branching driven by the closure of cracks.
\end{abstract}

\maketitle

\begin{figure*}[h!]
	\includegraphics[width=0.66\textwidth]{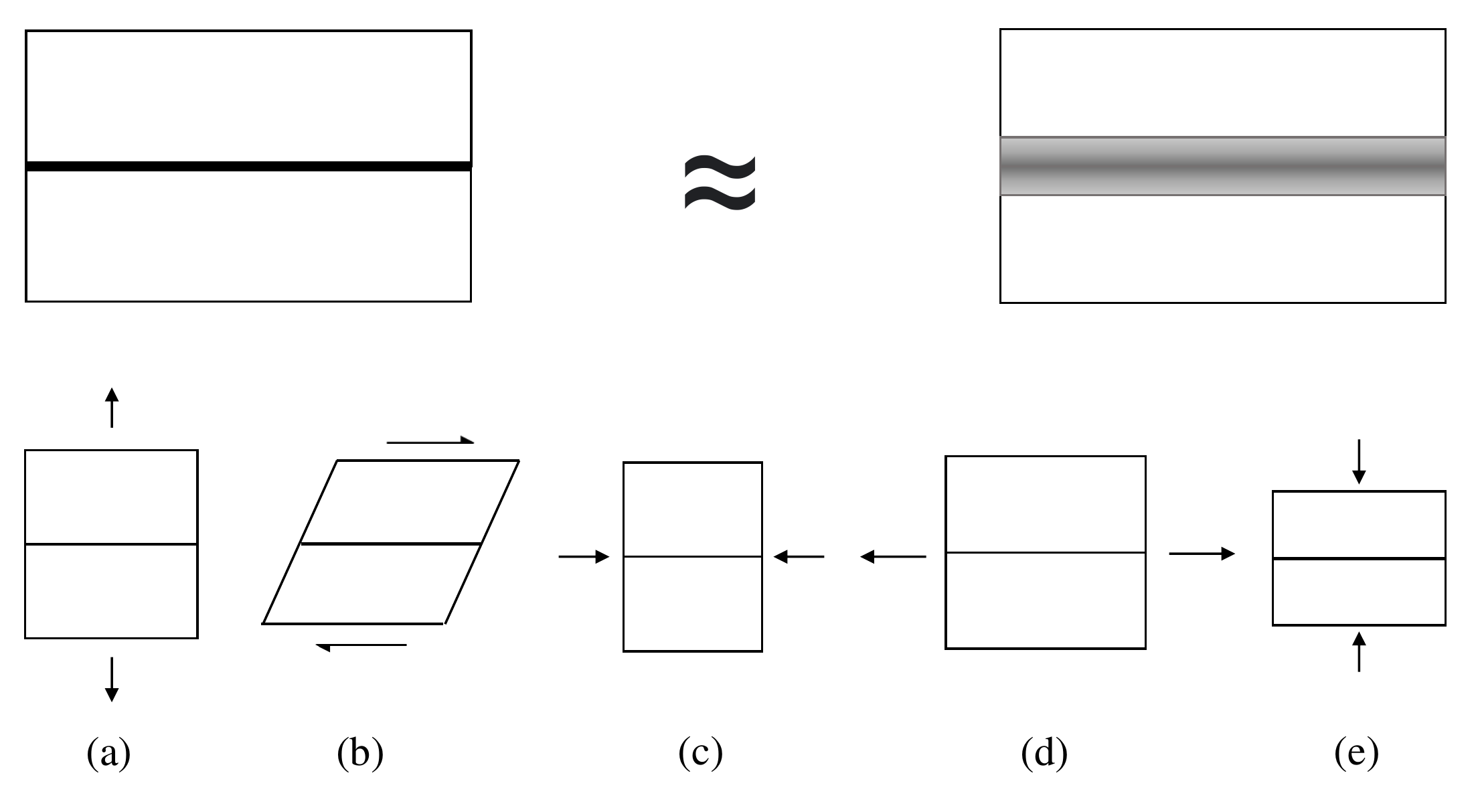}
	\caption{The goal is to formulate an effective crack volume energy such that the regularized crack volume (top right) has the same response as the sharp crack (top left).
	The QR decomposition enables us to clearly separate the deformation into distinct modes, and we then minimize over modes (a) and (b), while retaining the energy associated with modes (c), (d), and (e).}
    \label{fig:overview}
\end{figure*}


\section{Introduction}

Phase-field approaches to fracture regularize the singular crack and thereby enable easy numerical calculations for complex fracture problems.
The key idea is to introduce an additional scalar field $\phi$ and then regularize the field through the introduction of gradient terms, of the form $|\nabla\phi|^2$, in the energy.
The field $\phi$ tracks the level of damage in the domain, and the crack corresponds to regions that are completely damaged.
The regularization term $|\nabla\phi|^2$ is consequently related to the energy associated with the formation of fracture surfaces.
This regularization enables the use of standard computational methods, specifically the finite element method (FEM), to approximate the evolution of cracks in a specimen under load.

The regularization, however, causes cracks to occupy a finite volume in the reference configuration, and leads to unphysical behavior when an existing crack is subject to loads that causes the crack faces to close\footnote{
    We use ``close'' to denote crack faces that come in contact and {\em not} that the crack heals; we do not allow cracks to heal in this work.
}.
For instance, the earliest phase-field models of fracture used simply that the damaged region has zero elastic energy; however, this causes cracks in this model to grow even under compressive load as it provides a means of relieving elastic energy of all kinds, not just those that drive crack growth in real systems.
In addition, these models allow for interpenetration of the crack faces.
To address these issues, \cite{miehe-ijnme} proposed a partition of the elastic energy into tensile and compressive parts, and allowed the damaged region to sustain only the compressive part.
Related to this, \cite{amor-jmps} proposed a partition of the energy into compressive hydrostatic, tensile hydrostatic, and deviatoric parts, and allowed the damaged region to sustain only the compressive hydrostatic part.
However, this class of approaches has a key shortcoming: the energy partitioning does not consider the crack direction at all.
For instance, tension across the crack faces that drives crack growth is not distinguished from tension along the length of the crack that does not directly drive crack growth.
Appendix \ref{sec:Vaibhav-comparison} describes specific instances where this class of approaches gives incorrect stresses; in turn, this can lead to incorrect driving forces for crack growth and affect the crack-parallel T-stress.

Based on the recognition that it is essential to consider the crack orientation in defining the energy of the damaged region, \cite{strobl2015,strobl2016,agrawal-thesis,steinke2019phase,fei2020phase} proposed models, in the linear regime, for the damaged elastic energy that account for the crack orientation.
The current work builds on these approaches in that we account for the crack normal in defining the energy density of the damaged region.
Figure \ref{fig:overview} summarizes our overall approach: we aim to derive an effective crack energy such that a regularized crack volume has the same response as an idealized sharp crack.
We achieve this by separating the kinematics associated with each of the deformation modes in Figure \ref{fig:overview}, and then appropriately treating each mode.

Our technical strategy in brief is as follows.
Rather than starting with the typical polar decomposition of the deformation gradient $\bfF$, we use the QR or Gram–Schmidt multiplicative decomposition of $\bfF$ into a rotation and an upper-triangular part in the frame of the crack.
The upper triangular representation in the frame of the crack provides a transparent and direct measure of the opening or closing of the crack face, the stretch along the length of the crack, and the crack face shear deformation.
We then obtain the damaged elastic energy by minimizing the intact energy over the crack face shear -- using that the idealized crack cannot sustain shear tractions -- and minimizing it over the crack opening stretch if the crack is opening.
This strategy provides a constructive approach to obtaining the effective damaged energy given the form of the intact energy.
We demonstrate the efficacy of this approach in numerical calculations: first, of the elastic response of cracks that do not evolve; and, second, of the response and growth of cracks in complex settings that include compressive and cyclic loadings that cause crack closure.

\subsection{The Classical Phase-Field Model of Fracture}\label{subse:classical}

The starting point of phase-field modeling of brittle fracture is the influential variational model of quasistatic crack evolution due to Francfort and Marigo \cite{FrMa98}.
In its simplest form, it consists in minimizing the energy
\begin{equation}\label{eq:introEnergy}
     E[\bfy, \Gamma] = \int_{\Omega \setminus \Gamma} W(\nabla \bfy) \dm V_{\bfx} + G_c \mathcal{H}^{n-1} (\Gamma)
\end{equation}
over all admissible deformations $\bfy\colon \Omega \to \mathbb{R}^n$ and cracks $\Gamma \subset \Omega$.
Here $\Omega$ is an open set of $\mathbb{R}^n$ ($n=2,3$) representing the body in its reference configuration, $\mathcal{H}^{n-1}$ is the surface measure, $G_c$ is the toughness constant or work to fracture, and $W$ is the stored energy density of the material.
The evolution of the system is typically given by an external loading that depends on the time $t$, whose energy is added to \eqref{eq:introEnergy} and make up the total energy of the system.
The proposed evolution law postulates that at each time $t$, the pair deformation-crack $(\bfy (t), \Gamma (t))$ minimizes the total energy and that the crack set $\Gamma (t)$ is nondecreasing with $t$.
This model is inspired by Griffith's \cite{Griffith21} theory of fracture, but its main advantage is that the crack path is not specified \emph{a priori} but rather it is selected by energy minimization.

This model can be recast as a free-discontinuity problem in the same way that the Mumford--Shah \cite{MuSh89} model for image segmentation was recast as a free-discontinuity model by \cite{DeCaLe89}.
In fact, the earlier work by Ambrosio and Braides \cite{AmBr95} shows the following preliminary version of a free-discontinuity model that encompasses elastic and fracture energies, but in the static case.
As typical in free-discontinuity problems \cite{AmFuPa00}, the main idea is to unify the pair deformation-crack $(\bfy, \Gamma)$ not as a Sobolev function $\bfy$ defined in $\Omega \setminus \Gamma$ together with a crack set, but as an $SBV$ function $\bfy$ defined in $\Omega$ whose jump set $J_{\bfy}$ is $\Gamma$.
In this way, the only variable is the deformation $\bfy$ and the energy \eqref{eq:introEnergy} is substituted by
\begin{equation}\label{eq:introEnergySBV}
     E[\bfy] = \int_{\Omega} W(\nabla \bfy) \dm V_{\bfx} + G_c \mathcal{H}^{n-1} (J_{\bfy}) .
\end{equation}
When the time-dependence is taken into account, and, hence, the irreversibility of the crack reflecting that cracks do not heal, the term $\mathcal{H}^{n-1} (J_{\bfy})$ is substituted by $\mathcal{H}^{n-1} (J_{\bfy} \cup K(t))$, where $K(t)$ is the union of all previous crack sets of the deformation, i.e., $\bigcup_{s<t} J_{\bfy (s)}$.
The existence of a quasistatic evolution for this model was first proved by Francfort and Larsen \cite{FrLa03} and Dal Maso, Francfort and Toader \cite{DaFrTo05}, and then underwent many generalizations.

A direct approach to the numerical minimization of the functional is intractable using standard methods.
A fruitful procedure is the construction of an approximating sequence of elliptic functionals that $\Gamma$-converge to the functional to approximate (see, e.g., \cite{Braides98}).
Inspired by the result of Modica and Mortola \cite{MoMo77,Modica87}, Ambrosio and Tortorelli \cite{AmTo90,AmTo92} introduced an approximation of the Mumford--Shah model, which, in its vectorial version, turns out to be an approximation of \eqref{eq:introEnergySBV}.
It reads as follows:
\begin{equation}\label{eq:introAT}
    E[\bfy,\phi] 
    = 
    \int_{\Omega} \left( \phi^2 + \eta_{\e} \right) W (\nabla \bfy) \dm V_{\bfx} + G_c \int_{\Omega} \left( \frac{(1-\phi)^2}{4\epsilon} + \epsilon |\nabla\phi|^2 \right) \dm V_\bfx .
\end{equation}
Here $\phi$ is a new variable, the crack indicator field, and $\eta_{\e}$ is an infinitesimal that goes to zero faster than $\e$.
The field $\phi$ satisfies $0 \leq \phi \leq 1$ everywhere and when $\phi (\bfx) \simeq 1$ it signals that the material at $\bfx$ is healthy, whereas $\phi (\bfx) \simeq 0$ means that the material at $\bfx$ is damaged.
The number $\eta_{\e}$ makes the first integral of \eqref{eq:introAT} elliptic, and, hence it avoids the degeneracy in the regions when $\phi=0$.
The $\Gamma$-convergence of \eqref{eq:introAT} to \eqref{eq:introEnergySBV}  was proved (under different assumptions on $W$) by Focardi \cite{Focardi01} and Chambolle \cite{Chambolle04} (see also \cite{BrChSo07,HeMoXu15}).
As a consequence of the $\Gamma$-convergence result, as $\e \to 0$, minimizers $(\bfy_{\e}, \phi_{\epsilon})$ at the level $\e$ of the functional \eqref{eq:introAT} tend to $(\bfy, 1)$, where $\bfy$ is a minimizer of \eqref{eq:introEnergySBV}.

When time-independence is put into the model, the irreversibility of the crack is translated into the restriction that $\phi(t)$ is nondecreasing with $t$.

Numerical studies and experiments for this model can be found in \cite{BoCh00,BoFrMa00,Giacomini05,Bourdin07,BuOrSu10}.
See also the review paper \cite{BoFrMa08}.
The model has also been widely characterized and applied in the mechanics community; a sample from this vast literature include \cite{clayton2014geometrically,clayton2015nonlinear,da2013sharp,lo2019phase,ambati2015review,SUN2021101277,agrawal2017dependence,wu2020phase,kamensky2018hyperbolic,moutsanidis2018hyperbolic,diehl2022comparative,clayton2021nonlinear}; particularly, we highlight the important work of \cite{abdollahi2012phase} that deals with developing effective crack energies in the context of electrical boundary conditions on the crack face.

\subsection{The proposed model}\label{subse:proposed}

We make some modifications to the basic model \eqref{eq:introAT}.
First, we add an additional field that keeps track of the orientation of the crack\footnote{
    The essentially-similar idea of accounting for crack orientations and level of damage through tensorial internal variables was introduced in continuum damage mechanics a few decades ago, e.g., \cite{murakami1988mechanical}, \cite[Section 7.3.1]{lemaitre1994mechanics}; we thank Pradeep Sharma for mentioning this to us.
}; this is a vector field $\bfd$, which is imposed to have norm $|\bfd| \leq 1$.
A value $|\bfd| \simeq 0$ indicates that the material is healthy, whereas $|\bfd| \simeq 1$ indicates that the material is damaged (cracked).
The direction $\bfn := \bfd / |\bfd|$ indicates the normal to the crack.
We will explain later how $\bfd$ evolves in a quasistatic evolution.
The comparison between this $\bfd$ and the $\phi$ of model \eqref{eq:introAT} is not direct, but, in a sense, $|\bfd|$ plays the role of $1-\phi$, while $\bfd$ should be parallel to $\nabla \phi$.

We will also introduce a new energy term $\Wd$ depending on $\nabla \bfy$ and $\bfn$ that is only activated in the cracked region.
The subscript $\mathrm{d}$ in $\Wd$ stands for `damaged'.
The dependence on $\bfn$ makes it possible to distinguish the different deformation modes, according to whether they are of extension/compression type or shear type with respect to the crack.
Thus, $\Wd$ replaces $W$ as the energy in the cracked region.
The fictional effective material that is placed in the ``crack volume'' in the reference configuration is no longer taken to have zero elastic energy (as in the classical model explained in Section \ref{subse:classical}), but instead has nonzero elastic energy for specific deformation modes and zero elastic energy for other modes, as sketched in Figure \ref{fig:overview}.
The modes are distinguished through the local orientation of the crack (obtained from $\bfn$) and the elastic energy of the nonzero modes are set up to match the elastic response of the original material $W$.
The specific form of $\Wd$ will be described in Section \ref{se:W1}.

The volume energy will be a convex combination between the bulk energy $W$ and the cracked energy $\Wd$.
We impose the convex restriction $|\bfd| \leq 1$, and, in analogy with \eqref{eq:introAT}, the volume energy will be
\begin{equation}\label{eq:bulk1}
 \int_{\Omega} \left( (1- |\bfd|)^2 W(\nabla\bfy) + \left(1- (1- |\bfd|)^2\right) \, \Wd\left(\nabla\bfy,\bfn\right) \right) \dm V_{\bfx} .
\end{equation}
In the regions where the material is healthy ($|\bfd| \simeq 0$) the bulk energy is essentially $W$, and where the material is damaged ($|\bfd| \simeq 1$) the bulk energy is essentially $\Wd$.
When the material is healthy, we have $|\bfd| \simeq 0$ so the unit vector $\bfn$ is undefined or ill-defined, but this is not a problem because $\Wd$ is multiplied by $(1- (1- |\bfd|)^2)$.

In addition, we will add an infinitesimal $\eta_{\eps} >0$ to the factor in $W$ of the bulk energy in \eqref{eq:bulk1}, as in the Ambrosio--Tortorelli formulation.
The presence of this term $\eta_{\epsilon}$ prevents the loss of ellipticity (i.e., the degeneracy) in a region where the material is totally cracked (when $|\bfd| \simeq 1$).
This infinitesimal has the property $\eta_{\epsilon} \ll \epsilon^{p-1}$, and is imposed to ensure the $\Gamma$-convergence of this model to the sharp-interface model of fracture \eqref{eq:introEnergySBV}, as in \cite{Focardi01,Chambolle04,BrChSo07,HeMoXu15}.
Here $p$ is the growth exponent of $W$ at infinity (basically, the exponent of $|\bfF|$ in the expression of $W$), which does not play an important role for the moment; for example, for a Mooney--Rivlin material, $p=2$.

We will also add the (Modica--Mortola or Ambrosio--Tortorelli) term
\[
 \int_{\Omega} \left( \frac{|\bfd|^2}{2\epsilon} + \frac{\epsilon}{2} |\nabla \bfd|^2 \right) \dm V_{\bfx} ,
\]
which is expected to converge to the sharp-energy term as $\eps \to 0$.

All in all, the proposed energy is
\begin{equation}
\label{eqn:full-model}
\begin{split}
    & E[\bfy,\bfd] =
    \\
    & \int_{\Omega} \left( ( (1- |\bfd|)^2 + \eta_{\epsilon}) \, W(\nabla\bfy) + \left(1- (1- |\bfd|)^2\right) \, \Wd\left(\nabla\bfy,\bfn\right) \right) \dm V_{\bfx} 
 	+ 
	G_c \int_{\Omega} \left( \frac{|\bfd|^2}{2\epsilon} + \frac{\epsilon}{2} |\nabla \bfd|^2 \right) \dm V_{\bfx}
\end{split}
\end{equation}
with the restriction $|\bfd| \leq 1$.

\paragraph*{Notation.}

Vector, matrices and higher order tensors are written in boldface.

The set of $3 \times 3$ matrices with positive determinant is denoted by $\R^{3 \times 3}_+$, while $SO(3)$ stands for its subset of rotations.
Analogous notation is used in dimension $2$.

The inverse of an invertible matrix $\bfF$ is $\bfF^{-1}$, its transpose is $\bfF^T$ and the inverse of its transpose is $\bfF^{-T}$.
Its determinant is $\det \bfF$, its cofactor $\cof \bfF$, which satisfies $\cof \bfF = (\det \bfF) \bfF^{-T}$.
Its norm $|\bfF|$ is the square root of $\sum_{ij} F_{ij}^2$.

Given two vectors $\bfa, \bfb$, its tensor product $\bfa \otimes \bfb$ is the matrix with components $(a \otimes b)_{ij} = a_i b_j$.

\paragraph*{Structure of the paper.}

In Section \ref{se:W1}, we define the effective energy $\Wd$ and show its main properties. As a preliminary, we recall the QR decomposition of a matrix.
In Section \ref{se:W1stress}, we recall the classical crack face traction condition for smooth frictionless cracks and compare it with the condition satisfied by $\Wd$.
Sections \ref{se:Mooney}, \ref{se:pq} and \ref{se:ExampleGeneral} are devoted to the calculation of examples of effective energies $\Wd$ given a stored energy $W$.
We treat both the 2D and the 3D cases.
Precisely, in Section \ref{se:Mooney}, we deal with a Mooney--Rivlin energy, in Section \ref{se:pq} with a $(p,q)$-energy (a generalization of Mooney--Rivlin allowing for general -- not necessarily quadratic -- exponents $p$ and $q$), and in Section \ref{se:ExampleGeneral} with a very general energy $W$; in the latter case, of course, we cannot give an explicit expression for $\Wd$.
In Section \ref{se:linear}, we develop the theory for small strain: given an energy $W$, we build the effective energy $\Wd$, and then approximate it by neglecting terms that are higher-order than quadratic in the strain. 
In Section \ref{se:numerical}, we describe the numerical implementation, with special emphasis on the irreversibility of the crack.
In Section \ref{se:calculations}, we present some numerical examples; all of these are in the setting of large deformations.
Section \ref{se:discussion} is the concluding discussion.
In Appendix \ref{sec:Vaibhav-comparison}, we discuss some deficiencies of the energy-splitting method, which, in fact, was one of the motivation to construct the energy $\Wd$ of this article.
In Appendix \ref{ap:proofs}, we collect the proofs of all results stated in the article.

\section{Effective energy for regularized cracks}\label{se:W1}

In this section we define the effective energy $\Wd$ given an energy $W$.
More precisely, in Subsection \ref{subse:motivation} we describe the  properties that such an effective energy should have.
In Subsection \ref{subse:QR} we recall the QR decomposition of a matrix, which provides us with a language to express the desirable properties of $\Wd$.
We also recall the concept of frame-indifference for the energy density.
In Subsection \ref{subse:relaxation} we define $\Wd$ as a minimization of $W$ over certain modes.
We also state the main properties of $\Wd$.
Subsection \ref{subse:multiple} is a remark about the graph of $\Wd (\bfF, \bfn)$ in terms of $\bfn$: while the definition of $\Wd$ was done so that $\Wd (\bfF, \bfn) = 0$ for certain orientations $\bfn$ (depending on $\bfF$), we show that, in addition, sometimes $\Wd (\bfF, \bfn_1) = 0$ for other $\bfn_1$.

\subsection{Motivation and heuristics}\label{subse:motivation}

The term $\Wd$ is the main contribution of this work, and will make a crucial difference between our model and the original Ambrosio--Tortorelli model and variants (see Subsection \ref{subse:classical}).
The starting idea for a definition of $\Wd$ is the following.
Imagine that $W$ is isotropic and a crack has been formed.
Then, close to the crack the effective response should not be isotropic anymore, and this non-isotropy will be detected by the effective energy $\Wd$.

This $\Wd$ has two variables: the deformation gradient $\nabla \bfy$ and a unit vector $\bfn$, which is expected to be normal to the crack.
The definition of $\Wd$ should be such that, after the crack has been formed, some basic deformation modes (shear, compresion, extension) will or will not have energy, according to their orientation relative to the crack.
In this heuristic explanation, we assume implicitly that $W \geq 0$ and $W(\bfI) = 0$.
Specifically, we want the following basic modes to carry no effective energy:
\begin{enumerate}[label=(\alph*)]

\item\label{item:eP} Extension perpendicular to the crack.

\item\label{item:sT} Shear parallel to the crack.
\end{enumerate}
On the other hand, we expect the following basic modes to have positive energy:
\begin{enumerate}[label=(\alph*),resume]
\item\label{item:ccT} Compression parallel to the crack.

\item\label{item:ecT} Extension parallel to the crack.

\item\label{item:cP} Compression perpendicular to the crack.
\end{enumerate}
Specifically, modes \ref{item:ccT} and \ref{item:cP} should have the same energy as $W$, while mode \ref{item:ecT}, slightly less so as to take into account the energy necessary to increase the volume of the body; in the linear setting, this is equivalent to considering that the Poisson ratio of the material is typically strictly positive.  
Any other mode will carry some effective energy, but less than $W$, so $0 < \Wd < W$.

See Figure \ref{fig:modes} for a 2D representation of modes \ref{item:eP}--\ref{item:cP}.

\begin{figure}[h!]
\centering
    \begin{tabular}{c|ccccc}
         Deformation mode & (a) & (b) & (c) & (d) & (e) \\
         \hline
        \raisebox{2em}{Loading}
         &{\includegraphics[width=0.1\textwidth, height = 0.1\textwidth]{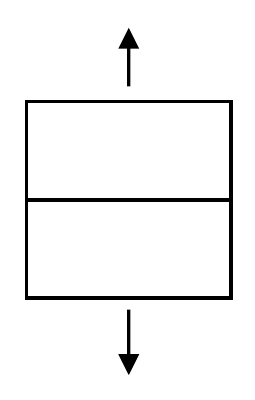}}
         & 	{\includegraphics[width=0.1\textwidth, height = 0.1\textwidth]{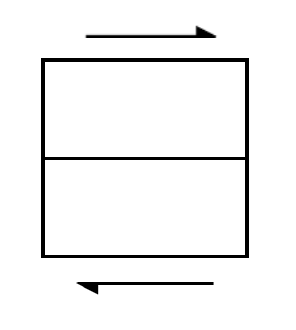}}
         & 	{\includegraphics[width=0.14\textwidth, height = 0.1\textwidth]{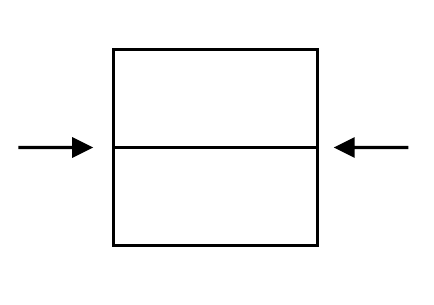}}
         & {\includegraphics[width=0.13\textwidth, height = 0.1\textwidth]{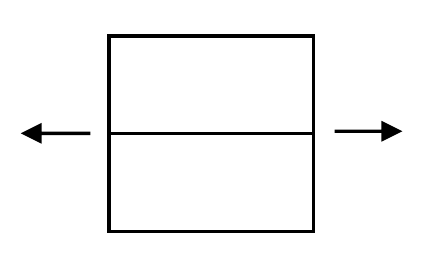}}
         & \includegraphics[width=0.1\textwidth, height = 0.12\textwidth]{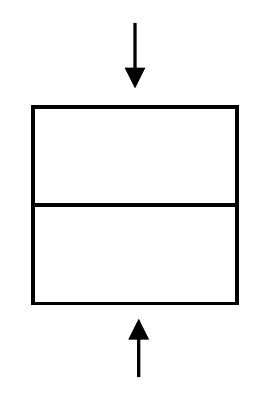}
         \\        \hline
         \raisebox{2em}{Intact Response}
         & {\includegraphics[width=0.1\textwidth, height = 0.1\textwidth]{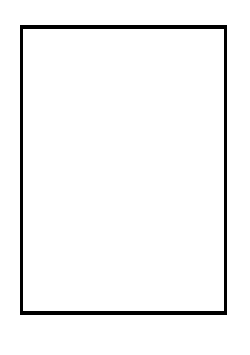}}
         & {\includegraphics[width=0.1\textwidth, height = 0.1\textwidth]{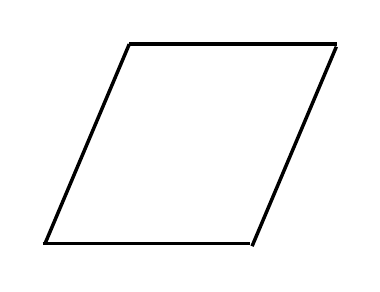}}
         & {\includegraphics[width=0.07\textwidth, height = 0.1\textwidth]{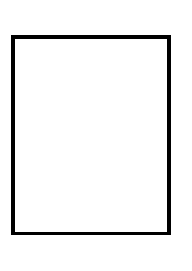}}
         & {\includegraphics[width=0.1\textwidth, height = 0.1\textwidth]{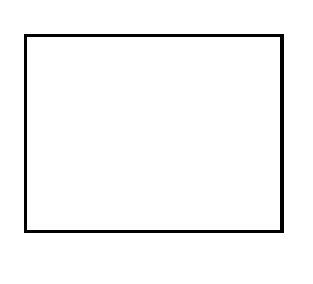}}
         & {\includegraphics[width=0.1\textwidth, height = 0.08\textwidth]{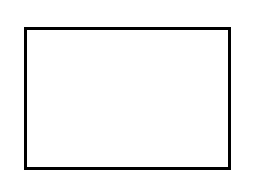}}
         \\ \hline
         \raisebox{2em}{Crack Response}
         &{\includegraphics[width=0.1\textwidth, height = 0.1\textwidth]{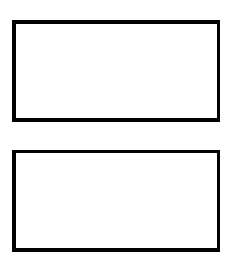}}
         & {\includegraphics[width=0.1\textwidth, height = 0.1\textwidth]{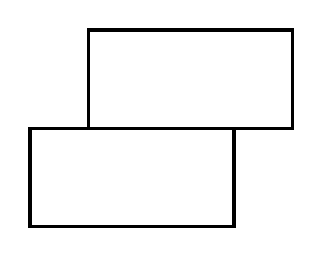}}
         & {\includegraphics[width=0.07\textwidth, height = 0.1\textwidth]{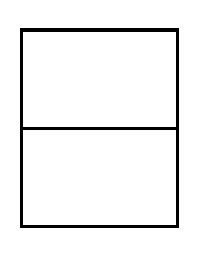}}
         & {\includegraphics[width=0.1\textwidth, height = 0.1\textwidth]{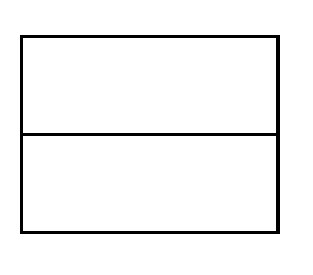}}
         & {\includegraphics[width=0.1\textwidth, height = 0.08\textwidth]{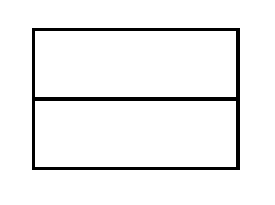}} \\
    \end{tabular}
	\caption{The top row shows different loadings, and the middle and lower rows show the idealized deformation for intact and cracked specimens, respectively. Based on this idealization, we assign zero energy to modes (a) and (b).}
\label{fig:modes}
\end{figure}

\subsection{QR decomposition and frame-indifference}\label{subse:QR}

In this section we state the key representation result of the deformation gradient that will be used to formulate the definition of $\Wd$.
Henceforth we will use the following notation for triangular matrices: for an orthonormal basis $\{ \bft_1, \bft_2, \bfn \}$, and numbers $A_{\bfn \bfn} , A_{\bft_1 \bft_1}, A_{\bft_2 \bft_2} > 0$ and $A_{\bft_1 \bfn}, A_{\bft_2 \bfn}, A_{\bft_1 \bft_2} \in \R$, we define
\begin{equation}\label{eq:At1t2n}
\begin{split}
 & \bfA_{\bft_1, \bft_2, \bfn} (A_{\bfn \bfn}, A_{\bft_1 \bft_1}, A_{\bft_2 \bft_2}, A_{\bft_1 \bfn}, A_{\bft_2 \bfn}, A_{\bft_1 \bft_2}) \\
 & := A_{\bfn \bfn} \bfn \otimes \bfn + A_{\bft_1 \bft_1} \bft_1 \otimes \bft_1 + A_{\bft_2 \bft_2} \bft_2 \otimes \bft_2 + A_{\bft_1 \bfn} \bft_1 \otimes \bfn + A_{\bft_2 \bfn} \bft_2 \otimes \bfn + A_{\bft_1 \bft_2} \bft_1 \otimes \bft_2 .
\end{split}
\end{equation}
If the basis considered is the canonical basis $\{ \bfe_1, \bfe_2, \bfe_3 \}$, instead of $A_{\bfe_i \bfe_j}$ we will write $A_{ij}$.

\begin{proposition}\label{prop:QR}
For any $\bfF \in \R^{3 \times 3}_+$ and any orthonormal basis $\{ \bft_1, \bft_2, \bfn \}$ there exist unique
\[
 \bfR \in SO(3) , \qquad A_{\bfn \bfn} , A_{\bft_1 \bft_1}, A_{\bft_2 \bft_2} > 0 , \qquad A_{\bft_1 \bfn}, A_{\bft_2 \bfn}, A_{\bft_1 \bft_2} \in \R
\]
such that
\[
 \bfF = \bfR \bfA_{\bft_1, \bft_2, \bfn} (A_{\bfn \bfn}, A_{\bft_1 \bft_1}, A_{\bft_2 \bft_2}, A_{\bft_1 \bfn}, A_{\bft_2 \bfn}, A_{\bft_1 \bft_2}) .
\]
\end{proposition}

We remark that QR decompositions have been applied for creating constitutive equations in elasticity in the past few years \cite{paul2021coordinate,clayton2020constitutive,Srinivasa12,freed2020laplace}, where this approach was shown to provide important advantages.
The coefficients of an upper triangular matrix provide a clear interpretation in terms of compression, extension and shear in a specified frame, as opposed to the typical polar decomposition.
With this language, we can describe the modes  \ref{item:eP}--\ref{item:cP} precisely:

\begin{enumerate}[label=(\alph*)]
    
    \item
    $\bfF = A_{\bfn \bfn} \bfn \otimes \bfn + \bft_1 \otimes \bft_1 + \bft_2 \otimes \bft_2$ with $A_{\bfn \bfn} \geq 1$.
    
    \item
    $\bfF = \bfI + A_{\bft_1 \bfn} \bft_1 \otimes \bfn + A_{\bft_2 \bfn} \bft_2 \otimes \bfn$ with $A_{\bft_1 \bfn}, A_{\bft_2 \bfn} \in \R$.
    
    \item
    $\bfF = \bfn \otimes \bfn + A_{\bft_1 \bft_1} \bft_1 \otimes \bft_1 + A_{\bft_2 \bft_2} \bft_2 \otimes \bft_2$ with $0 < A_{\bft_1 \bft_1}, A_{\bft_2 \bft_2} < 1$.
    
    \item
    $\bfF = \bfn \otimes \bfn + A_{\bft_1 \bft_1} \bft_1 \otimes \bft_1 + A_{\bft_2 \bft_2} \bft_2 \otimes \bft_2$ with $A_{\bft_1 \bft_1}, A_{\bft_2 \bft_2} \geq 1$.
    
    \item
    $\bfF = A_{\bfn \bfn} \bfn \otimes \bfn + \bft_1 \otimes \bft_1 + \bft_2 \otimes \bft_2$ with $0 < A_{\bfn \bfn} < 1$.

\end{enumerate}

The explicit expression of $\bfR$ and $\bfA$ of the QR decomposition is cumbersome.
Nevertheless, the coefficient $A_{\bfn \bfn}$, which will be the most relevant in the sequel, can be given an easy expression.

\begin{lemma}\label{le:Ann}
    Let $\bfF \in \R^{3 \times 3}_+$ and let $\{ \bft_1, \bft_2, \bfn \}$ be an orthonormal basis.
    Let $\bfF = \bfR \bfA$ be its QR decomposition with respect to that basis, according to Proposition \ref{prop:QR}.
    Then
    \begin{equation}\label{eq:Ann}
      A_{\bfn \bfn} = \frac{1}{\left| \bfF^{-T} \, \bfn \right|}.
    \end{equation}
\end{lemma}

For the sake of completeness, we recall that a function $W\colon\R^{3 \times 3}_+ \to \R$ is \emph{frame-indifferent} when $W (\bfF) = W (\bfR \bfF)$ for all $\bfF \in \R^{3 \times 3}_+$ and $\bfR \in SO(3)$.
Since frame-indifference is a physical requirement we will henceforth impose that our stored-energy function $W$ is frame-indifferent. 

Naturally, $\Wd$ must be frame-indifferent, too.
Since the crack (and, hence $\bfn$) is viewed in the reference configuration, the frame-indifference for $\Wd$ takes the form
\begin{equation}\label{eq:Wdindifference}
 \Wd (\bfF, \bfn) = \Wd (\bfR \bfF, \bfn) , 
\end{equation}
for all $\bfF \in \R^{3 \times 3}_+$, all $\bfR \in SO(3)$ and all unit vectors $\bfn$.
In addition, as $\bfn$ is determined up to a sign, $\Wd$ must satisfy
\begin{equation}\label{eq:Wdinvariance}
 \Wd (\bfF, \bfn) = \Wd (\bfF, -\bfn) .
\end{equation}

\begin{remark}[Frame Indifference in terms of the Cauchy-Green tensor]
    An equivalent way of expressing frame-indifference is to require that $W(\bfF)$ be able to be written as $W^\bfC(\bfC) = W(\bfF)$, where $\bfC=\bfF^T \bfF$; equivalently, that $W(\bfF)$ can be expressed as $W^\bfU(\bfU) = W(\bfF)$, where $\bfU$ is the tensor square root of $\bfC$ (the symmetric positive definite matrix of the polar decomposition of $\bfF$).

    Here, we propose to work with the following decomposition of $\bfC$:
    \begin{equation}
        \bfA^T \bfA = \bfC, \text{ where } \bfA = \bfA_{\bft_1, \bft_2, \bfn} (A_{\bfn \bfn}, A_{\bft_1 \bft_1}, A_{\bft_2 \bft_2}, A_{\bft_1 \bfn}, A_{\bft_2 \bfn}, A_{\bft_1 \bft_2}) .
    \end{equation}
    From Proposition \ref{prop:QR}, $\bfA$ is unique given the orthonormal basis, i.e., $\bfA$ is a function of $\bfC$ given $\bft_1, \bft_2, \bfn$.
    Therefore, any energy function based on $\bfA$ is frame-indifferent.
    \hfill\qedsymbol
\end{remark}

\subsection{Relaxation over modes}\label{subse:relaxation}

In this section, given $W$, we provide a definition of $\Wd$, based on relaxation over modes.
We recapitulate the properties we require for $\Wd$: assuming that $W\geq 0$ and $W(\bfI) = 0$, the effective energy must have zero enery for modes \ref{item:eP}--\ref{item:sT} of Section \ref{subse:motivation}, the same energy as $W$ for modes \ref{item:ccT} and \ref{item:cP}, and an intermediate energy for mode \ref{item:ecT} and, in fact, any other mode.
Moreover, it must satisfy \eqref{eq:Wdindifference} and \eqref{eq:Wdinvariance}.
In addition, since it is an effective energy it should be $0 \leq \Wd \leq W$.

Among the many ways to define an energy density $\Wd$ with the properties above, we propose one based on minimization over modes, as we will develop in the following paragraphs.

\begin{proposition}\label{pr:alpha*}
Let $W \colon \R^{3 \times 3}_+ \to \R$ be continuous and satisfy that
\begin{equation}\label{eq:Winfty}
 W (\bfF) \to \infty \quad \text{as } \det \bfF \to 0 \text{ or } |\bfF| \to \infty .
\end{equation}
Let $\{\bft_1, \bft_2, \bfn\}$ be an orthonormal basis of $\R^3$.
Then:
\begin{enumerate}[label=\alph*)]
\item\label{item:alpha*} For each $A_{\bft_1 \bft_1}, A_{\bft_2 \bft_2} >0$ and $A_{\bft_1 \bft_2} \in \R$, the minimum \begin{equation}\label{eq:propertyW}
 \min_{\substack{A_{\bfn \bfn} >0 \\ A_{\bft_1 \bfn}, A_{\bft_2 \bfn} \in \R}} W \left( \bfA_{\bft_1, \bft_2, \bfn} (A_{\bfn \bfn}, A_{\bft_1 \bft_1}, A_{\bft_2 \bft_2}, A_{\bft_1 \bfn}, A_{\bft_2 \bfn}, A_{\bft_1 \bft_2}) \right)
\end{equation}
exists.

\item For each $A_{\bft_1 \bft_1}, A_{\bft_2 \bft_2}, A_{\bfn \bfn}>0$ and $A_{\bft_1 \bft_2} \in \R$, the minimum
\[
 \min_{A_{\bft_1 \bfn}, A_{\bft_2 \bfn} \in \R} W \left( \bfA_{\bft_1, \bft_2, \bfn} (A_{\bfn \bfn}, A_{\bft_1 \bft_1}, A_{\bft_2 \bft_2}, A_{\bft_1 \bfn}, A_{\bft_2 \bfn}, A_{\bft_1 \bft_2}) \right)
\]
exists.

\item For each $A_{\bft_1 \bft_1}, A_{\bft_2 \bft_2} >0$ and $A_{\bft_1 \bft_2} \in \R$, the minimum
\begin{multline*}
 A_{\bfn \bfn}^* = \min \Big\{ \bar{A}_{\bfn \bfn} > 0: \text{there exist } A_{\bft_1 \bfn}^* \text{ and } A_{\bft_2 \bfn}^* \text{ such that } \\
 \inf_{\substack{A_{\bfn \bfn} >0 \\ A_{\bft_1 \bfn}, A_{\bft_2 \bfn} \in \R}} W \left( \bfA_{\bft_1, \bft_2, \bfn} (A_{\bfn \bfn}, A_{\bft_1 \bft_1}, A_{\bft_2 \bft_2}, A_{\bft_1 \bfn}, A_{\bft_2 \bfn}, A_{\bft_1 \bft_2}) \right) \\
 = W \left( \bfA_{\bft_1, \bft_2, \bfn} (\bar{A}_{\bfn \bfn}, A_{\bft_1 \bft_1}, A_{\bft_2 \bft_2}, A_{\bft_1 \bfn}^*, A_{\bft_2 \bfn}^*, A_{\bft_1 \bft_2}) \right) \Big\}
\end{multline*}
exists.
\end{enumerate}
\end{proposition}

We are in a position to define the effective energy.

\begin{definition}\label{de:Wd}
Let $\{ \bft_1, \bft_2, \bfn \}$ be an orthonormal basis of $\R^3$.
Let $W \colon \R^{3 \times 3}_+ \to \R$ be continuous, frame-indifferent and satisfy \eqref{eq:Winfty}.
Given $\bfF \in \R^{3 \times 3}_+$, let $\bfF = \bfR \bfA$ be the QR decomposition of $\bfF$ with respect to the basis $\{ \bft_1, \bft_2, \bfn \}$, with
\[
 \bfA = \bfA_{\bft_1, \bft_2, \bfn} (A_{\bfn \bfn}, A_{\bft_1 \bft_1}, A_{\bft_2 \bft_2}, A_{\bft_1 \bfn}, A_{\bft_2 \bfn}, A_{\bft_1 \bft_2}) .
\]
Let $A_{\bfn \bfn}^*$ be as in Proposition \ref{pr:alpha*}.
We define
\[
 \Wd \left( \bfF, \bfn \right) = \begin{cases}
 \min_{\substack{A'_{\bfn \bfn} >0 \\ A'_{\bft_1 \bfn}, A'_{\bft_2 \bfn} \in \R}} W \left( \bfA_{\bft_1, \bft_2, \bfn} (A'_{\bfn \bfn}, A_{\bft_1 \bft_1}, A_{\bft_2 \bft_2}, A'_{\bft_1 \bfn}, A'_{\bft_2 \bfn}, A_{\bft_1 \bft_2}) \right) , & \text{if } A_{\bfn \bfn} \geq A_{\bfn \bfn}^* , \\
 \min_{A'_{\bft_1 \bfn}, A'_{\bft_2 \bfn} \in \R} W \left( \bfA_{\bft_1, \bft_2, \bfn} (A_{\bfn \bfn}, A_{\bft_1 \bft_1}, A_{\bft_2 \bft_2}, A'_{\bft_1 \bfn}, A'_{\bft_2 \bfn}, A_{\bft_1 \bft_2}) \right) , & \text{if } A_{\bfn \bfn} < A_{\bfn \bfn}^* .
 \end{cases}
\]
\end{definition}

The next result provides an alternative definition of $\Wd$ that does not rely so heavily on the QR decomposition; in fact, only on $A_{\bfn \bfn}$, which, in turn, can be given the closed-form formula of Lemma \ref{le:Ann}.
Thus, we provide a way of computing $\Wd$ without using Definition \ref{de:Wd}, which in some situations is useful for computational and theoretical purposes.

\begin{proposition}\label{pr:Wdalt}
Let $\{ \bft_1, \bft_2, \bfn \}$ be an orthonormal basis of $\R^3$.
Let $W : \R^{3 \times 3}_+ \to \R$ be continuous, frame-indifferent and satisfy \eqref{eq:Winfty}.
Given $\bfF \in \R^{3 \times 3}_+$, let $A_{\bfn \bfn}$ be as in \eqref{eq:Ann}.
Then
\begin{multline}\label{eq:Anns*}
 A_{\bfn \bfn}^* = A_{\bfn \bfn} \times \min \bigg\{ \bar{A}_{\bfn \bfn}''>0 : \text{there exist } \bar{A}_{\bft_1 \bfn}'', \bar{A}_{\bft_1\bfn}'' \in \R \text{ such that } \\
 \min_{\substack{A''_{\bfn \bfn} >0 \\ A''_{\bft_1 \bfn}, A''_{\bft_2 \bfn} \in \R}} \! \! W \left( \bfF \bfA_{\bft_1, \bft_2, \bfn} \left( A''_{\bfn \bfn}, 1, 1, A''_{\bft_1 \bfn}, A''_{\bft_2 \bfn}, 0 \right) \right) = W \left( \bfF \bfA_{\bft_1, \bft_2, \bfn} \left( \bar{A}''_{\bfn \bfn}, 1, 1, \bar{A}''_{\bft_1 \bfn}, \bar{A}''_{\bft_2 \bfn}, 0 \right) \right) \bigg\} 
\end{multline}
and
\[
 \Wd \left( \bfF, \bfn \right) = \begin{cases}
 \min_{\substack{A''_{\bfn \bfn} >0 \\ A''_{\bft_1 \bfn}, A''_{\bft_2 \bfn} \in \R}} W \left( \bfF \bfA_{\bft_1, \bft_2, \bfn} \left( A''_{\bfn \bfn}, 1, 1, A''_{\bft_1 \bfn}, A''_{\bft_2 \bfn}, 0 \right) \right) , & \text{if } A_{\bfn \bfn} \geq A_{\bfn \bfn}^* , \\
 \min_{A''_{\bft_1 \bfn}, A''_{\bft_2 \bfn} \in \R} W \left( \bfF \bfA_{\bft_1, \bft_2, \bfn} \left( A_{\bfn \bfn}, 1, 1, A''_{\bft_1 \bfn}, A''_{\bft_2 \bfn}, 0 \right) \right) , & \text{if } A_{\bfn \bfn} < A_{\bfn \bfn}^* .
 \end{cases}
\]
\end{proposition}

The following result shows the invariance properties of $\Wd$.
We recall that the group $\mathcal{S}$ of symmetries of $W$ is the set of $\bfQ \in SO(3)$ such that $W (\bfF \bfQ) = W (\bfF)$ for all $\bfF \in \R^{3 \times 3}_+$, and that a material is isotropic if $\mathcal{S} = SO(3)$.

\begin{proposition}\label{pr:invariance}
Let $\{ \bft_1, \bft_2 , \bfn \}$ be an orthonormal basis and let $\bfF \in \R^{3 \times 3}_+$.
Let $W \colon \R^{3 \times 3}_+ \to \R$ be continuous, frame-indifferent and satisfy \eqref{eq:Winfty}.
Then:
\begin{enumerate}[label=\alph*)]
\item\label{item:Wcompa2} $\Wd (\bfF, \bfn)$ does not depend on $\bft_1, \bft_2$.

\item\label{item:Wcompa3} $\Wd (\bfF, \bfn) = \Wd (\bfF, -\bfn)$

\item\label{item:Wcompa4} $\Wd (\bfQ \bfF, \bfn) = \Wd (\bfF, \bfn)$ for all $\bfQ \in SO(3)$.

\item\label{item:Wcompa45} If $\mathcal{S} \subset SO(3)$ is the group of symmetries of $W$ then $\Wd (\bfF \bfQ^T, \bfQ \bfn) = \Wd (\bfF, \bfn)$ for all $\bfQ \in \mathcal{S}$.

\item\label{item:Wcompa5} If $W$ is isotropic then $\Wd (\bfF \bfQ^T, \bfQ \bfn) = \Wd (\bfF, \bfn)$ for all $\bfQ \in SO(3)$.
\end{enumerate}
\end{proposition}

Of course, this proposition describes desirable properties for an effective energy $\Wd$: properties \ref{item:Wcompa2}--\ref{item:Wcompa3} show that $\Wd$ depends on the direction perpendicular to the crack (and not on the sense or the other elements of the orthonormal basis), while property \ref{item:Wcompa4} expresses its frame-indifference; see \eqref{eq:Wdindifference} and \eqref{eq:Wdinvariance}.
Properties \ref{item:Wcompa45}--\ref{item:Wcompa5} express that the symmetries of $W$ are transferred to $\Wd$, but with a caveat: when the deformation gradient $\bfF$ changes to $\bfF \bfQ^T$ then normal $\bfn$ has to change to $\bfQ \bfn$.
Indeed, one of the original motivations for $\Wd$ was that the isotropy of $W$ is broken when the deformation gradient $\bfF$ changes to $\bfF \bfQ^T$ and the normal $\bfn$ remains unchanged.

\subsection{Low Energy Crack Orientations from the Effective Energy}\label{subse:multiple} 

Typically, we expect cracks to be oriented such that the crack normal is aligned with the local dominant tensile direction.
In our model, this is achieved if, given $\bfF$, we have that $\Wd\left(\bfF,\bfn\right)$ is minimized for $\bfn$ aligned with the dominant tensile direction.
We examine this question here, and find that the expected direction is a global minimum but there exist other local minima.

In order to simplify the exposition, we work in dimension $2$.
In Section \ref{subse:NeoH2D} we will compute the effective energy $\Wd$ of the stored energy
\[
 W (\bfF) = \frac{\mu}{2} \left( |\bfF|^2 -2 - 2 \log \det \bfF \right) + \frac{\lambda}{2} \left( \det \bfF -1 \right)^2 .
\]

For a given $\bfF$, we parametrize $\bfn = (\cos \theta, \sin \theta)$ for $0 \leq \theta \leq \pi$, because of the invariance of Proposition \ref{pr:invariance}.\ref{item:Wcompa3}.
We plot $\Wd (\bfF, \bfn)$ against $\theta$.
The results, for the following two choices of $\bfF$ are as follows.

In the first case we let
\[
 \bfF = \begin{pmatrix}
 1 & 0 \\
 0 & F_{22}
 \end{pmatrix}
\]
with $F_{22} \geq 1$, which represents an extension perpendicular to the crack when $\bfn = \bfe_2$, and expect that the minimum of $\Wd (\bfF, \bfn)$ is only attained at $\bfn = \bfe_2$.
Particularizing the formulas of Section \ref{subse:NeoH2D} for this $\bfF$, we obtain that, when we define
\[
 A_{11} = \sqrt{n_2^2 + F_{22}^2 n_1^2} , \qquad 
 A_{22}^* = \frac{ \lambda A_{11} + \sqrt{4 \mu^2 + 4 \mu \lambda  A_{11}^2 + \lambda^2 A_{11}^2}}{2 (\mu + \lambda A_{11}^2)} ,
\]
the formula for the relaxed energy is
\[
 \Wd (\bfF, \bfn) = 
 \frac{\mu}{2} \left( A_{11}^2 + (A_{22}^*)^2 -2 - 2 \log (A_{11} A_{22}^*) \right) + \frac{\lambda}{2} \left( A_{11} A_{22}^* -1 \right)^2 .
\]
Numerically, one can check that the only minimum is attained at $\bfn = \bfe_2$, as expected.
See Figure \ref{fi:oneminimum}, left, for the graph of $\Wd (\bfF, \bfn)/\mu$ as a function of $\theta$, with $F_{22} = \frac{3}{2}$ and $\lambda = \mu = 1$.
\begin{figure}
    \includegraphics[width=.4\textwidth]{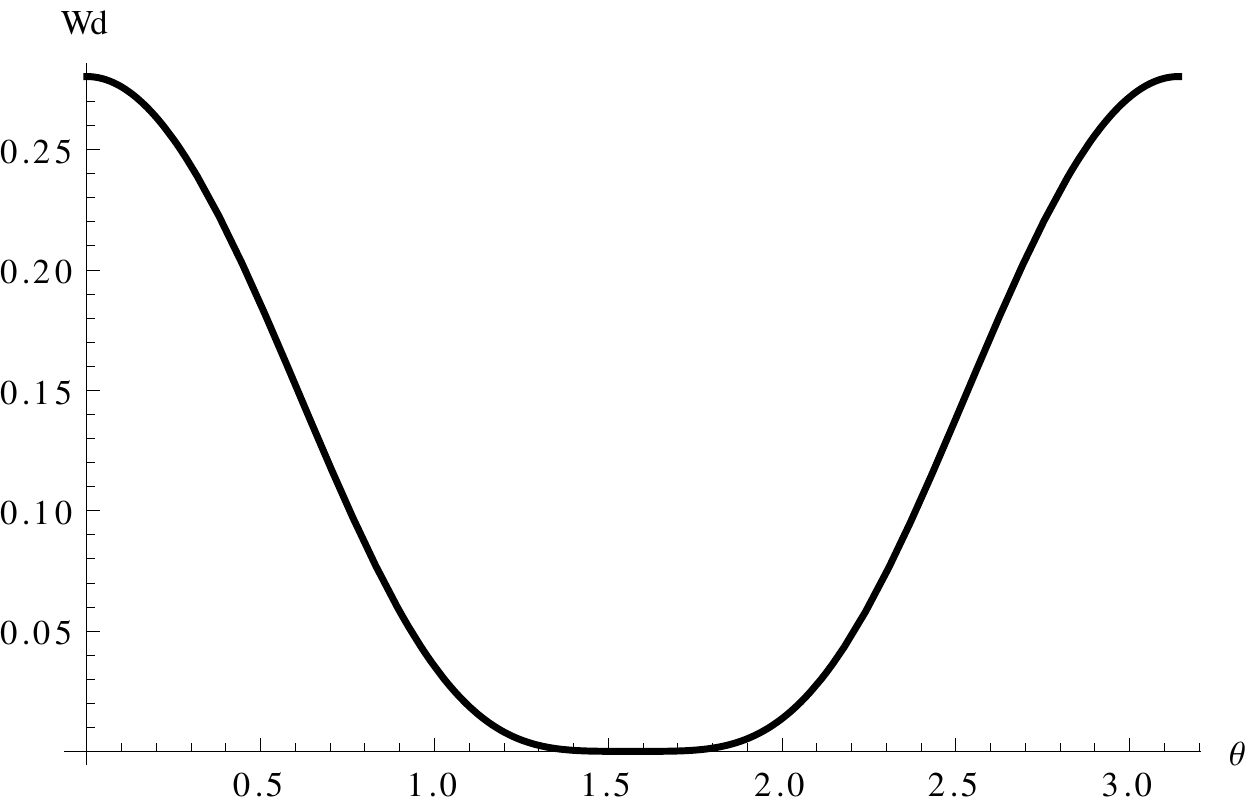} \hfill
    \includegraphics[width=.4\textwidth]{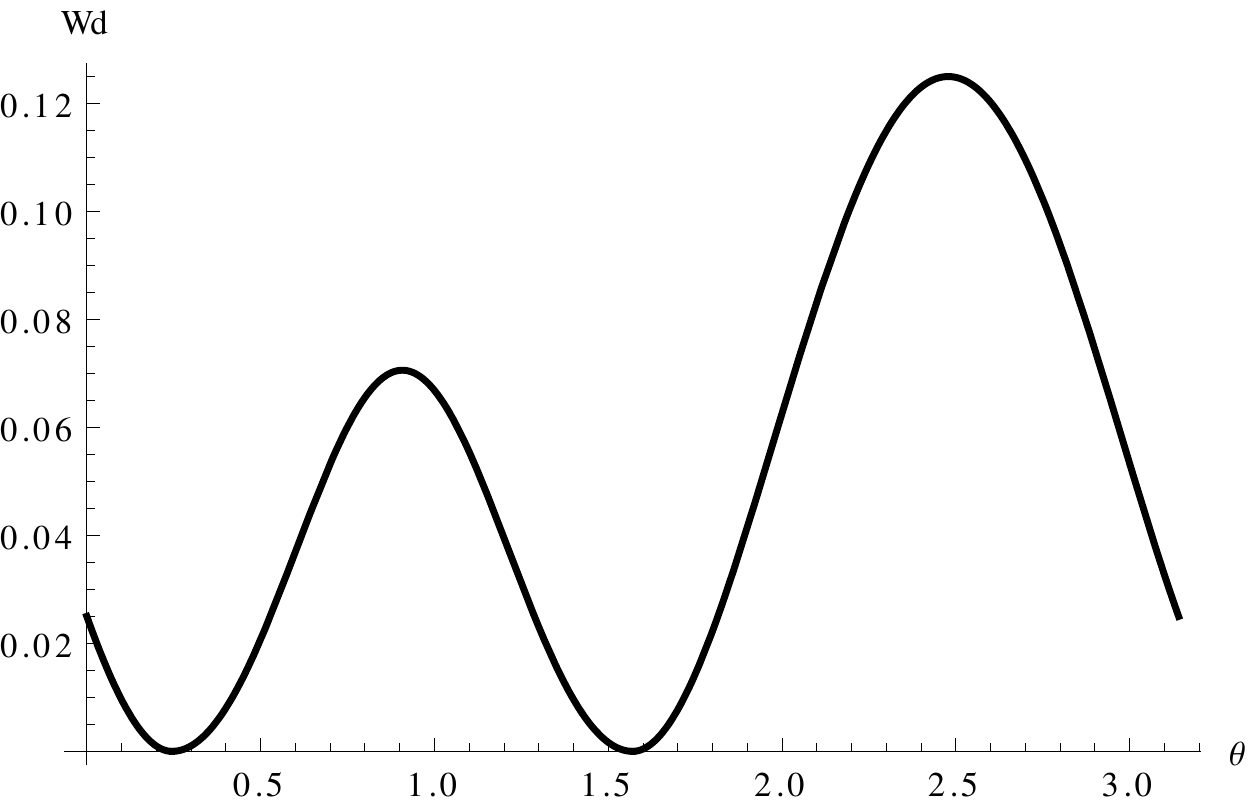}
    \caption{Graph of $\Wd (\bfF, \bfn)/\mu$ as a function of $\theta$, with $\lambda = \mu = 1$ and $\bfn = (\cos \theta, \sin \theta)$.
    Left: $\bfF = \bfI + \frac{1}{2} \bfe_2 \otimes \bfe_2$; the minimum is attained at $\theta = \frac{\pi}{2}$.
    Right: $\bfF = \bfI + \frac{1}{2} \bfe_1 \otimes \bfe_2$; the minimum is attained at $\theta = \frac{\pi}{2}$ and at another value depending on $F_{12}$.
    \label{fi:oneminimum}}
\end{figure}

In the second example we let
\[
 \bfF = \begin{pmatrix}
 1 & F_{12} \\
 0 & 1
 \end{pmatrix}
\]
with $F_{12} \in \R$, which represents a shear parallel to the crack when $\bfn = \bfe_2$, and expect that the minimum of $\Wd (\bfF, \bfn)$ is only attained at $\bfn = \bfe_2$.
Particularizing the formulas of Section \ref{subse:NeoH2D} for this $\bfF$, we obtain that, when we define
\[
  A_{11} = \sqrt{1 - 2 F_{12} n_1 n_2 + F_{12}^2 n_1^2} , \qquad A_{22} = \frac{1}{A_{11}} , \qquad
 A_{22}^* = \frac{ \lambda A_{11} + \sqrt{4 \mu^2 + 4 \mu \lambda  A_{11}^2 + \lambda^2 A_{11}^2}}{2 (\mu + \lambda A_{11}^2)} ,
\]
the formula for the relaxed energy is
\[
 \Wd (\bfF, \bfn) = \begin{cases}
 \frac{\mu}{2} \left( A_{11}^2 + (A_{22}^*)^2 -2 - 2 \log (A_{11} A_{22}^*) \right) + \frac{\lambda}{2} \left( A_{11} A_{22}^* -1 \right)^2 & \text{if } A_{22} > A_{22}^* , \\
 \frac{\mu}{2} \left( A_{11}^2 + A_{22}^2 -2 \right) & \text{if } A_{22} \leq A_{22}^* .
 \end{cases}
\]
Numerically, one can check that it has two minima: one located at $\bfn = \bfe_2$, as expected, and the other one depending on $F_{12}$ but not on $\lambda$ or $\mu$.
Figure \ref{fi:oneminimum}, right, shows the graph of $\Wd (\bfF, \bfn)/\mu$ as a function of $\theta$ with $F_{12} = \frac{1}{2}$ and $\lambda = \mu = 1$.

\section{The Classical Crack Face Traction Condition}\label{se:W1stress}

The classical crack face traction conditions for a smooth frictionless crack are that the shear traction is zero and that the normal traction is either zero (when the crack is open) or compressive (when the crack is closed).
To this, we append the natural condition that the compressive normal traction, when the crack is closed, is identical to that in the intact material under the same deformation.
The key question in this section is to understand if the proposed energy $\Wd$ satisfies these traction conditions. 
In summary, we find that it does so for some but not all materials.

In terms of the (Piola--Kirchhoff) stress $\bfT^{\mathrm{d}}$ in the damaged material, and the stress $\bfT$ in the intact material, the classical crack face conditions can be written as
\begin{equation}
\label{eqn:classical-traction}
    T_{\bft_1 \bfn}^{\mathrm{d}} = T_{\bft_2 \bfn}^{\mathrm{d}} = 0  
    \quad 
    \text{and} \quad T_{\bfn \bfn}^{\mathrm{d}} = \min \{ T_{\bfn \bfn} , 0 \} .
\end{equation}
The notation for subindices is so that $T^{\mathrm{d}}_{\bft_i \bfn}$ indicates the shear components of the traction, while $T^{\mathrm{d}}_{\bfn \bfn}$ denotes the normal component of the traction, and analogously for $\bfT$.
Of course, the stresses are the derivatives of the corresponding elastic energy densities.

A potential strategy to construct $\Wd$ that satisfies \eqref{eqn:classical-traction} could be to simply use \eqref{eqn:classical-traction} as the starting point, and integrate appropriately to construct the corresponding stored energy density.
By QR decomposition and frame-indifference, in order for \eqref{eqn:classical-traction} to be satisfied for all deformation gradients, it is enough that it holds for those of the form \eqref{eq:At1t2n}.
Standard aguments based on the symmetry of the second derivative show that, given $W$ (hence $\bfT$), an energy $\Wd$ exists satisfying \eqref{eqn:classical-traction} if and only if
\begin{equation}\label{eq:compatibility}
     \frac{\partial^2}{\partial A_{\bfn \bfn} \, \partial A_{\bft_i \bfn}} W \left( \bfA_{\bft_1, \bft_2, \bfn} (A_{\bfn \bfn}, A_{\bft_1 \bft_1}, A_{\bft_2 \bft_2}, A_{\bft_1 \bfn}, A_{\bft_2 \bfn}, A_{\bft_1 \bft_2}) \right) = 0 , \qquad i = 1, 2.
\end{equation}

\subsection{The Traction Condition is Generally Incompatible with a Stored Energy Density}

An interpretation of condition \eqref{eq:compatibility} is that the dependence of $W$ on the $A_{\bfn \bfn}$ and $A_{\bft_i \bfn}$ components is somewhat uncoupled.
This condition is not satisfied for every $W$ but it is satisfied, for example, for Mooney--Rivlin materials.
Indeed, let the stored energy density $W$ take the form
\[
 W(\bfF) = a \left| \bfF \right|^2 + b \left| \cof \bfF \right|^2 + h (\det \bfF) 
\]
with $a, b >0$ and any function $h$.
Then, with respect to any orthonormal basis $\{ \bft_1, \bft_2, \bfn \}$,
\begin{align*}
 & W \left( \bfA_{\bft_1, \bft_2, \bfn} (A_{\bfn \bfn}, A_{\bft_1 \bft_1}, A_{\bft_2 \bft_2}, A_{\bft_1 \bfn}, A_{\bft_2 \bfn}, A_{\bft_1 \bft_2}) \right) \\
 & = a \left( A_{\bft_1 \bft_1}^2 + A_{\bft_2 \bft_2}^2 + A_{\bfn \bfn}^2 + A_{\bft_1 \bft_2}^2 + A_{\bft_1 \bfn}^2 + A_{\bft_2 \bfn}^2 \right) \\
 & \quad + b \left( A_{\bft_1 \bft_1}^2 A_{\bfn \bfn}^2 + A_{\bft_2 \bft_2}^2 A_{\bfn \bfn}^2 + A_{\bft_1 \bft_1}^2 A_{\bft_2 \bft_2}^2 + A_{\bft_1 \bft_1}^2 A_{\bft_2 \bfn}^2 + A_{\bfn \bfn}^2 A_{\bft_1 \bft_2}^2 + \left( A_{\bft_2 \bft_2} A_{\bft_1 \bfn} - A_{\bft_2 \bfn} A_{\bft_1 \bft_2} \right)^2 \right) \\
 & \quad + h (A_{\bft_1 \bft_1} A_{\bft_2 \bft_2} A_{\bfn \bfn}) ,
\end{align*}
which is readily seen to satisfy \eqref{eq:compatibility}.

On the other hand, if we consider exponents $p, q >0$ with $(p, q) \neq (2,2)$ then it is easy to check that the energy
\[
 W(\bfF) = a \left| \bfF \right|^p + b \left| \cof \bfF \right|^q + h (\det \bfF) 
\]
does not satisfy \eqref{eq:compatibility}.
In summary, it shows that, for general energies, there do not exist effective crack energy densities that both satisfy the crack face traction conditions and have the correct intact response when the crack closes.

\subsection{Traction on the crack face satisfied by the effective energy}

In this subsection we calculate the tractions $T_{\bft_i \bfn}^{\mathrm{d}}$ and $T_{\bfn \bfn}^{\mathrm{d}}$ for the effective energy constructed in Definition \ref{de:Wd}, and compare them with \eqref{eqn:classical-traction}.

Let $A_{\bfn \bfn}, A_{\bft_1 \bft_1}, A_{\bft_2 \bft_2} >0$ and $A_{\bft_1 \bfn}, A_{\bft_2 \bfn}, A_{\bft_1 \bft_2} \in \R$.
From Definition \ref{de:Wd} we have
\begin{align*}
 & \Wd \left( \bfA_{\bft_1, \bft_2, \bfn} (A_{\bfn \bfn}, A_{\bft_1 \bft_1}, A_{\bft_2 \bft_2}, A_{\bft_1 \bfn}, A_{\bft_2 \bfn}, A_{\bft_1 \bft_2}), \bfn \right) \\
 & = \begin{cases} W \left( \bfA_{\bft_1, \bft_2, \bfn} (A^*_{\bfn \bfn}, A_{\bft_1 \bft_1}, A_{\bft_2 \bft_2}, A^*_{\bft_1 \bfn}, A^*_{\bft_2 \bfn}, A_{\bft_1 \bft_2}) \right) , & \text{if } A_{\bfn \bfn} \geq A_{\bfn \bfn}^* , \\
  W \left( \bfA_{\bft_1, \bft_2, \bfn} (A_{\bfn \bfn}, A_{\bft_1 \bft_1}, A_{\bft_2 \bft_2}, A^{**}_{\bft_1 \bfn}, A^{**}_{\bft_2 \bfn}, A_{\bft_1 \bft_2}) \right) , & \text{if } A_{\bfn \bfn} < A_{\bfn \bfn}^* ,
 \end{cases}
\end{align*}
for some functions
\begin{align*}
 & A^*_{\bfn \bfn} = A^*_{\bfn \bfn} (A_{\bft_1 \bft_1}, A_{\bft_2 \bft_2}, A_{\bft_1 \bft_2}), \qquad A^*_{\bft_i \bfn} = A^*_{\bft_i \bfn} (A_{\bft_1 \bft_1}, A_{\bft_2 \bft_2}, A_{\bft_1 \bft_2}) \quad (i=1,2), \\
 & A^{**}_{\bft_i \bfn} = A^{**}_{\bft_i \bfn} (A_{\bfn \bfn}, A_{\bft_1 \bft_1}, A_{\bft_2 \bft_2}, A_{\bft_1 \bft_2}) \quad (i=1,2) .
\end{align*}
Assume, for simplicity, that the functions $A^*_{\bfn \bfn}, A^*_{\bft_i \bfn}, A^{**}_{\bft_i \bfn}$ can be defined uniquely, at least locally; we will see in Section \ref{se:ExampleGeneral} an example of this situation.
The derivatives of $\Wd$ in the directions $(\bft_i , \bfn)$ and $(\bfn , \bfn)$ are
\begin{equation}\label{eq:tractionWdtn}
 \parderiv{\Wd}{A_{\bft_i \bfn}} \left( \bfA_{\bft_1, \bft_2, \bfn} (A_{\bfn \bfn}, A_{\bft_1 \bft_1}, A_{\bft_2 \bft_2}, A_{\bft_1 \bfn}, A_{\bft_2 \bfn}, A_{\bft_1 \bft_2}), \bfn \right) 
 = 0, \qquad i=1,2 
\end{equation}
and
\begin{equation}\label{eq:tractionWdnn}
\begin{split}
 & \parderiv{\Wd}{A_{\bfn \bfn}} \left( \bfA_{\bft_1, \bft_2, \bfn} (A_{\bfn \bfn}, A_{\bft_1 \bft_1}, A_{\bft_2 \bft_2}, A_{\bft_1 \bfn}, A_{\bft_2 \bfn}, A_{\bft_1 \bft_2}), \bfn \right) \\
 & = \begin{cases} 0 & \text{if } A_{\bfn \bfn} > A_{\bfn \bfn}^* , \\
 \parderiv{W}{A_{\bfn \bfn}} \left( \bfA_{\bft_1, \bft_2, \bfn} (A_{\bfn \bfn}, A_{\bft_1 \bft_1}, A_{\bft_2 \bft_2}, A^{**}_{\bft_1 \bfn}, A^{**}_{\bft_2 \bfn}, A_{\bft_1 \bft_2}) \right) & \text{if } A_{\bfn \bfn} < A_{\bfn \bfn}^* ,
 \end{cases}
\end{split}
\end{equation}
since
\[
 \parderiv{W}{A_{\bft_i \bfn}} \left( \bfA_{\bft_1, \bft_2, \bfn} (A_{\bfn \bfn}, A_{\bft_1 \bft_1}, A_{\bft_2 \bft_2}, A^{**}_{\bft_1 \bfn}, A^{**}_{\bft_2 \bfn}, A_{\bft_1 \bft_2}) \right) = 0 , \qquad i=1,2
\]
because $A^{**}_{\bft_1 \bfn}, A^{**}_{\bft_2 \bfn}$ are minimizers of $W \left( \bfA_{\bft_1, \bft_2, \bfn} (A_{\bfn \bfn}, A_{\bft_1 \bft_1}, A_{\bft_2 \bft_2}, A_{\bft_1 \bfn}, A_{\bft_2 \bfn}, A_{\bft_1 \bft_2}) \right)$.

We compare these results with the crack frace conditions \eqref{eqn:classical-traction}.
Let $\bfT$ and $\bfT^{\mathrm{d}}$ be as in the beginning of this section.
Formulas \eqref{eq:tractionWdtn} and \eqref{eq:tractionWdnn} show that, for strains of the form $\bfA_{\bft_1, \bft_2, \bfn}$ (see \eqref{eq:At1t2n}),
\begin{equation}\label{eq:traction1}
 T_{\bft_1 \bfn}^{\mathrm{d}} = T_{\bft_2 \bfn}^{\mathrm{d}} = 0
\end{equation}
and
\begin{equation}\label{eq:traction2}
 T_{\bfn \bfn}^{\mathrm{d}} = \begin{cases} 0 & \text{if } A_{\bfn \bfn} > A_{\bfn \bfn}^* , \\
 T_{\bfn \bfn} \left( \bfA_{\bft_1, \bft_2, \bfn} (A_{\bfn \bfn}, A_{\bft_1 \bft_1}, A_{\bft_2 \bft_2}, A^{**}_{\bft_1 \bfn}, A^{**}_{\bft_2 \bfn}, A_{\bft_1 \bft_2}) \right) & \text{if } A_{\bfn \bfn} < A_{\bfn \bfn}^* .
 \end{cases}
\end{equation}
As we can see, \eqref{eq:traction1} corresponds to the first part of \eqref{eqn:classical-traction}, but in general,  \eqref{eq:traction2} differs from the second part of \eqref{eqn:classical-traction}.
Nevertheless, in the particular case that 
\begin{equation}\label{eq:particular}
 \begin{cases}
 0 \leq T_{\bfn \bfn} & \text{if } A_{\bfn \bfn} > A_{\bfn \bfn}^* , \\
  T_{\bfn \bfn} \left( \bfA_{\bft_1, \bft_2, \bfn} (A_{\bfn \bfn}, A_{\bft_1 \bft_1}, A_{\bft_2 \bft_2}, A^{**}_{\bft_1 \bfn}, A^{**}_{\bft_2 \bfn}, A_{\bft_1 \bft_2}) \right) = T_{\bfn \bfn} \leq 0 & \text{if } A_{\bfn \bfn} < A_{\bfn \bfn}^* ,
 \end{cases}
\end{equation}
we have that conditions \eqref{eq:traction1}--\eqref{eq:traction2} are equivalent to \eqref{eqn:classical-traction}.

Let us have a closer look to \eqref{eq:particular}.
As $A_{\bfn \bfn}^*, A^*_{\bft_1 \bfn}, A^*_{\bft_2 \bfn}$ are minimizers of $W \left( \bfA_{\bft_1, \bft_2, \bfn} \right)$ (where $\bfA_{\bft_1, \bft_2, \bfn}$ is as in \eqref{eq:At1t2n}), we have that
\[
 T_{\bfn \bfn} \left( \bfA_{\bft_1, \bft_2, \bfn} (A^*_{\bfn \bfn}, A_{\bft_1 \bft_1}, A_{\bft_2 \bft_2}, A^*_{\bft_1 \bfn}, A^*_{\bft_2 \bfn}, A_{\bft_1 \bft_2}) \right) = 0 .
\]
A natural condition would then be that
\[
 T_{\bfn \bfn} \left( \bfA_{\bft_1, \bft_2, \bfn} (A_{\bfn \bfn}, A_{\bft_1 \bft_1}, A_{\bft_2 \bft_2}, A^*_{\bft_1 \bfn}, A^*_{\bft_2 \bfn}, A_{\bft_1 \bft_2}) \right) \geq 0 \quad \text{when } A_{\bfn \bfn} > A^*_{\bfn \bfn} ,
\]
and
\[
 T_{\bfn \bfn} \left( \bfA_{\bft_1, \bft_2, \bfn} (A_{\bfn \bfn}, A_{\bft_1 \bft_1}, A_{\bft_2 \bft_2}, A^*_{\bft_1 \bfn}, A^*_{\bft_2 \bfn}, A_{\bft_1 \bft_2}) \right) \leq 0 \quad \text{when } A_{\bfn \bfn} < A^*_{\bfn \bfn} ,
\]
while only in the few situations where the expression of the energy $W$ has an uncoupled dependence on the terms $A_{\bfn \bfn}, A_{\bft_1 \bfn}, A_{\bft_2 \bfn}$ we additionally have that
\[
  T_{\bfn \bfn} \left( \bfA_{\bft_1, \bft_2, \bfn} (A_{\bfn \bfn}, A_{\bft_1 \bft_1}, A_{\bft_2 \bft_2}, A_{\bft_1 \bfn}, A_{\bft_2 \bfn}, A_{\bft_1 \bft_2}) \right) \geq 0 \quad \text{when } A_{\bfn \bfn} > A^*_{\bfn \bfn} ,
\]
and
\begin{align*}
 & T_{\bfn \bfn} \left( \bfA_{\bft_1, \bft_2, \bfn} (A_{\bfn \bfn}, A_{\bft_1 \bft_1}, A_{\bft_2 \bft_2}, A^{**}_{\bft_1 \bfn}, A^{**}_{\bft_2 \bfn}, A_{\bft_1 \bft_2}) \right) \\
 & = T_{\bfn \bfn} \left( \bfA_{\bft_1, \bft_2, \bfn} (A_{\bfn \bfn}, A_{\bft_1 \bft_1}, A_{\bft_2 \bft_2}, A_{\bft_1 \bfn}, A_{\bft_2 \bfn}, A_{\bft_1 \bft_2}) \right) \leq 0  \quad \text{when } A_{\bfn \bfn} < A^*_{\bfn \bfn} .
\end{align*}
We thus recover the interpretation that \eqref{eq:compatibility} shows that $W$ has an uncoupled dependence of $A_{\bfn \bfn}$ and $A_{\bft_i \bfn}$.


\section{Example material: Mooney--Rivlin}\label{se:Mooney}

In this section we compute the effective energy for a Mooney--Rivlin material, in both 2D and 3D. 

\subsection{Neo-Hookean material in 2D}\label{subse:NeoH2D}

We consider the neo-Hookean energy in 2D
\begin{equation}\label{eq:WneoH}
 W (\bfF) = \frac{\mu}{2} \left( |\bfF|^2 -2 - 2 \log \det \bfF \right) + \frac{\lambda}{2} \left( \det \bfF -1 \right)^2 ,
\end{equation}
for $\mu , \lambda >0$.
Note that the minimum value of $W$ is zero and is attained at $SO(2)$.
Moreover, it is easy to check that the elasticity tensor $\C = D^2 W (\bfI)$ is given by
\begin{equation}\label{eq:isolinear}
 \C \bfvareps : \bfvareps = 2 \mu |\bfvareps|^2 + \lambda  (\tr \bfvareps)^2
\end{equation}
for symmetric $\bfvareps$, so $\lambda$ and $\mu$ are the Lam\'e parameters.

We compute $\Wd$ according to Definition \ref{de:Wd} with the obvious modifications for 2D.
As a consequence of Proposition \ref{pr:invariance}, it is enough to calculate $\Wd (\bfA, \bfe_2)$ for upper triangular matrices $\bfA$ with positive diagonal elements.
The orthonormal basis chosen will be, of course, $\{ \bfe_1, \bfe_2 \}$.
Since it is the canonical basis, we will use the usual triangular representation of a matrix, instead of the notation $A_{22} \bfe_2 \otimes \bfe_2 + A_{11} \bfe_1 \otimes \bfe_1  + A_{12} \bfe_1 \otimes \bfe_2$.

For $A_{11}, A_{22} >0$ and $A_{12} \in \R$, we have
\[
 W \begin{pmatrix} A_{11} & A_{12} \\
 0 & A_{22}
 \end{pmatrix} = \frac{\mu}{2} \left( A_{11}^2 + A_{12}^2 + A_{22}^2 - 2 - 2 \log (A_{11} A_{22}) \right) + \frac{\lambda}{2} (A_{11} A_{22} - 1)^2  .
\]
Given $A_{11} > 0$ the minimum of the above expression in $A_{22} >0$ and $A_{12} \in \R$ is easily seen to be attained at $A^*_{12} = 0$ and
\begin{equation}\label{eq:A22*}
 A_{22}^* (A_{11}) := \frac{ \lambda A_{11} + \sqrt{4 \mu^2 + 4 \mu \lambda  A_{11}^2 + \lambda^2 A_{11}^2}}{2 (\mu + \lambda A_{11}^2)} ,
\end{equation}
while given $A_{11} , A_{22}> 0$ the minimum in $A_{12} \in \R$ is easily seen to be attained at $A^{**}_{12} = 0$.
For future reference, note that
\begin{equation}\label{eq:A22*A11}
 A_{22}^* (A_{11}) > 1 \quad \text{if and only if} \quad A_{11} < 1 .
\end{equation}
The expression for $\Wd$ is, therefore,
\begin{equation}\label{eq:WdMR}
 \Wd \left( \begin{pmatrix}
 A_{11} & A_{12} \\
 0 & A_{22} 
  \end{pmatrix} , \bfe_2 \right) = \begin{cases}
 W \begin{pmatrix}
 A_{11} & 0 \\
 0 & A_{22}^* 
  \end{pmatrix} , & \text{if } A_{22} > A_{22}^* , \\
  W \begin{pmatrix}
 A_{11} & 0 \\
 0 & A_{22} 
  \end{pmatrix} , & \text{if } A_{22} \leq A_{22}^* .
 \end{cases}
\end{equation}
The values of $\Wd (\bfF, \bfn)$ for any $(\bfF, \bfn)$ can be calculated with the formula above and the relations
\begin{equation}\label{eq:Wdinvariance2D}
 \Wd (\bfQ \bfF, \bfn) = \Wd (\bfF, \bfn) = \Wd (\bfF \bfQ^T , \bfQ \bfn) , \qquad \bfQ \in SO(2) .
\end{equation}
(see Proposition \ref{pr:invariance} and note that $W$ is isotropic).
More explicitly, we consider, for any unit vector $\bfn$,
\begin{equation}\label{eq:Q2D}
 \bfQ = \begin{pmatrix}
 n_2 & -n_1 \\
 n_1 & n_2
 \end{pmatrix} ,
\end{equation}
which satisfies $\bfQ \in SO(2)$ and $\bfQ \bfn = \bfe_2$.
Then $\Wd (\bfF, \bfn) = \Wd (\bfF \bfQ^T, \bfe_2)$.
We apply QR decomposition to $\bfG := \bfF \bfQ^T$ and obtain that $\bfG = \bfR \bfA$ with $\bfR \in SO(2)$ and $\bfA$ upper triangular.
Elementary calculations show that
\[
 \bfA = \begin{pmatrix}
 A_{11} & A_{12} \\
 0 & A_{22}
 \end{pmatrix}
\]
with
\begin{equation}\label{eq:formulaspolar}
 A_{11} = \sqrt{G_{11}^2 + G_{21}^2} = \sqrt{(F_{11} n_2 - F_{12} n_1)^2 + (F_{21} n_2 - F_{22} n_1)^2} , \qquad A_{22} = \frac{\det \bfF}{A_{11}} ,
\end{equation}
while the expression for $A_{12}$ is not important.
Let $A_{22}^*$ be as in \eqref{eq:A22*}.
Then
\begin{equation}
  \boxed{
  \Wd (\bfF, \bfn) = \begin{cases}
 \frac{\mu}{2} \left( A_{11}^2 + (A_{22}^*)^2 -2 - 2 \log (A_{11} A_{22}^*) \right) + \frac{\lambda}{2} \left( A_{11} A_{22}^* -1 \right)^2 & \text{if } A_{22} > A_{22}^* , \\
  \frac{\mu}{2} \left( A_{11}^2 + A_{22}^2 -2 - 2 \log (A_{11} A_{22}) \right) + \frac{\lambda}{2} \left( A_{11} A_{22} -1 \right)^2 & \text{if } A_{22} \leq A_{22}^* .
 \end{cases}
 }
\end{equation}

Using formulas \eqref{eq:A22*} and \eqref{eq:WdMR}, we compute the effective energy $\Wd$ of the basic modes \ref{item:eP}--\ref{item:cP} described in Section \ref{subse:QR}, with the obvious modifications for 2D, in the special case $\bfn = \bfe_2$ and $\bft = \bfe_1$:
\begin{enumerate}[label=(\alph*)]
\item
$\bfF = A_{22} \bfe_2 \otimes \bfe_2 + \bfe_1 \otimes \bfe_1$ with $A_{22} >1 1$.
Then
\[
 \Wd \left( \begin{pmatrix} 1 & 0 \\
 0 & A_{22}
 \end{pmatrix} , \bfe_2 \right) = W \begin{pmatrix} 1 & 0 \\
 0 & 1
 \end{pmatrix} = 0 .
\]

\item 
$\bfF = \bfI + A_{12} \bfe_1 \otimes \bfe_2$ with $A_{12} \in \R$.
Then
\[
 \Wd \left( \begin{pmatrix} 1 & A_{12} \\
 0 & 1
 \end{pmatrix} , \bfe_2 \right) = W \begin{pmatrix} 1 & 0 \\
 0 & 1
 \end{pmatrix} = 0 .
\]
\item
$\bfF = \bfe_2 \otimes \bfe_2 + A_{11} \bfe_1 \otimes \bfe_1$ with $0 < A_{11} <1$.
 Then, thanks to \eqref{eq:A22*A11},
 \[
 \Wd \left( \begin{pmatrix}
 A_{11} & 0 \\
 0 & 1 
  \end{pmatrix}  , \bfe_2 \right) = 
  W \begin{pmatrix}
 A_{11} & 0 \\
 0 & 1 
  \end{pmatrix} .
\]

\item
$\bfF = \bfe_2 \otimes \bfe_2 + A_{11} \bfe_1 \otimes \bfe_1$ with $A_{11} > 1$.
 Then, thanks to \eqref{eq:A22*A11},
 \[
 \Wd \left( \begin{pmatrix}
 A_{11} & 0 \\
 0 & 1 
  \end{pmatrix}  , \bfe_2 \right) = 
 W \begin{pmatrix}
 A_{11} & 0 \\
 0 & A_{22}^* 
  \end{pmatrix}  .
\]
It is interesting to note that, contrary to mode \ref{item:ccT}, in this mode the energy of the cracked material is less than that of the intact material.
Under extension parallel to the crack, the damaged material undergoes -- in addition to an extension of magnitude $A_{11} >1$ parallel to the crack -- a compression of magnitude $A_{22}^* <1$ perpendicular to the crack due to the transverse shrinkage, because the material has a positive Poisson ratio.

\item
$\bfF = A_{22} \bfe_2 \otimes \bfe_2 + \bfe_1 \otimes \bfe_1$ with $0 < A_{22} < 1$.
Then
\[
 \Wd \left( \begin{pmatrix} 1 & 0 \\
 0 & A_{22}
 \end{pmatrix}  , \bfe_2 \right) =W \begin{pmatrix} 1 & 0 \\
 0 & A_{22}
 \end{pmatrix} .
\]

\end{enumerate}

These examples \ref{item:eP}--\ref{item:cP} show that the effective energy meet our expectations.
Now we calculate the effective energy of two other representative modes, for which an intermediate behavior is presented: $\Wd$ will carry some energy but less than $W$, i.e., $0 < \Wd < W$.

\begin{enumerate}[label=(\alph*),resume]

\item Isotropic compression with shear parallel to the crack:
$\bfF = \alpha \bfI + A_{12} \bfe_1 \otimes \bfe_2$ with $0 < \alpha < 1$ and $A_{12} \in \R$.
Then
\[
 \Wd \left( \begin{pmatrix} \alpha & A_{12} \\
 0 & \alpha
 \end{pmatrix}  , \bfe_2 \right) = W \begin{pmatrix} \alpha & 0 \\
 0 & \alpha
 \end{pmatrix} .
\]

\item Isotropic compression with shear perpendicular to the crack:
$\bfF = \alpha \bfI + A_{21} \bfe_2 \otimes \bfe_1$ with $0 < \alpha < 1$ and $A_{21} \in \R$.
By QR decomposition,
 \begin{align*}
  \Wd \left( \begin{pmatrix} \alpha & 0 \\
 A_{21} & \alpha
 \end{pmatrix}  , \bfe_2 \right) & = \Wd \left( \begin{pmatrix} \sqrt{\alpha^2 + A_{21}^2} & \frac{\alpha A_{21}}{\sqrt{\alpha^2 + A_{21}^2}} \\
 0 & \frac{\alpha^2}{\sqrt{\alpha^2 + A_{21}^2}}
 \end{pmatrix}  , \bfe_2 \right) \\
 & =
  W \begin{pmatrix}
 \sqrt{\alpha^2 + A_{21}^2} & 0 \\
 0 & \frac{\alpha^2}{\sqrt{\alpha^2 + A_{21}^2}} 
  \end{pmatrix} .
 \end{align*}

\end{enumerate}

\subsection{Mooney--Rivlin material in 3D}\label{subse:Mooney}

We consider the Mooney--Rivlin energy
\[
W (\bfF) = \frac{\mu_1}{2} \left( \left| \bfF \right|^2 - 2 \log \det \bfF - 3 \right) + \frac{\mu_2}{2} \left( \left| \cof \bfF \right|^2 - 4 \log \det \bfF - 3 \right) + \frac{\bar{\lambda}}{2} \left( \det \bfF - 1 \right)^2 ,
\]
where $\mu_1,\mu_2, \bar{\lambda} > 0$.
The minimum of $W$ is $0$ and is attained at $SO(3)$.
The Lam\'e parameters of this material are $\lambda = \bar{\lambda} + 2 \mu_2$ and $\mu = \mu_1 + \mu_2$, since the elasticity tensor $\C$ at the origin is given by \eqref{eq:isolinear}.
They are related with Young's modulus $E$ and Poisson's ratio $\nu$ as
\begin{equation*}
\mu =\frac{E}{2(1+\nu)} , \qquad
\lambda =\frac{E\nu}{(1+\nu)(1-2\nu)} , \qquad \nu = \frac{\lambda}{2 (\lambda + \mu)} .
\end{equation*}
In particular, this material has a positive Poisson ratio.

We do the calculations of $\Wd$ corresponding to Definition \ref{de:Wd}.
As in Section \ref{subse:NeoH2D}, as a consequence of Proposition \ref{pr:invariance} and the isotropy of $W$, it is enough to calculate $\Wd (\bfA, \bfe_3)$ for triangular matrices $\bfA$ with positive diagonal elements.
We have
\begin{align*}
  W & \begin{pmatrix}
 A_{11} & A_{12} & A_{13} \\
 0 & A_{22} & A_{23} \\
 0 & 0 & A_{33}
 \end{pmatrix} = \frac{\mu_1}{2} \left( A_{11}^2 + A_{22}^2 + A_{33}^2 + A_{12}^2 + A_{13}^2 + A_{23}^2 - 2 \log (A_{11} A_{22} A_{33}) - 3 \right) \\
 & \qquad \qquad \qquad + \frac{\mu_2}{2} \Big( 
 A_{11}^2 A_{22}^2 + A_{11}^2 A_{23}^2 + ( -A_{22} A_{13} + A_{23} A_{12} )^2 + A_{11}^2 A_{33}^2 + A_{33}^2 A_{12}^2 + A_{22}^2 A_{33}^2 \\
 & \qquad \qquad \qquad \qquad \quad - 4 \log (A_{11} A_{22} A_{33}) - 3 \Big) + \frac{\bar{\lambda}}{2} \left( A_{11} A_{22} A_{33} - 1 \right)^2 .
\end{align*}
Given $A_{11}, A_{22} >0$ and $A_{12} \in \R$, the minimum in $A_{33} >0$ and $A_{13}, A_{23} \in \R$ of the expression above is easily seen to be attained at $A^*_{13} = A^*_{23} = 0$ and $A_{33}^* (A_{11}, A_{22}, A_{12})$ given by
\begin{equation}\label{eq:A33*}
 A_{33}^* = \frac{\bar{\lambda} A_{11} A_{22} + \sqrt{(\bar{\lambda} A_{11} A_{22})^2 + 4 \left( \mu_1 + \mu_2 A_{11}^2 + \mu_2 A_{12}^2 + \mu_2 A_{22}^2 + \bar{\lambda} A_{11}^2 A_{22}^2 \right) \left( \mu_1 + 2 \mu_2 \right)}}{2 \left( \mu_1 + \mu_2 A_{11}^2 + \mu_2 A_{12}^2 + \mu_2 A_{22}^2 + \bar{\lambda} A_{11}^2 A_{22}^2 \right)} ,
\end{equation}
while given $A_{11}, A_{22}, A_{33} >0$ and $A_{12} \in \R$, its minimum in $A_{13}, A_{23} \in \R$ is easily seen to be attained at $A^{**}_{13} = A^{**}_{23} = 0$.
For future reference, we note that
\[
 A_{33}^* > 1 \quad \text{if and only if} \quad \mu_2 A_{11}^2 + \mu_2 A_{12}^2 + \mu_2 A_{22}^2 + \bar{\lambda} A_{11}^2 A_{22}^2 < \bar{\lambda} A_{11} A_{22} + 2 \mu_2 .
\]
In fact, the following particular case will be useful in the analysis of some examples: given $\alpha>0$,
\begin{equation}\label{eq:A33*alpha}
 A_{33}^* (\alpha, \alpha, 0) > 1 \quad \text{if and only if} \quad \alpha < 1 .
\end{equation}

The expression for $\Wd$ is, therefore,
\begin{equation}\label{eq:W1example}
 \Wd \left( \begin{pmatrix}
 A_{11} & A_{12} & A_{13} \\
 0 & A_{22} & A_{23} \\
 0 & 0 & A_{33}
 \end{pmatrix}, \bfe_3 \right) = \begin{cases}
 W \begin{pmatrix}
 A_{11} & A_{12} & 0 \\
 0 & A_{22} & 0 \\
 0 & 0 & A_{33}^*
 \end{pmatrix} , & \text{if } A_{33} > A_{33}^* , \\
 W \begin{pmatrix}
 A_{11} & A_{12} & 0 \\
 0 & A_{22} & 0 \\
 0 & 0 & A_{33}
 \end{pmatrix} , & \text{if } A_{33} \leq A_{33}^* .
 \end{cases}
\end{equation}

In order to calculate the values of $\Wd (\bfF, \bfe_3)$ for any $\bfF \in \R^{3 \times 3}_+$, we apply QR decomposition to $\bfF$: elementary but long calculations show that $\bfF = \bfR \bfA$ for some $\bfR \in SO(3)$ and
\[
 \bfA = \begin{pmatrix}
 A_{11} & A_{12} & A_{13} \\
 0 & A_{22} & A_{23} \\
 0 & 0 & A_{33}
 \end{pmatrix} ,
\]
with
\begin{align*}
 & A_{11} = \sqrt{F_{11}^2 + F_{21}^2 + F_{31}^2} , \\
 & A_{22} = \frac{\sqrt{(F_{11} F_{22} - F_{12} F_{21})^2 + ( F_{11} F_{32} - F_{12} F_{31})^2 + (F_{21} F_{32} - F_{22} F_{31})^2}}{A_{11}} , \\
 & A_{33} = \frac{\det \bfF}{A_{11} A_{22}} , \qquad
 A_{12} = \frac{F_{11} F_{12} + F_{21} F_{22} + F_{31} F_{32}}{A_{11}} .
\end{align*}
The expressions for $A_{13}, A_{23}$ are not relevant in this calculation.
With this, we have by Proposition \ref{pr:invariance} that $\Wd (\bfF, \bfe_3) = \Wd (\bfA, \bfe_3)$ and can apply formula \eqref{eq:W1example}.

Finally, the values of $\Wd (\bfF, \bfn)$ for any $(\bfF, \bfn)$ can be calculated with the formula above and the relation
\begin{equation}\label{eq:Wdinvariance3D}
 \Wd (\bfQ \bfF, \bfn) = \Wd (\bfF, \bfn) = \Wd (\bfF \bfQ^T , \bfQ \bfn) , \qquad \bfQ \in SO(3)
\end{equation}
(see Proposition \ref{pr:invariance} and note that $W$ is isotropic).
More explicitly, we consider, for any unit vector $\bfn$, two vectors $\bft_1 , \bft_2$ such that $\{ \bft_1 , \bft_2 , \bfn \}$ is an orthonormal basis.
Then we consider the rotation $\bfQ = \bfe_1 \otimes \bft_1 + \bfe_2 \otimes \bft_2 + \bfe_3 \otimes \bfn$, which in coordinates takes the form
\begin{equation}\label{eq:Q}
 \bfQ = \begin{pmatrix}
 (\bft_1)_1 & (\bft_1)_2 & (\bft_1)_3 \\
 (\bft_2)_1 & (\bft_2)_2 & (\bft_2)_3 \\
 n_1 & n_2 & n_3
 \end{pmatrix} .
\end{equation}
Clearly, $\bfQ \bfn = \bfe_3$.
Then $\Wd (\bfF, \bfn) = \Wd (\bfF \bfQ^T, \bfe_3)$ and we can apply the formulas above.
Specific choices of $\bft_1, \bft_2$ can be
\[
 \bft_1 = \frac{(n_1 n_3, n_2 n_3, -n_1^2 - n_2^2)}{\sqrt{n_1^2 + n_2^2}} , \qquad \bft_2 = \frac{(n_2, - n_1, 0)}{\sqrt{n_1^2 + n_2^2}} .
\]
  
Now we compute the effective energy $\Wd$ of the basic modes described in \ref{item:eP}--\ref{item:cP} descibed in Section \ref{subse:QR} for the special case that $\bfn = \bfe_3$, $\bft_1 = \bfe_1$ and $\bft_2 = \bfe_2$:
\begin{enumerate}[label=(\alph*)]

\item
$\bfF = \bfe_1 \otimes \bfe_1 + \bfe_2 \otimes \bfe_2 + A_{33} \bfe_3 \otimes \bfe_3$ with $A_{33} \geq 1$.
We have $A_{33}^* = 1$ and, hence,
\[
 \Wd \left( \begin{pmatrix}
 1 & 0 & 0 \\
 0 & 1 & 0 \\
 0 & 0 & A_{33}
 \end{pmatrix}  , \bfe_3 \right) = W \begin{pmatrix}
 1 & 0 & 0 \\
 0 & 1 & 0 \\
 0 & 0 & 1
 \end{pmatrix} = 0 .
\]

\item
$\bfF = \bfI + A_{13} \bfe_1 \otimes \bfe_3 + A_{23} \bfe_2 \otimes \bfe_3$ with $A_{13}, A_{23} \in \R$.
We have $A_{33}^* = 1$ and, hence,
\[
 \Wd \left( \begin{pmatrix}
 1 & 0 & A_{13} \\
 0 & 1 & A_{23} \\
 0 & 0 & 1
 \end{pmatrix} , \bfe_3 \right) = W \begin{pmatrix}
 1 & 0 & 0 \\
 0 & 1 & 0 \\
 0 & 0 & 1
 \end{pmatrix} = 0 .
\]

\item
$\bfF = \alpha \bfe_1 \otimes \bfe_1 + \alpha \bfe_2 \otimes \bfe_2 + \bfe_3 \otimes \bfe_3$ with $0 < \alpha < 1$.
Thanks to \eqref{eq:A33*alpha}, we have $A_{33}^* > 1$ and, hence,
\[
 \Wd \left( \begin{pmatrix}
 \alpha & 0 & 0 \\
 0 & \alpha & 0 \\
 0 & 0 & 1
 \end{pmatrix}  , \bfe_3 \right) = 
 W \begin{pmatrix}
 \alpha & 0 & 0 \\
 0 & \alpha & 0 \\
 0 & 0 & 1
 \end{pmatrix} .
\]

\item
$\bfF = \alpha \bfe_1 \otimes \bfe_1 + \alpha \bfe_2 \otimes \bfe_2 + \bfe_3 \otimes \bfe_3$ with $\alpha > 1$.
Thanks to \eqref{eq:A33*alpha}, we have $A_{33}^* < 1$ and, hence,
\[
 \Wd \left( \begin{pmatrix}
 \alpha & 0 & 0 \\
 0 & \alpha & 0 \\
 0 & 0 & 1
 \end{pmatrix}  , \bfe_3 \right) =
 W \begin{pmatrix}
 \alpha & 0 & 0 \\
 0 & \alpha & 0 \\
 0 & 0 & A_{33}^*
 \end{pmatrix} .
\]

The same comments of modes \ref{item:ccT}-\ref{item:ecT} in the neo-Hookean material of Subsection \ref{subse:NeoH2D}
apply here: as the material has a positive Poisson ratio, the extension of magnitude $\alpha>1$ parallel to the crack induces a compression of magnitude $A_{33}^* <1$ perpendicular to the crack.

\item
$\bfF = A_{33} \bfe_3 \otimes \bfe_3 + \bfe_1 \otimes \bfe_1 + \bfe_2 \otimes \bfe_2$ with $0 < A_{33} < 1$.
We have $A_{33}^*=1$ and, hence,
\[
 \Wd \left( \begin{pmatrix}
 1 & 0 & 0 \\
 0 & 1 & 0 \\
 0 & 0 & A_{33}
 \end{pmatrix}  , \bfe_3 \right) = W \begin{pmatrix}
 1 & 0 & 0 \\
 0 & 1 & 0 \\
 0 & 0 & A_{33}
 \end{pmatrix} .
\]
\end{enumerate}
As in Section \ref{subse:NeoH2D}, we compute the energy of the following two  modes for which $0 < \Wd < W$.

\begin{enumerate}[label=(\alph*),resume]
\item Isotropic compression with shear parallel to the crack:
$\bfF = \alpha \bfI + A_{13} \bfe_1 \otimes \bfe_3 + A_{23} \bfe_2 \otimes \bfe_3$ with $0 < \alpha < 1$ and $A_{13}, A_{23} \in \R$.
We have
\[
 A_{33}^* = \frac{\bar{\lambda} \alpha^2 + \sqrt{(\bar{\lambda} \alpha^2)^2 + 4 \left( \mu_1 + 2 \mu_2 \alpha^2 + \bar{\lambda} \alpha^4 \right) \left( \mu_1 + 2 \mu_2 \right)}}{2 \left( \mu_1 + 2 \mu_2 \alpha^2 + \bar{\lambda} \alpha^4 \right)} \geq \alpha ,
\]
so
\[
 \Wd \left( \begin{pmatrix}
 \alpha & 0 & A_{13} \\
 0 & \alpha & A_{23} \\
 0 & 0 & \alpha
 \end{pmatrix} , \bfe_3 \right) = W \begin{pmatrix}
 \alpha & 0 & 0 \\
 0 & \alpha & 0 \\
 0 & 0 & \alpha
 \end{pmatrix} .
\]

\item Isotropic compression with shear perpendicular to the crack:
$\bfF = \alpha \bfI + A_{31} \bfe_3 \otimes \bfe_1 + A_{32} \bfe_3 \otimes \bfe_2$ with $0 < \alpha \leq 1$ and $A_{31}, A_{32} \in \R$.
Some long but elementary calculations show that, when we define
\begin{equation}\label{eq:barA}
\begin{split}
 & \bar{A}_{11} =\sqrt{\alpha^2+ A_{31} ^2} , \qquad \bar{A}_{12} = \frac{A_{31}  A_{32}}{\bar{A}_{11}} , \\
 & \bar{A}_{22} = \frac{\sqrt{\alpha^4+\left( A_{31} ^2+A_{32}^2\right) \alpha^2}}{\bar{A}_{11}} , \qquad \bar{A}_{33} = \frac{\alpha^3}{\bar{A}_{11} \bar{A}_{22}},
\end{split}
\end{equation}
the expression of  $A^*_{33}$ is obtained via \eqref{eq:A33*} by substituting $A_{ij}$ with the $\bar{A}_{ij}$ in \eqref{eq:barA}, and the effective energy is
\[
 \Wd \left( \begin{pmatrix}
 \alpha & 0 & 0 \\
 0 & \alpha & 0 \\
 A_{31} & A_{32} & \alpha
 \end{pmatrix} , \bfe_3 \right) = \begin{cases}
 W \begin{pmatrix}
 \bar{A}_{11} & \bar{A}_{12} & 0 \\
 0 & \bar{A}_{22} & 0 \\
 0 & 0 & A_{33}^* 
 \end{pmatrix} & \text{if } \bar{A}_{33} \geq A_{33}^* , \\
  W \begin{pmatrix}
 \bar{A}_{11} & \bar{A}_{12} & 0 \\
 0 & \bar{A}_{22} & 0 \\
 0 & 0 & \bar{A}_{33} 
 \end{pmatrix} & \text{if } \bar{A}_{33} < A_{33}^* .
 \end{cases} 
\]
Unlike the analog of this mode in 2D (see Subsection \ref{subse:NeoH2D}), one can have both cases $\bar{A}_{33} \geq A_{33}^*$ and $\bar{A}_{33} < A_{33}^*$, depending on the parameters and the coefficients.
\end{enumerate}

\section{Example material: a $(p,q)$ energy}\label{se:pq}

In this section we consider a generalization of Mooney--Rivlin energies by changing the terms $\left| \bfF \right|^2$ and $\left| \cof \bfF \right|^2$ to  $\left| \bfF \right|^p$ and $\left| \cof \bfF \right|^q$, respectively, for arbitrary exponents $p, q \geq 1$.
The drawback is that explicit formulas are not available.

This generalization of the Mooney--Rivlin energy for exponents $(p,q)$ is useful in the modeling of soft materials \cite{HeMoXu15,HeMoXu16}.
For example, a neo-Hookean energy in 2D with a term $|\bfF|^p$ with $1<p<2$ allows for cavitation, while for $p \geq 2$, it does not.
Roughly speaking, the lower the exponent $p$, the softer the material.
In 3D, the exponent $q$ in $|\cof \bfF|^q$ also models the strength of the material.
In fact, $1 < p < 3$ and $1 < q < \frac{3}{2}$, the material can exhibit cavitation, while for $p \geq 3$ or $q \geq \frac{3}{2}$, it cannot \cite{MuQiYa94}.

\subsection{A $p$-energy in 2D}

For $\bar{\mu}, \bar{\lambda} >0$ and $p \geq 1$, we consider the following generalization of neo-Hookean energy:
\[
 W (\bfF) = \frac{\bar{\mu}}{p} \left( |\bfF|^p -2^{\frac{p}{2}} - 2^{\frac{p}{2}-1} p \log \det \bfF \right) + \frac{\bar{\lambda}}{2} \left( \det \bfF -1 \right)^2 .
\]
The minimum of $W$ is $0$ and is attained at $SO(2)$.
In order to compare the parameters $\bar{\mu}, \bar{\lambda}, p$ with those of linear elasticity, we define 
\[
 \lambda = \bar{\lambda} - 2^{\frac{p}{2}-2} \bar{\mu}, \qquad \mu = 2^{\frac{p}{2}-1} \bar{\mu} .
\]
and note that the elasticity tensor $\C$ at the origin is given by \eqref{eq:isolinear}.

As in Section \ref{subse:NeoH2D}, we do the calculations corresponding to Definition \ref{de:Wd}, with the obvious modifications in 2D.
Again, it is enough to calculate $\Wd (\bfA, \bfe_2)$ for upper-triangular $\bfA$.

For $A_{11}, A_{22} >0$ and $A_{12} \in \R$, the function $W$ satisfies
\[
 W \begin{pmatrix}
 A_{11} & A_{12} \\
 0 & A_{22}
 \end{pmatrix} = \frac{\bar{\mu}}{p} \left( \left( A_{11}^2 + A_{22}^2 + A_{12}^2 \right)^{\frac{p}{2}} -2^{\frac{p}{2}} - 2^{\frac{p}{2}-1} p \log (A_{11} A_{22}) \right) + \frac{\bar{\lambda}}{2} \left( A_{11} A_{22} -1 \right)^2 .
\]
It is easy to see that the infimum in $A_{12}$ is attained at $A_{12}^*=0$.
Moreover, as $p \geq 1$ the function $W \begin{pmatrix}
 A_{11} & A_{12} \\
 0 & A_{22}
 \end{pmatrix}$ is strictly convex in $A_{22}$, and tends to infinity when $A_{22} \to 0$ or $A_{22} \to \infty$.
Hence, there exists a unique minimizer $A_{22}^* = A_{22}^* (A_{11}) >0$.
The expression of $A_{22}^*$ cannot be given in closed form except for specific choices of $p$. 
Thus, the expression of the energy is
 \[
 \Wd \left( \begin{pmatrix}
 A_{11} & A_{12} \\
 0 & A_{22}
 \end{pmatrix} , \bfe_2 \right) = \begin{cases}
 W \begin{pmatrix}
 A_{11} & 0 \\
 0 & A_{22}^* 
  \end{pmatrix} , & \text{if } A_{22} > A_{22}^* , \\
  W \begin{pmatrix}
 A_{11} & 0 \\
 0 & A_{22} 
  \end{pmatrix} , & \text{if } A_{22} \leq A_{22}^* 
 \end{cases}
\]
and the value of $\Wd$ for any $(\bfF, \bfn)$ is reduced to this via \eqref{eq:Wdinvariance2D} (see Proposition \ref{pr:invariance} and note that $W$ is isotropic).
Alternatively, one can use Proposition \ref{pr:Wdalt}.

\subsection{A $(p,q)$-energy in 3D}

For $\mu_1, \mu_2, \bar{\lambda} >0$ and $p,q \geq 1$ we consider the following generalization of Mooney--Rivlin energy:
\begin{align*}
W (\bfF) = & \frac{\mu_1}{p} \left( \left| \bfF \right|^p - 3^{\frac{p}{2}} - 3^{\frac{p}{2} - 1} p \log \det \bfF \right) + \frac{\mu_2}{q} \left( \left| \cof \bfF \right|^q - 3^{\frac{q}{2}} - 2 \cdot 3^{\frac{q}{2} -1} q \log \det \bfF \right) \\
 & + \frac{\bar{\lambda}}{2} \left( \det \bfF - 1 \right)^2 .
\end{align*}
The minimum of $W$ is $0$ and is attained at $SO(3)$.
In order to compare the parameters $\mu_1, \mu_2, \bar{\lambda} ,p,q$ with those of linear elasticity, we define
\[
 \lambda = \bar{\lambda} - 3^{\frac{p}{2}-2} \mu_1 (p-2) + 2 \cdot 3^{\frac{q}{2}-2} \mu_2 (2 q - 1) , \qquad \mu = 3^{\frac{p}{2}-1} \mu_1 + 3^{\frac{q}{2}-1} \mu_2
\]
and note that the elasticity tensor $\C$ at the origin is given by \eqref{eq:isolinear}.

For $A_{11}, A_{22}, A_{33} >0$ and $A_{12}, A_{13}, A_{23} \in \R$, the function $W$ satisfies
\begin{align*}
 & W \begin{pmatrix}
 A_{11} & A_{12} & A_{13} \\
 0 & A_{22} & A_{23} \\
 0 & 0 & A_{33}
 \end{pmatrix} = \frac{\mu_1}{p} \left( \left( A_{11}^2 + A_{22}^2 + A_{33}^2 + A_{12}^2 + A_{13}^2 + A_{23}^2 \right)^{\frac{p}{2}} - 3^{\frac{p}{2}} - 3^{\frac{p}{2} - 1} p \log (A_{11} A_{22} A_{33}) \right) \\
 & + \frac{\mu_2}{q} \Big( \left( A_{11}^2 A_{22}^2 + A_{11}^2 A_{23}^2 + ( -A_{22} A_{13} + A_{23} A_{12} )^2 + A_{11}^2 A_{33}^2 + A_{33}^2 A_{12}^2 + A_{22}^2 A_{33}^2 \right)^{\frac{q}{2}} \\
 & \qquad \quad - 3^{\frac{q}{2}} - 2 \cdot 3^{\frac{q}{2} -1} q \log (A_{11} A_{22} A_{33}) \Big) \\
 & + \frac{\bar{\lambda}}{2} \left( A_{11} A_{22} A_{33} - 1 \right)^2 .
\end{align*}
It is easy to see that the infimum of this expression in $A_{13}, A_{23} \in \R$ is attained when $A_{13}^* = A_{23}^* = 0$.
Moreover, as $p , q \geq 1$ the function
\[
 W \begin{pmatrix}
 A_{11} & A_{12} & A_{13} \\
 0 & A_{22} & A_{23} \\
 0 & 0 & A_{33}
 \end{pmatrix}
\]
is strictly convex in $A_{33}$, and tends to infinity when $A_{33} \to 0$ or $A_{33} \to \infty$.
Hence, there exists a unique minimizer $A_{33}^* = A_{33}^* (A_{11}, A_{22}, A_{12}) >0$.
The expression of $A_{33}^*$ cannot be given in closed form except for specific choices of $p,q$. 
Thus, the expression of the energy is
 \[
 \Wd \left( \begin{pmatrix}
 A_{11} & A_{12} & A_{13} \\
 0 & A_{22} & A_{23} \\
 0 & 0 & A_{33}
 \end{pmatrix} , \bfe_3 \right) = \begin{cases}
 W \begin{pmatrix}
 A_{11} & A_{12} & 0 \\
 0 & A_{22} & 0 \\
 0 & 0 & A_{33}^*
 \end{pmatrix} , & \text{if } A_{33} > A_{33}^* , \\
  W \begin{pmatrix}
 A_{11} & A_{12} & 0 \\
 0 & A_{22} & 0 \\
 0 & 0 & A_{33}
 \end{pmatrix} , & \text{if } A_{33} \leq A_{33}^* 
 \end{cases}
\]
and the value of $\Wd$ at any $(\bfF, \bfn)$ is reduced to this via \eqref{eq:Wdinvariance3D}.
Alternatively, one can use Proposition \ref{pr:Wdalt}.

\section{Example material: General energy near the identity}\label{se:ExampleGeneral}

In this section we show how far we can go with the expression of $\Wd$ near the identity without an explicit formula for $W$.
The calculations here will be useful in Section \ref{se:linear} when we develop the linear theory.
We perform the calculations in detail for the 2D case and just write the final formulas for the 3D case.

The assumptions on $W$ are as follows: $W$ is a $C^2$ function, its set of global minimizers is $SO(3)$ (in the 2D case, $SO(2)$), it satisfies the coercivity assumption \eqref{eq:Winfty}, and the restriction of the elasticity tensor $\C = D^2 W (\bfI)$ to the set of symmetric matrices is positive definite.
Note that, necessarily, $D W (\bfI) = \bf0$.

\subsection{General 2D energy}\label{subse:General2D}

We calculate $\Wd$ according to Definition \ref{de:Wd}, with the obvious modification for 2D.
The basis chosen is the canonical one: $\{ \bfe_1, \bfe_2 \}$.

We use the usual notation
\[
 c_{ijkl} = \parderiv{^2 W}{F_{ij} \partial F_{kl}} (\bfI)
\]
for the components of the elasticity tensor $\C$.
The following inequalities are possibly well known to experts, but we have not found a proper reference.

\begin{lemma}\label{le:posdef}
Assume that the $\C$ is positive definite in the set of symmetric matrices.
Then
\[
 c_{1212} > 0 \quad \text{and} \quad \text c_{1212} c_{2222} - c_{1222}^2 > 0 .
\]
\end{lemma}

Define $\bff : \R^3 \to \R^2$ as
\[
 \bff (A_{22}, A_{11}, A_{12}) = \left( \parderiv{W}{F_{12}} \begin{pmatrix} A_{11} & A_{12} \\ 0 & A_{22} \end{pmatrix} , \parderiv{W}{F_{22}} \begin{pmatrix} A_{11} & A_{12} \\ 0 & A_{22} \end{pmatrix} \right) .
\]
Then
\[
 \bff (1, 1, 0) = (0, 0) , \qquad \parderiv{\bff}{A_{22}} (1, 1, 0) = (c_{1222}, c_{2222}) , \qquad \parderiv{\bff}{A_{12}} (1, 1, 0) = (c_{1212}, c_{1222}) .
\]
Thanks to Lemma \ref{le:posdef}, we can apply the implicit function theorem, and find that there are unique $C^1$ functions $A_{22}^* = A_{22}^* (A_{11})$ and $A_{12}^* = A_{12}^* (A_{11})$ defined for $A_{11} \simeq 1$ such that
\begin{equation}\label{eq:critical}
 \parderiv{W}{F_{12}} \begin{pmatrix} A_{11} & A_{12}^* (A_{11}) \\ 0 & A_{22}^* (A_{11}) \end{pmatrix} = \parderiv{W}{F_{22}} \begin{pmatrix} A_{11} & A_{12}^* (A_{11}) \\ 0 & A_{22}^* (A_{11}) \end{pmatrix} = 0 .
\end{equation}
Moreover,
\begin{equation}\label{eq:A12*}
\begin{aligned}
 & A_{12}^* (1) = 0, &&  A_{22}^* (1) = 1 , \\
 & (A_{12}^*)' (1) = \frac{c_{1122} c_{1222} - c_{1112} c_{2222}}{c_{1212} c_{2222} - c_{1222}^2} , &&
 (A_{22}^*)' (1) = \frac{c_{1112} c_{1222} - c_{1122} c_{1212}}{c_{1212} c_{2222} - c_{1222}^2} .
\end{aligned}
\end{equation}

In an analogous way, there is a unique $C^1$ function $A_{12}^{**} = A_{12}^{**} (A_{22}, A_{11})$ defined for $(A_{22}, A_{11}) \simeq (1, 1)$ such that
\[
 \parderiv{W}{F_{12}} \begin{pmatrix}
 A_{11} & A_{12}^{**} (A_{22}, A_{11}) \\
 0 & A_{22}
 \end{pmatrix} = 0 .
\]
Moreover,
\begin{equation}\label{eq:A12**}
 A_{12}^{**} (1,1) = 0 , \qquad \parderiv{A_{12}^{**}}{A_{11}} (1,1) = - \frac{c_{1112}}{c_{1212}} , \qquad
 \parderiv{A_{12}^{**}}{A_{22}} (1,1) = - \frac{c_{1222}}{c_{1212}}.
\end{equation}

A standard argument based on the coercivity \eqref{eq:Winfty}, the uniqueness of the functions $A_{22}^*, A_{12}^*, A_{12}^{**}$ and the fact that the minimum of $W$ is attained at $SO(2)$ shows that
\[
 \inf_{\substack{A_{22}' > 0 \\ A_{12}' \in \R}} W \begin{pmatrix}
 A_{11} & A_{12}' \\
 0 & A_{22}'
 \end{pmatrix} = W \begin{pmatrix}
 A_{11} & A_{12}^* (A_{11}) \\
 0 & A_{22}^* (A_{11})
 \end{pmatrix} ,
 \quad 
 \inf_{A_{12}' \in \R} W \begin{pmatrix}
 A_{11} & A_{12}' \\
 0 & A_{22}
 \end{pmatrix} =   W \begin{pmatrix}
 A_{11} & A_{12}^{**} (A_{22}, A_{11}) \\
 0 & A_{22}
 \end{pmatrix} .
\]
Indeed, the sketch of this argument is as follows.
We only do it for the first infimum, since the second is analogous.
The idea is to divide the range of the variables $A_{12}'$ and $A_{22}'$ into three regions, and ascertain whether the infimum is attained in those regions.
Region $R_1$ consists of values for which $|A_{12}'|$ is very big or $A_{22}' \simeq 0$ or $A_{22}'$ is very big.
Region $R_2$ consists of those values of $A_{12}'$ and $A_{22}'$ not in $R_1$ for which $A_{12}'$ is far from zero or $A_{22}'$ is far from $1$.
Finally, region $R_3$ consists of those values for which $A_{12}' \simeq 0$ and $A_{22}' \simeq 1$.
Thus, regions $R_1$, $R_2$ and $R_3$ cover the range of the variables  $A_{12}'$ and $A_{22}'$.
The continuity of $W$ and the coercivity condition \eqref{eq:Winfty} ensure that the infimum is actually a minimum and that it is not attained when $A_{22}' \to 0$ or $A_{22}' \to \infty$ or $|A_{12}'| \to \infty$.
Thus, the minimum is not attained in region $R_1$.
Since $\argmin W = SO(2)$ and $A_{11} \simeq 1$ the infimum has to be attained when
\[
\begin{pmatrix}
 A_{11} & A_{12}' \\
 0 & A_{22}'
 \end{pmatrix}
\]
is close to $SO(2)$.
This rules out the possibility that $A_{12}'$ and $A_{22}'$ are in region $R_2$.
Therefore, the infimum is attained in region $R_1$.
Now, the partial derivatives of $W$ with respect to $A_{12}$ and $A_{22}$ have to be zero at the minimum.
Finally, the argument \eqref{eq:critical} show that $A_{12}^*$ and $A_{22}^*$ are the only functions in region $R_1$ that make those partial derivatives be zero.

Thus, the expression of $\Wd$ becomes
\[
 \Wd \left( \begin{pmatrix}
 A_{11} & A_{12} \\
 0 & A_{22}
 \end{pmatrix} , \bfe_2 \right) = \begin{cases}
 W \begin{pmatrix}
 A_{11} & A_{12}^* (A_{11}) \\
 0 & A_{22}^* (A_{11})
 \end{pmatrix} , & \text{if } A_{22} \geq A_{22}^* (A_{11}) , \\
 W \begin{pmatrix}
 A_{11} & A_{12}^{**} (A_{22}, A_{11}) \\
 0 & A_{22}
 \end{pmatrix} , & \text{if } A_{22} < A_{22}^* (A_{11}) . 
 \end{cases}
\]
If $W$ is isotropic, this is enough to determine $\Wd$.
If not, one has to adapt the computation of this section to any orthonormal basis.
We omit the calculations, since they are totally analogous, and just write the final formulas.
The expression for $\Wd$ is
\[
 \Wd \left( \bfA (A_{\bfn \bfn} , A_{\bft \bft}, A_{\bft \bfn} ) , \bfn \right) = \begin{cases}
 W \left( \bfA ( A_{\bfn \bfn}^* (A_{\bft \bft}) , A_{\bft \bft}, A_{\bft \bfn}^* (A_{\bft \bft} ) \right) , & \text{if } A_{\bfn \bfn} \geq A_{\bfn \bfn}^* (A_{\bft \bft}) , \\
 W \left( \bfA ( A_{\bfn \bfn} , A_{\bft \bft}, A_{\bft \bfn}^{**} (A_{\bfn \bfn}, A_{\bft \bft}) \right) , & \text{if } A_{\bfn \bfn} < A_{\bfn \bfn}^* (A_{\bft \bft}) ,
 \end{cases}
\]
with the following definitions and properties:
\begin{itemize}
\item $\{ \bft, \bfn \}$ is an orthonormal basis.

\item $\bfA (A_{\bfn \bfn} , A_{\bft \bft}, A_{\bft \bfn} ) = A_{\bfn \bfn} \bfn \otimes \bfn + A_{\bft \bft} \bft \otimes \bft + A_{\bft \bfn} \bft \otimes \bfn$, which is the 2D analog of \eqref{eq:At1t2n}.

\item $\tilde{c}_{ijkl}$ are the coefficients of the elasticity tensor with respect to the basis $\{ \bft, \bfn\}$.
We will use a dual notation with the same meaning: the indices $i, j, k, l$ run in the set $\{ 1, 2 \}$ and also in the set $\{ \bft, \bfn \}$.

\item $\tilde{c}_{\bft \bfn \bft \bfn} > 0$ and $\tilde{c}_{\bft \bfn \bft \bfn} \tilde{c}_{\bfn \bfn \bfn \bfn} - \tilde{c}_{\bft \bfn \bfn \bfn}^2 > 0$.

\item There are unique $C^1$ functions $A_{\bfn \bfn}^* = A_{\bfn \bfn}^* (A_{\bft \bft})$ and $A_{\bft \bfn}^* = A_{\bft \bfn}^* (A_{\bft \bft})$ defined for $A_{\bft \bft} \simeq 1$ such that
\[
 \inf_{\substack{A_{\bfn \bfn}' > 0 \\ A_{\bft \bfn}' \in \R}} W \left( \bfA (A_{\bfn \bfn}' , A_{\bft \bft}, A_{\bft \bfn}' ) \right) = W \left( \bfA (A_{\bfn \bfn}^* (A_{\bft \bft})  , A_{\bft \bft}, A_{\bft \bfn}^* (A_{\bft \bft}) ) \right) .
\]
Moreover,
\begin{equation}\label{eq:Atn*}
\begin{aligned}
 & A_{\bft \bfn}^* (1) = 0, &&  A_{\bfn \bfn}^* (1) = 1 , \\
 & (A_{\bft \bfn}^*)' (1) = \frac{\tilde{c}_{\bft \bft \bfn \bfn} \tilde{c}_{\bft \bfn \bfn \bfn} - \tilde{c}_{\bft \bft \bft \bfn} \tilde{c}_{\bfn \bfn \bfn \bfn}}{\tilde{c}_{\bft \bfn \bft \bfn} \tilde{c}_{\bfn \bfn \bfn \bfn} - \tilde{c}_{\bft \bfn \bfn \bfn}^2} , &&
 (A_{\bfn \bfn}^*)' (1) = \frac{\tilde{c}_{\bft \bft \bft \bfn} \tilde{c}_{\bft \bfn \bfn \bfn} - \tilde{c}_{\bft \bft \bfn \bfn} \tilde{c}_{\bft \bfn \bft \bfn}}{\tilde{c}_{\bft \bfn \bft \bfn} \tilde{c}_{\bfn \bfn \bfn \bfn} - \tilde{c}_{\bft \bfn \bfn \bfn}^2} .
\end{aligned}
\end{equation}

\item There is a unique $C^1$ function $A_{\bft \bfn}^{**} = A_{\bft \bfn}^{**} (A_{\bfn \bfn}, A_{\bft \bft})$ defined for $(A_{\bfn \bfn}, A_{\bft \bft}) \simeq (1, 1)$ such that
\[
 \inf_{A_{\bft \bfn}' \in \R} W \left( \bfA (A_{\bfn \bfn} , A_{\bft \bft}, A_{\bft \bfn}' ) \right) = W \left( \bfA (A_{\bfn \bfn} , A_{\bft \bft}, A_{\bft \bfn}^{**} (A_{\bfn \bfn}, A_{\bft \bft}) ) \right) .
\]
Moreover,
\begin{equation}\label{eq:Atn**}
 A_{\bft \bfn}^{**} (1,1) = 0 , \qquad \parderiv{A_{\bft \bfn}^{**}}{A_{\bft \bft}} (1,1) = - \frac{\tilde{c}_{\bft \bft \bft \bfn}}{\tilde{c}_{\bft \bfn \bft \bfn}} , \qquad
 \parderiv{A_{\bft \bfn}^{**}}{A_{\bfn \bfn}} (1,1) = - \frac{\tilde{c}_{\bft \bfn \bfn \bfn}}{\tilde{c}_{\bft \bfn \bft \bfn}}.
\end{equation}

\end{itemize}

\subsection{General 3D energy}

In this section we just write the final formula for the effective energy, without a detailed description of the calculations, which, of course, follow the lines of the previous subsection.
There exists $C^1$ functions
\[
A_{13}^* = A_{13}^* (A_{11}, A_{22}, A_{12}) , \qquad A_{23}^* = A_{23}^* (A_{11}, A_{22}, A_{12}) , \qquad A_{33}^* = A_{33}^* (A_{11}, A_{22}, A_{12})
\]
defined for $(A_{11}, A_{22}, A_{12}) \simeq (1,1,0)$, and
\[
 A_{13}^{**} = A_{13}^{**} (A_{11}, A_{22}, A_{33}, A_{12}) , \qquad A_{23}^{**} = A_{23}^{**} (A_{11}, A_{22}, A_{33}, A_{12})
\]
defined for $(A_{11}, A_{22}, A_{33}, A_{12}) \simeq (1,1,1,0)$
such that
\[
 \inf_{\substack{A_{33}' > 0 \\ A_{13}', A_{23}' \in \R}} W \begin{pmatrix}
 A_{11} & A_{12} & A_{13}' \\
 0 & A_{22} & A_{23}'  \\
 0 & 0 & A_{33}' 
 \end{pmatrix} = W \begin{pmatrix}
 A_{11} & A_{12} & A_{13}^* \\
 0 & A_{22} & A_{23}^*  \\
 0 & 0 & A_{33}^* 
 \end{pmatrix} ,
\]
\[
 \inf_{A_{13}', A_{23}' \in \R} W \begin{pmatrix}
 A_{11} & A_{12} & A_{13}' \\
 0 & A_{22} & A_{23}'  \\
 0 & 0 & A_{33} 
 \end{pmatrix} =   W \begin{pmatrix}
 A_{11} & A_{12} & A_{13}^{**} \\
 0 & A_{22} & A_{23}^{**}  \\
 0 & 0 & A_{33} 
 \end{pmatrix}
\]
and
\[
 \Wd \left( \begin{pmatrix}
 A_{11} & A_{12} & A_{13} \\
 0 & A_{22} & A_{23} \\
 0 & 0 & A_{33}
 \end{pmatrix} , \bfe_3 \right) = \begin{cases}
 W \begin{pmatrix}
 A_{11} & A_{12} & A_{13}^* \\
 0 & A_{22} & A_{23}^*  \\
 0 & 0 & A_{33}^* 
 \end{pmatrix} , & \text{if } A_{33} \geq A_{33}^* , \\
 W \begin{pmatrix}
 A_{11} & A_{12} & A_{13}^{**} \\
 0 & A_{22} & A_{23}^{**}  \\
 0 & 0 & A_{33}
 \end{pmatrix} , & \text{if } A_{33} < A_{33}^* . 
 \end{cases}
\]
If $W$ is isotropic, this is enough to determine $\Wd$.
If not, one has to adapt the above calculations to any orthonormal basis $\{ \bft_1, \bft_2, \bfn\}$, as explained at the end of the previous subsection.

\section{Small Deformation Model}\label{se:linear}

The theory developed throughout this work is essentially nonlinear, since, for example, frame-indifference for $W$ is reflected in the equality $W (\bfR \bfF) = W (\bfF)$ for all rotations $\bfR$, while for $\Wd$ it takes the form \eqref{eq:Wdindifference}.
Consequently, the procedure to derive a linear theory is to start from a nonlinear energy $W$, compute its relaxation $\Wd$ and then linearize $\Wd$.
Changing the order of these operations (i.e., first linearize and then relax) would end up with a nonlinearly elastic energy, which eventually would need a further linearization.

The function obtained from $W$ by this process of relaxation and linearization will be denoted by $\Wdlin$, so that $\Wdlin (\bfvareps, \bfn)$ consists of the quadratic terms in $\bfvareps$ of $\Wd (\bfI + \bfvareps, \bfn)$.
It is well known that one can restrict to symmetric $\bfvareps$.

In this section we calculate $\Wdlin$ for a general material with a stored energy $W$ that satisfies the assumptions of Section \ref{se:ExampleGeneral}: $W$ is of class $C^2$, its set of global minimizers is $SO(3)$, it satisfies the coercivity assumption \eqref{eq:Winfty}, and the restriction of the elasticity tensor $\C = D^2 W (\bfI)$ to the set of symmetric matrices is positive definite.
We will see that the final formula for $\Wdlin$ only depends on $\C$.

For simplicity, the calculations in this section are detailed in dimension $2$, while in dimension $3$ only the final result are exposed.

\subsection{2D theory}\label{subse:linear2D}

In order to linearize $\Wd (\cdot, \bfe_2)$, we let $\bfF = \bfI + \bfvareps$ with $\bfvareps$ small, which plays the role of the displacement gradient.
We can assume that $\bfvareps$ is symmetric.

We linearize the formulas \eqref{eq:formulaspolar} for the QR decomposition of $\bfF$, and obtain that $\bfF = \bfR \bfA$ with $\bfR \in SO(2)$,
\[
 \bfA = \begin{pmatrix}
 A_{11} & A_{12} \\
 0 & A_{22}
 \end{pmatrix} 
\]
and
\begin{equation}\label{eq:A11A22}
 A_{11} = \sqrt{F_{11}^2 + F_{21}^2} = 1 + \vareps_{11} + o (|\bfvareps|), \qquad A_{22} = \frac{F_{11} F_{22} - F_{12} F_{21}}{\sqrt{F_{11}^2 + F_{21}^2}} = 1 + \vareps_{22} + o (|\bfvareps|) ,
\end{equation}
while the expression of $A_{12}$ is not important in this development.
We consider the functions $A_{22}^*, A_{12}^*, A_{12}^{**}$ of Section \ref{subse:General2D}.
Using  formulas \eqref{eq:A12*}, \eqref{eq:A12**} and \eqref{eq:A11A22}, we find that their linearization is as follows:
\begin{align*}
 A_{22}^* (A_{11}) & = 1 + \frac{c_{1112} c_{1222} - c_{1122} c_{1212}}{c_{1212} c_{2222} - c_{1222}^2} (A_{11} -1) + o (A_{11} -1) \\
 & = 1 + \frac{c_{1112} c_{1222} - c_{1122} c_{1212}}{c_{1212} c_{2222} - c_{1222}^2} \vareps_{11} + o (|\bfvareps|) , \\
 A_{12}^* (A_{11}) & = \frac{c_{1122} c_{1222} - c_{1112} c_{2222}}{c_{1212} c_{2222} - c_{1222}^2} (A_{11} -1) + o (A_{11} -1) = \frac{c_{1122} c_{1222} - c_{1112} c_{2222}}{c_{1212} c_{2222} - c_{1222}^2} \vareps_{11} + o (|\bfvareps|) , \\
 A_{12}^{**} (A_{22}, A_{11}) & = - \frac{c_{1112}}{c_{1212}} (A_{11} - 1) - \frac{c_{1222}}{c_{1212}} (A_{22} - 1) + o (|A_{11} - 1| + |A_{22} - 1|) \\
 & = - \frac{c_{1112}}{c_{1212}} \vareps_{11} - \frac{c_{1222}}{c_{1212}} \vareps_{22} + o (|\bfvareps|) .
\end{align*}
Now,
\begin{align*}
 \C \begin{pmatrix}
 A_{11} -1 & A_{12}^* \\
 0 & A_{22}^* -1
 \end{pmatrix} : & \begin{pmatrix}
 A_{11} -1 & A_{12}^* \\
 0 & A_{22}^* -1
 \end{pmatrix} = \C \begin{pmatrix}
 A_{11} -1 & A_{12}^*/2 \\
 A_{12}^*/2 & A_{22}^* -1
 \end{pmatrix} : \begin{pmatrix}
 A_{11} -1 & A_{12}^*/2 \\
 A_{12}^*/2 & A_{22}^* -1
 \end{pmatrix} \\
 & = c_{1111} (A_{11} -1)^2 + c_{1212} (A_{12}^*)^2 + c_{2222} (A_{22}^* -1)^2 \\
 & \quad + 2 \left( c_{1112} (A_{11} -1) A_{12}^* + c_{1122} (A_{11} -1) (A_{22}^* -1) + c_{1222}  A_{12}^* (A_{22}^* -1) \right) ,
\end{align*}
so
\begin{align*}
 & W \begin{pmatrix}
 A_{11} & A_{12}^* (A_{11}) \\
 0 & A_{22}^* (A_{11})
 \end{pmatrix} \\
 & = \frac{1}{2} \C \begin{pmatrix}
 A_{11} -1 & A_{12}^* \\
 0 & A_{22}^* -1
 \end{pmatrix} : \begin{pmatrix}
 A_{11} -1 & A_{12}^* \\
 0 & A_{22}^* -1
 \end{pmatrix} + o (|A_{11} - 1|^2 + |A_{12}^*|^2 + |A_{22}^* - 1|^2) \\
 & = \frac{1}{2} \vareps_{11}^2 \left[ \left( c_{1111} + c_{1212} \left( \frac{c_{1122} c_{1222} - c_{1112} c_{2222}}{c_{1212} c_{2222} - c_{1222}^2} \right)^2 + c_{2222} \left(  \frac{c_{1112} c_{1222} - c_{1122} c_{1212}}{c_{1212} c_{2222} - c_{1222}^2} \right)^2 \right) \right. \\
 & \left. \qquad \quad + 2 \left( c_{1112} \frac{c_{1122} c_{1222} - c_{1112} c_{2222}}{c_{1212} c_{2222} - c_{1222}^2} + c_{1122}  \frac{c_{1112} c_{1222} - c_{1122} c_{1212}}{c_{1212} c_{2222} - c_{1222}^2} \right. \right. \\
 & \left. \left. \qquad \qquad \quad + c_{1222}  \frac{c_{1122} c_{1222} - c_{1112} c_{2222}}{c_{1212} c_{2222} - c_{1222}^2}  \frac{c_{1112} c_{1222} - c_{1122} c_{1212}}{c_{1212} c_{2222} - c_{1222}^2} \right) \right] \\
 & \qquad \quad + o (|\bfvareps|^2) \\
 &= \frac{1}{2} \left( c_{1111} - \frac{c_{1122}^2 c_{1212} - 2 c_{1112} c_{1122} c_{1222} + c_{1112}^2 c_{2222}}{c_{1212} c_{2222} - c_{1222}^2} \right) \vareps_{11}^2 + o (|\bfvareps|^2) .
\end{align*}
Similarly,
\begin{align*}
 \C \begin{pmatrix}
 A_{11} -1 & A_{12}^{**} \\
 0 & A_{22} -1
 \end{pmatrix} : & \begin{pmatrix}
 A_{11} -1 & A_{12}^{**} \\
 0 & A_{22} -1
 \end{pmatrix} = \C \begin{pmatrix}
 A_{11} -1 & A_{12}^{**}/2 \\
 A_{12}^{**}/2 & A_{22} -1
 \end{pmatrix} : \begin{pmatrix}
 A_{11} -1 & A_{12}^{**}/2 \\
 A_{12}^{**}/2 & A_{22} -1
 \end{pmatrix} \\
 & = c_{1111} (A_{11} -1)^2 + c_{1212} (A_{12}^{**})^2 + c_{2222} (A_{22} -1)^2 \\
 & \quad + 2 \left( c_{1112} (A_{11} -1) A_{12}^{**} + c_{1122} (A_{11} -1) (A_{22} -1) + c_{1222}  A_{12}^{**} (A_{22} -1) \right) ,
\end{align*}
so
\begin{align*}
 W & \begin{pmatrix}
 A_{11} & A_{12}^{**} (A_{22}, A_{11}) \\
 0 & A_{22}
 \end{pmatrix} \\
 & = \frac{1}{2} \left( c_{1111}-\frac{c_{1112}^2}{c_{1212}} \right) \vareps_{11}^2 + \left( c_{1122} -\frac{c_{1112} c_{1222} }{c_{1212}}\right) \vareps_{11} \vareps_{22} + \frac{1}{2} \left( c_{2222} -\frac{c_{1222}^2}{c_{1212}} \right) \vareps_{22}^2 + o (|\bfvareps|^2) .
\end{align*}
The final formula for $\Wdlin (\cdot, \bfe_2)$ is
\begin{equation}\label{eq:Wline2}
\begin{split}
 & \Wdlin \left( \bfvareps , \bfe_2 \right) \\
 & = \begin{cases} \textstyle \frac{1}{2}\left( c_{1111} - \frac{c_{1112}^2 c_{2222}-2 c_{1112} c_{1122} c_{1222}+c_{1122}^2 c_{1212}}{c_{1212} c_{2222} - c_{1222}^2} \right) \vareps_{11}^2 \qquad \qquad \text{if } \vareps_{22} > \frac{c_{1112} c_{1222} - c_{1122} c_{1212}}{c_{1212} c_{2222} - c_{1222}^2} \vareps_{11} , \\
 \textstyle \frac{1}{2}\left( c_{1111}-\frac{c_{1112}^2}{c_{1212}} \right) \vareps_{11}^2 + \left( c_{1122} -\frac{c_{1112} c_{1222} }{c_{1212}}\right) \vareps_{11} \vareps_{22} + \frac{1}{2} \left( c_{2222} -\frac{c_{1222}^2}{c_{1212}} \right) \vareps_{22}^2 \qquad  \text{otherwise} . 
 \end{cases}
\end{split}
\end{equation}
Notice that $\Wdlin (\cdot, \bfe_2)$ is continuous.

The formula for $\Wdlin (\bfvareps, \bfn)$ can be deduced as follows.
If $W$ is isotropic, given a unit vector $\bfn$, we consider the rotation \eqref{eq:Q2D}.
The linearization of $\Wd$ consists of the quadratic terms in $\bfvareps$ of
\[
 \Wd (\bfI + \bfvareps, \bfn) = \Wd ((\bfI + \bfvareps) \bfQ^T, \bfe_2) = \Wd ( \bfQ (\bfI + \bfvareps) \bfQ^T, \bfe_2) = \Wd ( \bfI + \bfQ \bfvareps \bfQ^T, \bfe_2) ,
\]
where we have used Proposition \ref{pr:invariance} and the isotropy of $W$.
We define $\bfvareps^{\bfn} = \bfQ \bfvareps \bfQ^T$ and apply formula above for $\bfvareps^{\bfn}$ to find that
\begin{equation}\label{eq:Wdlin2D}
 \boxed{
\begin{aligned}
 & \Wdlin \left( \bfvareps , \bfn \right) \\
 & = \begin{cases} \textstyle \frac{1}{2}\left( c_{1111} - \frac{c_{1112}^2 c_{2222}-2 c_{1112} c_{1122} c_{1222}+c_{1122}^2 c_{1212}}{c_{1212} c_{2222} - c_{1222}^2} \right) (\vareps_{11}^{\bfn})^2 \qquad \qquad \text{if } \vareps^{\bfn}_{22} > \frac{c_{1112} c_{1222} - c_{1122} c_{1212}}{c_{1212} c_{2222} - c_{1222}^2} \vareps^{\bfn}_{11} , \\
 \textstyle \frac{1}{2}\left( c_{1111}-\frac{c_{1112}^2}{c_{1212}} \right) (\vareps_{11}^{\bfn})^2 + \left( c_{1122} -\frac{c_{1112} c_{1222} }{c_{1212}}\right) \vareps^{\bfn}_{11} \vareps^{\bfn}_{22} + \frac{1}{2} \left( c_{2222} -\frac{c_{1222}^2}{c_{1212}} \right) (\vareps_{22}^{\bfn})^2 \quad \text{otherwise} ,
 \end{cases}
\end{aligned}
}
\end{equation}
with 
\begin{equation}\label{eq:Hn}
 \vareps_{11}^{\bfn} = \vareps_{22} n_1^2 - 2 \vareps_{12} n_1 n_2 + \vareps_{11} n_2^2 , \qquad
 \vareps_{22}^{\bfn} = \vareps_{11} n_1^2 + 2 \vareps_{12} n_1 n_2 + \vareps_{22} n_2^2 .
\end{equation}

In fact, formula \eqref{eq:Wdlin2D} can be simplified because of the isotropy of $W$.
Indeed, taking into account the symmetries of $\C$, we find that it acts as
\begin{align*}
 \C \bfvareps : \bfvareps & = \sum_{ijkl} c_{ijkl} \vareps_{ij} \vareps_{kl} \\
 & = c_{1111} \vareps_{11}^2 + 4 c_{1112} \vareps_{11} \vareps_{12} + 2 c_{1122} \vareps_{11} \vareps_{22} + 4 c_{1212} \vareps_{12}^2 + 4 c_{1222} \vareps_{12} \vareps_{22} + c_{2222} \vareps_{22}^2 .
\end{align*}
for symmetric $\bfvareps$.
Comparing this expression with the familiar one in the isotropic case
\[
 \C \bfvareps : \bfvareps = 2 \mu |\bfvareps|^2 + \lambda (\tr \bfvareps )^2 = 2 \mu \left( \vareps_{11}^2 + 2 \vareps_{12}^2 + \vareps_{22}^2 \right) + \lambda \left(\vareps_{11}^2 + 2 \vareps_{11} \vareps_{22} + \vareps_{22}^2\right)
\]
we find that
\[
 c_{1111} = \lambda + 2 \mu , \qquad
 c_{1112} = 0 , \qquad
 c_{1122} = \lambda , \qquad
 c_{1212} = \mu , \qquad
 c_{1222} = 0 , \qquad
 c_{2222} = \lambda + 2 \mu .
\]
Therefore, the formula of $\Wdlin$ of \eqref{eq:Wdlin2D} reduces as follows:
\begin{equation}\label{eq:Wdlinison}
    \boxed{
     \Wdlin \left( \bfvareps , \bfn \right) = \begin{cases} \frac{1}{2}\left( \lambda + 2 \mu - \frac{\lambda^2}{\lambda + 2 \mu} \right) (\vareps_{11}^{\bfn})^2 & \text{if } \vareps^{\bfn}_{22} > - \frac{\lambda}{\lambda + 2 \mu} \vareps^{\bfn}_{11} , \\
     \frac{\lambda + 2 \mu}{2} (\vareps_{11}^{\bfn})^2 + \lambda \vareps^{\bfn}_{11} \vareps^{\bfn}_{22} + \frac{\lambda + 2 \mu}{2} (\vareps_{22}^{\bfn})^2 & \text{if } \vareps^{\bfn}_{22} \leq - \frac{\lambda}{\lambda + 2 \mu} \vareps^{\bfn}_{11} ,
     \end{cases}
     }
\end{equation}
with $\bfvareps^{\bfn}$ as in \eqref{eq:Hn}.

If $W$ is not isotropic, instead of formula \eqref{eq:Wdlin2D}, we have to adapt the calculations leading to \eqref{eq:Wline2} to any orthonormal basis $\{ \bft, \bfn \}$, as in the end of Subsection \ref{subse:General2D}.
The final formula is
\begin{equation}\label{eq:Wdlinaniso2D}
\begin{aligned}
 & \Wdlin \left( \bfvareps , \bfn \right) \\
 & = \begin{cases} \textstyle \frac{1}{2}\left( \tilde{c}_{\bft \bft \bft \bft} - \frac{\tilde{c}_{\bft \bft \bft \bfn}^2 \tilde{c}_{\bfn \bfn \bfn \bfn}-2 \tilde{c}_{\bft \bft \bft \bfn} \tilde{c}_{\bft \bft \bfn \bfn} \tilde{c}_{\bft \bfn \bfn \bfn}+\tilde{c}_{\bft \bft \bfn \bfn}^2 \tilde{c}_{\bft \bfn \bft \bfn}}{\tilde{c}_{\bft \bfn \bft \bfn} \tilde{c}_{\bfn \bfn \bfn \bfn} - \tilde{c}_{\bft \bfn \bfn \bfn}^2} \right) \tilde{\vareps}_{\bft \bft}^2 \qquad \qquad \text{if } \tilde{\vareps}_{\bfn \bfn} > \frac{\tilde{c}_{\bft \bft \bft \bfn} \tilde{c}_{\bft \bfn \bfn \bfn} - \tilde{c}_{\bft \bft \bfn \bfn} \tilde{c}_{\bft \bfn \bft \bfn}}{\tilde{c}_{\bft \bfn \bft \bfn} \tilde{c}_{\bfn \bfn \bfn \bfn} - \tilde{c}_{\bft \bfn \bfn \bfn}^2} \tilde{\vareps}_{\bft \bft} , \\
 \textstyle \frac{1}{2}\left( \tilde{c}_{\bft \bft \bft \bft}-\frac{\tilde{c}_{\bft \bft \bft \bfn}^2}{\tilde{c}_{\bft \bfn \bft \bfn}} \right) \tilde{\vareps}_{\bft \bft}^2 + \left( \tilde{c}_{\bft \bft \bfn \bfn} -\frac{\tilde{c}_{\bft \bft \bft \bfn} \tilde{c}_{\bft \bfn \bfn \bfn} }{\tilde{c}_{\bft \bfn \bft \bfn}}\right) \tilde{\vareps}_{\bft \bft} \tilde{\vareps}_{\bfn \bfn} + \frac{1}{2} \left( \tilde{c}_{\bfn \bfn \bfn \bfn} -\frac{\tilde{c}_{\bft \bfn \bfn \bfn}^2}{\tilde{c}_{\bft \bfn \bft \bfn}} \right) \tilde{\vareps}_{\bfn \bfn}^2 \quad \text{otherwise} ,
 \end{cases}
\end{aligned}
\end{equation}
where $\tilde{c}_{ijkl}$ are the coefficients of the elasticity tensor with respect to the basis $\{ \bft, \bfn\}$, and $\tilde{\vareps}_{ij}$ the components of the strain $\bfvareps$ with respect to the basis $\{ \bft, \bfn\}$.
They are related to the coefficients $c_{pqrs}$ with respect to the canonical basis via
\[
 \tilde{c}_{ijkl} = \sum_{p, q, r, s = 1}^2 \left( \tilde{\bfe}_i \cdot \bfe_p \right) \left( \tilde{\bfe}_j \cdot \bfe_q \right) \left( \tilde{\bfe}_k \cdot \bfe_r \right) \left( \tilde{\bfe}_l \cdot \bfe_s \right) c_{pqrs}
\]
\cite{Ting1996}, while
\[
 \tilde{\vareps}_{ij} = \sum_{p, q = 1}^2 \left( \tilde{\bfe}_i \cdot \bfe_p \right) \left( \tilde{\bfe}_j \cdot \bfe_q \right)  \vareps_{pq} .
\]
We have used a dual notation with the same meaning: the indices $i, j, k, l$ run in the set $\{ 1, 2 \}$ and also in the set $\{ \bft, \bfn \}$.
Moreover, the basis $\{ \bft, \bfn \}$ has been renamed to $\{ \tilde{\bfe}_1, \tilde{\bfe}_2 \}$.

As we can see from \eqref{eq:Wdlinaniso2D}, the linearization of $\Wd$ only depends on $W$ through its elasticity tensor.
Moreover, it satisfies the traction condition \eqref{eqn:classical-traction}.
Indeed, the $(\bft, \bfn)$ and the $(\bfn, \bfn)$ components of the Piola--Kirchhoff stress of the material given by $\Wdlin$ are, respectively,
\[
 \parderiv{\Wdlin}{\vareps_{\bft \bfn}} \left( \bfvareps , \bfn \right) = 0
\]
and
\begin{equation*}
 \parderiv{\Wdlin}{\vareps_{\bfn \bfn}} \left( \bfvareps , \bfn \right) = \max \left\{ \textstyle \left( \tilde{c}_{\bft \bft \bfn \bfn} -\frac{\tilde{c}_{\bft \bft \bft \bfn} \tilde{c}_{\bft \bfn \bfn \bfn} }{\tilde{c}_{\bft \bfn \bft \bfn}}\right) \tilde{\vareps}_{\bft \bft} +  \left( \tilde{c}_{\bfn \bfn \bfn \bfn} -\frac{\tilde{c}_{\bft \bfn \bfn \bfn}^2}{\tilde{c}_{\bft \bfn \bft \bfn}} \right) \tilde{\vareps}_{\bfn \bfn} , 0 \right\} = \max \left\{ T_{\bfn \bfn} , 0 \right\} ,
\end{equation*}
where $T_{\bfn \bfn}$ is the $(\bfn, \bfn)$ component of the Piola--Kirchhoff stress corresponding to an intact material with energy
\[
 \left( \tilde{c}_{\bft \bft \bfn \bfn} -\frac{\tilde{c}_{\bft \bft \bft \bfn} \tilde{c}_{\bft \bfn \bfn \bfn} }{\tilde{c}_{\bft \bfn \bft \bfn}}\right) \tilde{\vareps}_{\bft \bft} \tilde{\vareps}_{\bfn \bfn} + \frac{1}{2} \left( \tilde{c}_{\bfn \bfn \bfn \bfn} -\frac{\tilde{c}_{\bft \bfn \bfn \bfn}^2}{\tilde{c}_{\bft \bfn \bft \bfn}} \right) \tilde{\vareps}_{\bfn \bfn}^2 + h (\tilde{\vareps}_{\bft \bft}, \tilde{\vareps}_{\bft \bfn})
\]
for an arbitrary function $h$.
For the sake of symmetry, we choose the $h$ so that the energy is
\[
 \frac{1}{2} \left( \tilde{c}_{\bft \bft \bft \bft} -\frac{\tilde{c}_{\bft \bft \bft \bfn}^2}{\tilde{c}_{\bft \bfn \bft \bfn}} \right) \tilde{\vareps}_{\bft \bft}^2 + \left( \tilde{c}_{\bft \bft \bfn \bfn} -\frac{\tilde{c}_{\bft \bft \bft \bfn} \tilde{c}_{\bft \bfn \bfn \bfn} }{\tilde{c}_{\bft \bfn \bft \bfn}}\right) \tilde{\vareps}_{\bft \bft} \tilde{\vareps}_{\bfn \bfn} + \frac{1}{2} \left( \tilde{c}_{\bfn \bfn \bfn \bfn} -\frac{\tilde{c}_{\bft \bfn \bfn \bfn}^2}{\tilde{c}_{\bft \bfn \bft \bfn}} \right) \tilde{\vareps}_{\bfn \bfn}^2 ,
\]
but this energy depends on $\bfn$, so in general it does not correspond to a single intact material.

\subsection{3D theory}

An analogous calculation can be done in 3D.
We do not repeat the argument but write down the final formulas.
We make the following abbrevations:
\[
 \bar{\bfvareps} = (\vareps_{11}, 2 \vareps_{12}, \vareps_{22}) , \qquad \hat{\bfvareps} = (\vareps_{11}, 2 \vareps_{12}, \vareps_{22}, \vareps_{33}) ,
\]
$\nabla A_{13}^*, \nabla A_{23}^*, \nabla A_{33}^*$ are the row vectors in $\R^3$ solving the system
\[
 \begin{pmatrix}
 c_{1313} & c_{1323} & c_{1333} \\
 c_{1323} & c_{2323} & c_{2333} \\
 c_{1333} & c_{2333} & c_{3333}
 \end{pmatrix}
 \begin{pmatrix}
 \nabla A_{13}^* \\
 \nabla A_{23}^* \\
 \nabla A_{33}^* 
 \end{pmatrix} = -
 \begin{pmatrix}
 c_{1113} & c_{1213} & c_{1322} \\
 c_{1123} & c_{1223} & c_{2223} \\
 c_{1133} & c_{1233} & c_{2233}
 \end{pmatrix} ,
\]
while
$\nabla A_{13}^{**}, \nabla A_{23}^{**}$ are the row vectors in $\R^4$ solving the system
\[
\begin{pmatrix}
 c_{1313} & c_{1323} \\
 c_{1323} & c_{2323} \\
 \end{pmatrix} \begin{pmatrix}
 \nabla A_{13}^{**} \\
 \nabla A_{23}^{**}
 \end{pmatrix} = -
 \begin{pmatrix}
 c_{1113} & c_{1213} & c_{1322} & c_{1333} \\
 c_{1123} & c_{1223} & c_{2223} & c_{2333}
 \end{pmatrix} .
\]
Then, $\Wdlin (\bfvareps, \bfe_3)$ equals
\begin{equation}\label{eq:Wdline3case1}
\begin{split}
 & \frac{1}{2} \Big[ c_{1111} \vareps_{11}^2 + 4 c_{1112} \vareps_{11} \vareps_{12} + 2 c_{1113} \vareps_{11} \nabla A_{13}^* \cdot \bar{\bfvareps} + 2 c_{1122} \vareps_{11} \vareps_{22} + 2 c_{1123} \vareps_{11} \nabla A_{23}^* \cdot \bar{\bfvareps} \\
 & + 2 c_{1133} \vareps_{11} \nabla A_{33}^* \cdot \bar{\bfvareps} + 4 c_{1212} \vareps_{12}^2 + 4 c_{1213} \vareps_{12} \nabla A_{13}^* \cdot \bar{\bfvareps} + 4 c_{1222} \vareps_{12} \vareps_{22} + 4 c_{1223} \vareps_{12} \nabla A_{23}^* \cdot \bar{\bfvareps} \\
 & + 4 c_{1233} \vareps_{12} \nabla A_{33}^* \cdot \bar{\bfvareps} + c_{1313} \left( \nabla A_{13}^* \cdot \bar{\bfvareps} \right)^2 + 2 c_{1322} \nabla A_{13}^* \cdot \bar{\bfvareps} \vareps_{22} + 2 c_{1323} \nabla A_{13}^* \cdot \bar{\bfvareps} \nabla A_{23}^* \cdot \bar{\bfvareps} \\
 & + 2 c_{1333} \nabla A_{13}^* \cdot \bar{\bfvareps} \nabla A_{33}^* \cdot \bar{\bfvareps} + c_{2222} \vareps_{22}^2 + 2 c_{2223} \vareps_{22} \nabla A_{23}^* \cdot \bar{\bfvareps} + 2 c_{2233} \vareps_{22} \nabla A_{33}^* \cdot \bar{\bfvareps} + c_{2323} \left( \nabla A_{23}^* \cdot \bar{\bfvareps} \right)^2 \\
 & + 2 c_{2333} \nabla A_{23}^* \cdot \bar{\bfvareps} \nabla A_{33}^* \cdot \bar{\bfvareps} + c_{3333} \left( \nabla A_{33}^* \cdot \bar{\bfvareps} \right)^2 \Big] ,
\end{split}
\end{equation}
if $\vareps_{33} \geq \nabla A_{33}^* \cdot \bar{\bfvareps}$, while it equals
\begin{equation}\label{eq:Wdline3case2}
\begin{split}
 & \frac{1}{2} \Big[ c_{1111} \vareps_{11}^2 + 4 c_{1112} \vareps_{11} \vareps_{12} + 2 c_{1113} \vareps_{11} \nabla A_{13}^{**} \cdot \hat{\bfvareps} + 2 c_{1122} \vareps_{11} \vareps_{22} + 2 c_{1123} \vareps_{11} \nabla A_{23}^{**} \cdot \hat{\bfvareps} \\
 & + 2 c_{1133} \vareps_{11} \vareps_{33} + 4 c_{1212} \vareps_{12}^2 + 4 c_{1213} \vareps_{12} \nabla A_{13}^{**} \cdot \hat{\bfvareps} + 4 c_{1222} \vareps_{12} \vareps_{22} + 4 c_{1223} \vareps_{12} \nabla A_{23}^{**} \cdot \hat{\bfvareps} \\
 & + 4 c_{1233} \vareps_{12} \vareps_{33} + c_{1313} \left( \nabla A_{13}^{**} \cdot \hat{\bfvareps} \right)^2 + 2 c_{1322} \nabla A_{13}^{**} \cdot \hat{\bfvareps} \vareps_{22} + 2 c_{1323} \nabla A_{13}^{**} \cdot \hat{\bfvareps} \nabla A_{23}^{**} \cdot \hat{\bfvareps} \\
 & + 2 c_{1333} \nabla A_{13}^{**} \cdot \hat{\bfvareps} \vareps_{33} + c_{2222} \vareps_{22}^2 + 2 c_{2223} \vareps_{22} \nabla A_{23}^{**} \cdot \hat{\bfvareps} + 2 c_{2233} \vareps_{22} \vareps_{33} + c_{2323} \left( \nabla A_{23}^{**} \cdot \hat{\bfvareps} \right)^2 \\
 & + 2 c_{2333} \nabla A_{23}^{**} \cdot \hat{\bfvareps} \vareps_{33} + c_{3333} \vareps_{33}^2 \Big] ,
\end{split}
\end{equation}
if $\vareps_{33} < \nabla A_{33}^* \cdot \bar{\bfvareps}$.

In the presence of isotropy, the final formula is
\begin{equation}
\boxed{
\begin{aligned}
     & \Wdlin \left( \bfvareps , \bfe_3 \right) = \\
     & \begin{cases} \textstyle \frac{1}{2} \left( \lambda + 2 \mu + \frac{3\lambda^2}{\lambda + 2 \mu} \right) (\vareps_{11}^2 + \vareps_{22}^2 ) + \left( \lambda + \frac{3\lambda^2}{\lambda + 2 \mu} \right) \vareps_{11} \vareps_{22} + 2 \mu \vareps_{12}^2 & \text{if } \vareps_{33} \geq \frac{\lambda}{\lambda + 2 \mu} (\vareps_{11} + \vareps_{22}) , \\
     \textstyle \frac{\lambda + 2 \mu}{2} (\vareps_{11}^2 + \vareps_{22}^2 + \vareps_{33}^2) + \lambda ( \vareps_{11} \vareps_{22} + \vareps_{11} \vareps_{33} + \vareps_{22} \vareps_{33} ) + 2 \mu  \vareps_{12}^2 & \text{if } \vareps_{33} < \frac{\lambda}{\lambda + 2 \mu} (\vareps_{11} + \vareps_{22}) ,
     \end{cases}
     \end{aligned}
 }
\end{equation}
and the formula for a general unit vector $\bfn$ is calculated as follows: we consider any $\bfQ \in SO(3)$ such that $\bfQ \bfn = \bfe_3$, as in \eqref{eq:Q}, and use that
\[
 \Wdlin (\bfvareps, \bfn) = \Wdlin (\bfQ \bfvareps \bfQ^T , \bfe_3) .
\]

If $W$ is not isotropic, then $\Wdlin (\bfvareps, \bfn)$ is calculated as in \eqref{eq:Wdline3case1}--\eqref{eq:Wdline3case2}, but replacing $c_{ijkl}$ with $\tilde{c}_{ijkl}$ and $\vareps_{ij}$ with $\tilde{\vareps}_{ij}$, similarly to Subsection \ref{subse:linear2D}.


\section{Numerical Implementation}\label{se:numerical}

We use the finite element method to compute numerical solutions, implemented in the open source FEniCS library, which is now used widely for problems in mechanics, e.g., \cite{logg2012automated,barchiesi2021computation}; our implementation adapts the code from \cite{natarajan2019phase}.
We use $\bfy(\bfx)$ and $\bfd(\bfx)$ as the primary unknown functions, and approximate them using standard piecewise linear elements.
We consider two classes of problems.

First, we compute the energy-minimizing deformation $\bfy$, keeping $\bfd$ fixed at a prescribed configuration that corresponds roughly to a crack; we refer to this type of problem as a ``frozen crack''.
For these problems, the energy is minimized under load using the FEniCS Dolfin Adjoint library \cite{mitusch2019dolfin}. 

Second, we compute the fully-coupled problem of crack growth, wherein both $\bfy$ and $\bfd$ evolve to minimize the energy under a time-dependent external loading.
To compute the minimization at each load step, we find that alternate minimization over $\bfy$ and $\bfd$ works well, following \cite{natarajan2019phase}.
Specifically, at each load step, we perform first a minimization over $\bfy$ with $\bfd$ fixed, and then minimize over $\bfd$ with $\bfy$ fixed, and alternate this until we converge.
That is, given the iterates $\bfy^{i}$ and $\bfd^{i}$, we obtain the next iterate using:
\begin{equation}\label{eq:staggered}
\begin{split}
     & \bfy^{i+1} = \argmin_{\bfy} E[\bfy, \bfd^{i}],\\
     &\bfd^{i+1} = \argmin_{\bfd} E[\bfy^{i+1}, \bfd]
\end{split}
\end{equation}
and we repeat until $\max\{\lVert \bfy^{i+1} - \bfy^{i} \rVert ,  \lVert \bfd^{i+1} - \bfd^{i} \rVert\} < 10^{-3}$. 
Once we reach a converged solution, we update the time-dependent load.

An important aspect of crack growth is that the minimization over $\bfd$ must respect the physics that cracks do not heal even when the load is reversed.
Following standard approaches, e.g. \cite{Bourdin07}, we implement this irreversibility as follows.
We first choose a critical value $d_c$, taken to be $0.95$ in this paper.
If the converged vector field $\bfd_0$ in a loading step satisfies $|\bfd_0| \geq d_c$, we impose the pointwise constraint $\bfd = \frac{\bfd_0}{|\bfd_0|}$ for all future times, and minimize over $\bfd$ that respect this constraint.

In addition to preventing the magnitude of $\bfd$ from decreasing once it satisfies $|\bfd| \geq d_c$, the constraint also prevents the crack normal from changing orientation.
This is important in the setting considered in this paper.
Specifically, if we did not prevent the orientation of the crack from changing, we could concieve of a sequence of loading steps that would lead to the unphysical result that a crack normal changes orientation even if the crack -- defined by the magnitude of $\bfd$ -- is fixed.

We use $\epsilon = 0.015$ in a nondimensionalized length scale for the phase-field regularization parameter introduced in \eqref{eqn:full-model}.
Given this value of $\epsilon$, we use a mesh that is sufficiently refined as to resolve the interfaces.

\section{Numerical Calculations}\label{se:calculations}

We solve the model numerically for three settings: 
first, we look at the elastic response of frozen cracks in which $\bfd$ is held fixed to a given configuration, as a means of testing the efficacy of the model; 
second, we look at a configuration with cyclic shear loading that leads to crack closure as well as crack branching;
and, third, we look at fracture in a cavity under cyclic loading that leads to crack closure and the growth of multiple cracks, motivated by recent experiments on similar settings.

All of the settings considered above use the large-deformation neo-Hookean model in 2D (Section \ref{subse:NeoH2D}).
We choose material parameters corresponding roughly to PMMA, which is a common material for experimental research \cite{sane2001interconversion}.
We use $\lambda = \SI{2.576}{\giga\pascal}$, $\mu = \SI{1.104}{\giga\pascal}$, and $G_c = \SI{285}{\newton\per\meter}$ for the frozen cracks.
For growing cracks, we set $\lambda = 0$ to avoid Poisson's ratio effects.
All plots of energy density and stress are normalized by $\mu$.

\subsection{Mechanical Response of a frozen crack}

Figure \ref{fig:frozen-crack} shows the specimen the chosen configuration for $\bfd$.
We use a circular crack to ensure that the model works even in configurations that are far from a classical sharp crack.
We compute the elastic response of the specimen for the fundamental deformation modes from Figure \ref{fig:modes}. 
The deformation modes are applied by imposing the expressions in Section \ref{subse:QR} as affine boundary conditions on the specimen.
For instance, considering mode (a) and using that $\bfn = \bfe_2$ in this specimen, we define $\bfF^0 = \bfe_1\otimes\bfe_1 + F^0_{22} \bfe_2\otimes\bfe_2$ with $F^0_{22} > 1$; and then minimize the elastic energy subject to $\bfy = \bfF^0 \bfx$ on the boundary of the specimen.

\begin{figure*}[ht!]
	\includegraphics[width=0.67\textwidth]{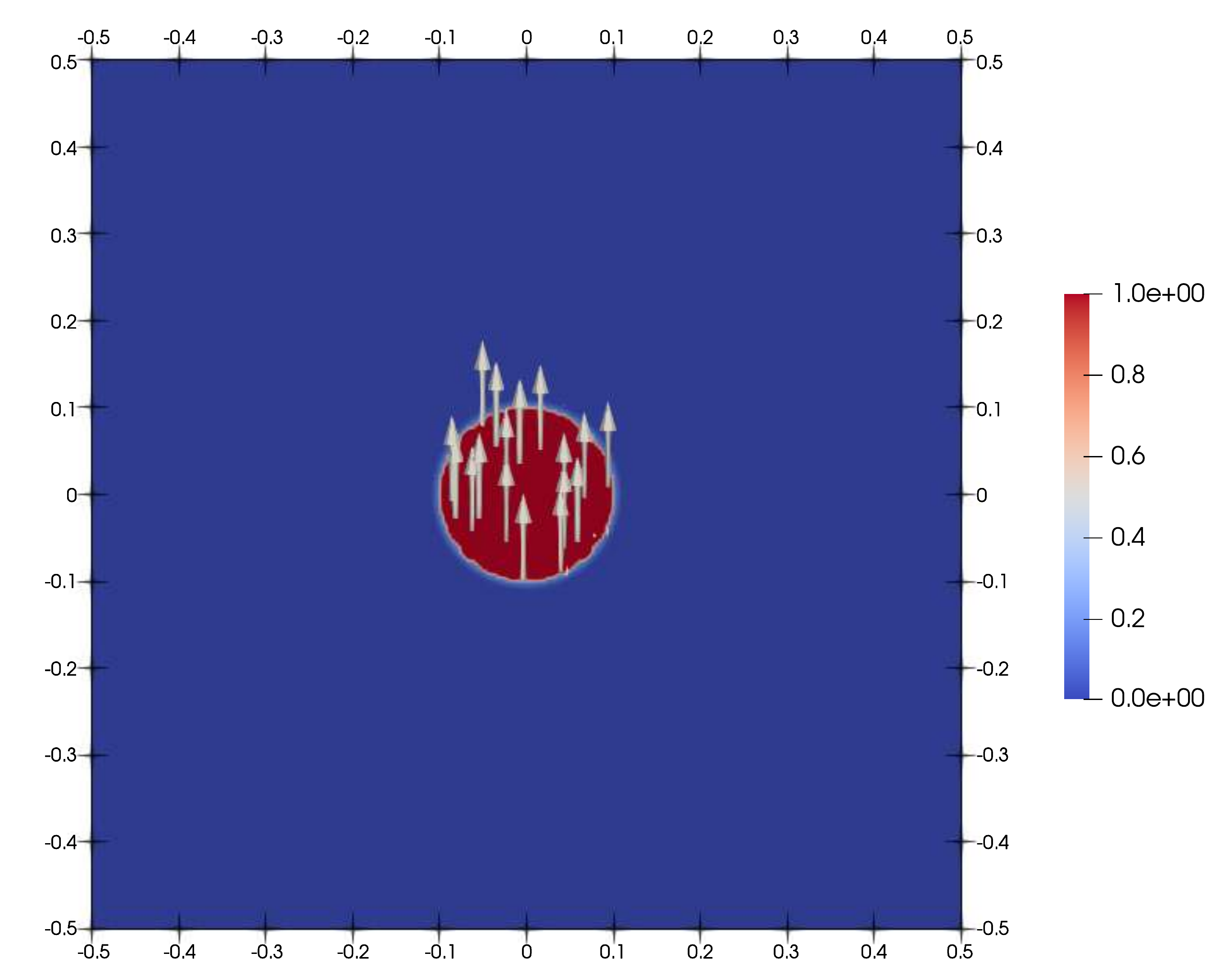}
	\caption{The chosen configuration for the $\bfd$ field for the frozen crack calculations. We use a circular configuration to test the effective energy in configurations that are far from an ideal sharp crack.}
	\label{fig:frozen-crack}
\end{figure*}

Figures \ref{fig:energy and traction at modes of zero energy} and \ref{fig:energy and traction at modes of positive energy} show the elastic energy density, deformation, and the traction on the plane with normal $\bfe_2$ for the fundamental deformation modes. 
Specifically, Figure \ref{fig:energy and traction at modes of zero energy} shows modes (a) and (b) for which our idealized crack should have no elastic response.
We see that the elastic energy in the crack is much smaller than the intact material; the deformations in the crack are much larger than in the intact material; and there are no shear or tensile normal tractions in the crack.

Figure \ref{fig:energy and traction at modes of positive energy} plots the same quantities for modes (c), (d), (e).
For modes (c) and (e), we notice that the elastic energy density, deformation, and traction are uniform to the precision of the numerical approximations; the specimen behaves as expected for a crack that is closing, and the traction across the crack faces are indeed compressive.
For mode (d), on the other hand, we find that the specimen has the clear signature of the crack in the elastic energy density, the deformation, and the traction.
This can be understood as arising due to an interplay between the imposed affine boundary conditions and the fact that the crack faces pull apart due to lateral shrinkage driven by the load -- i.e., a Poisson's ratio effect -- because the crack faces cannot support tension.
To show this, we use affine boundary conditions that are set up to give the normal stress in the vertical direction to be $0$, and find that the deformation and elastic energy density are indeed uniform in this setting (Fig.\ \ref{fig:case d with lambda =0.85}); i.e., when allowed to relax the stress, the specimen behaves like an idealized sharp crack.

\begin{figure}[h!]
\centering
	\subfloat[Elastic energy, overlaid on a mesh showing deformation, for deformation mode (a)]
	{\includegraphics[width=0.44\textwidth]{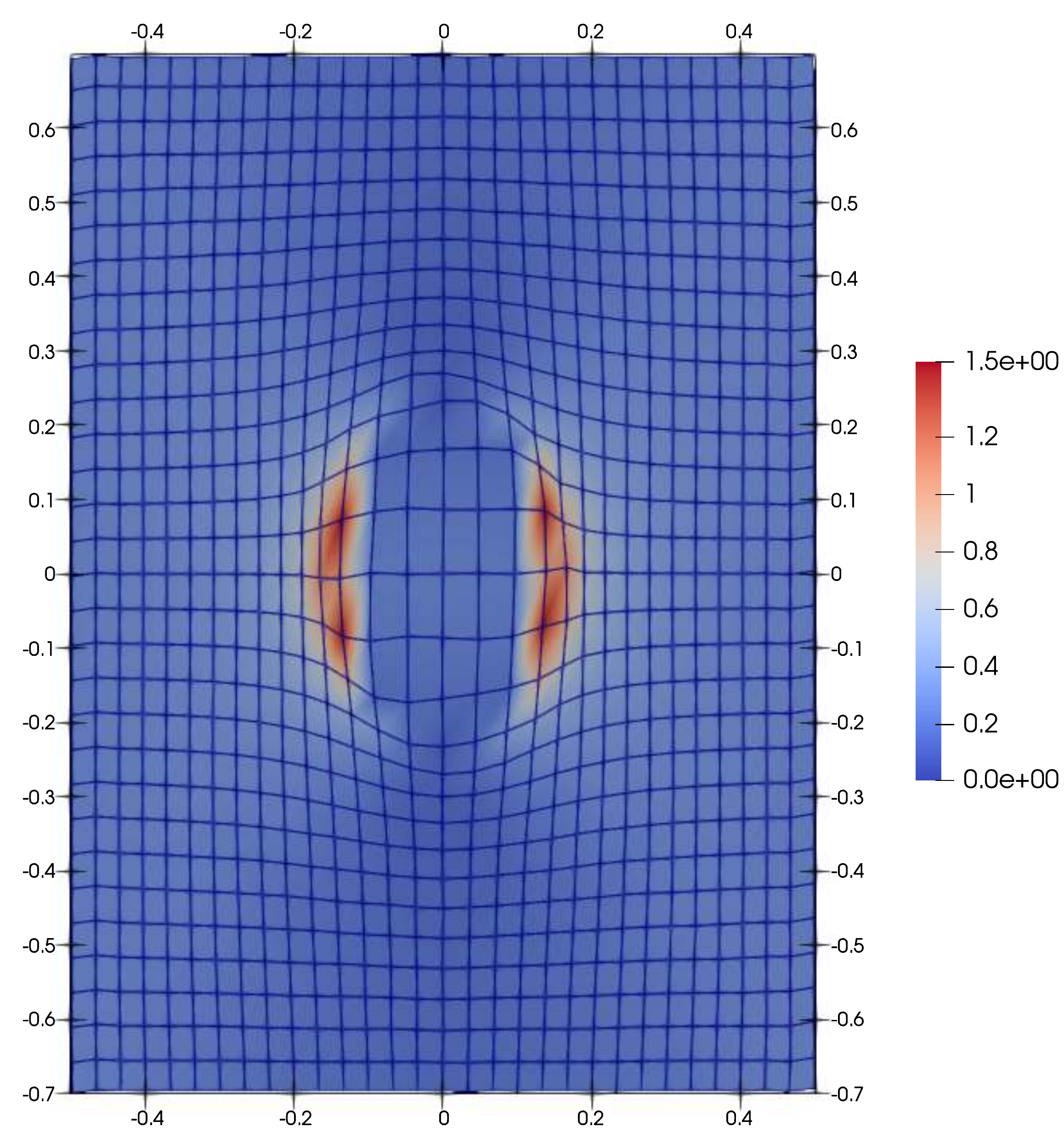}}
    \hfill
	\subfloat[Traction for deformation mode (a)]
	{\includegraphics[width=0.44\textwidth]{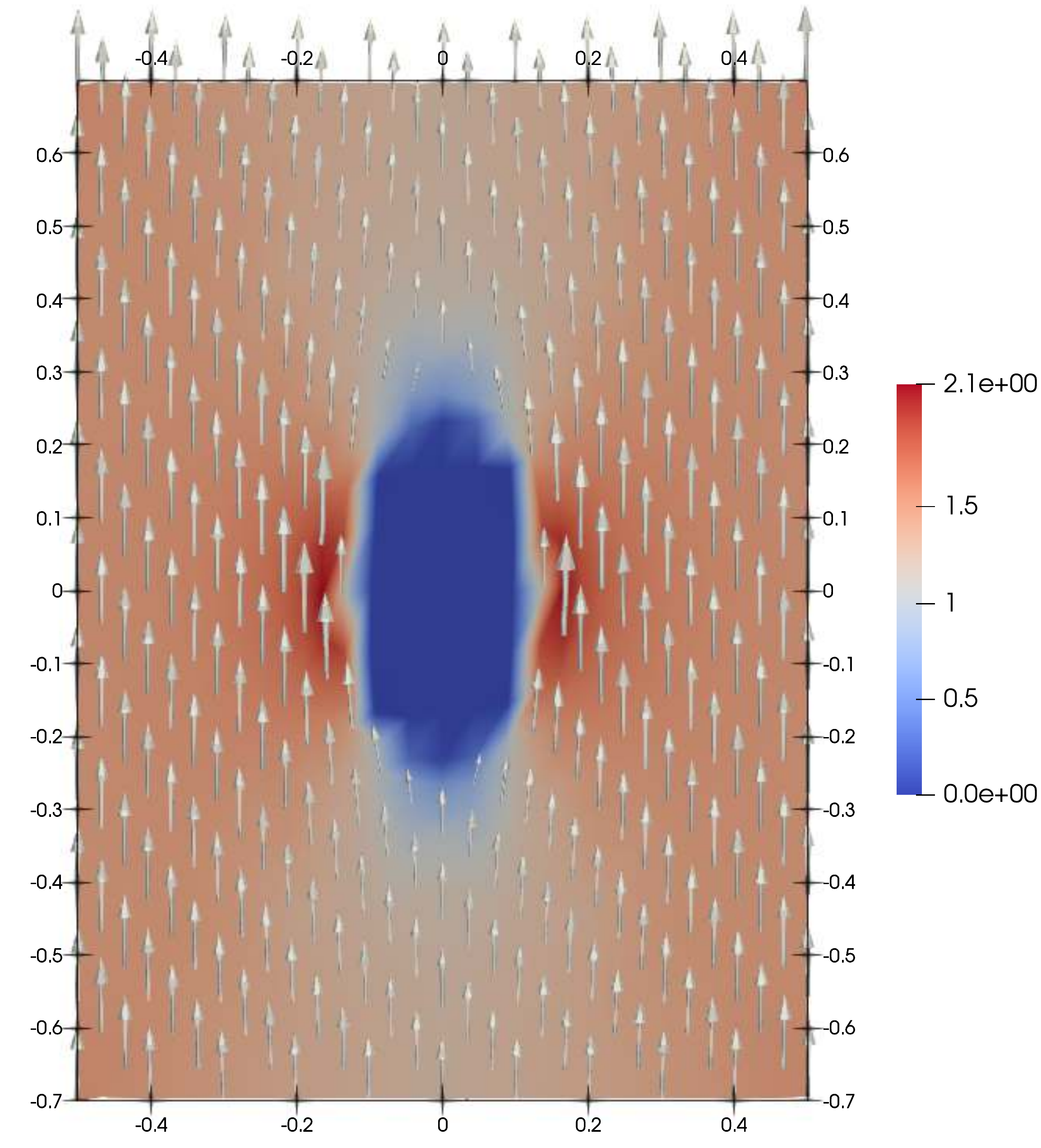}}
	\\
	\subfloat[Elastic energy, overlaid on a mesh showing deformation, for deformation mode (b)]
	{\includegraphics[width=0.44\textwidth]{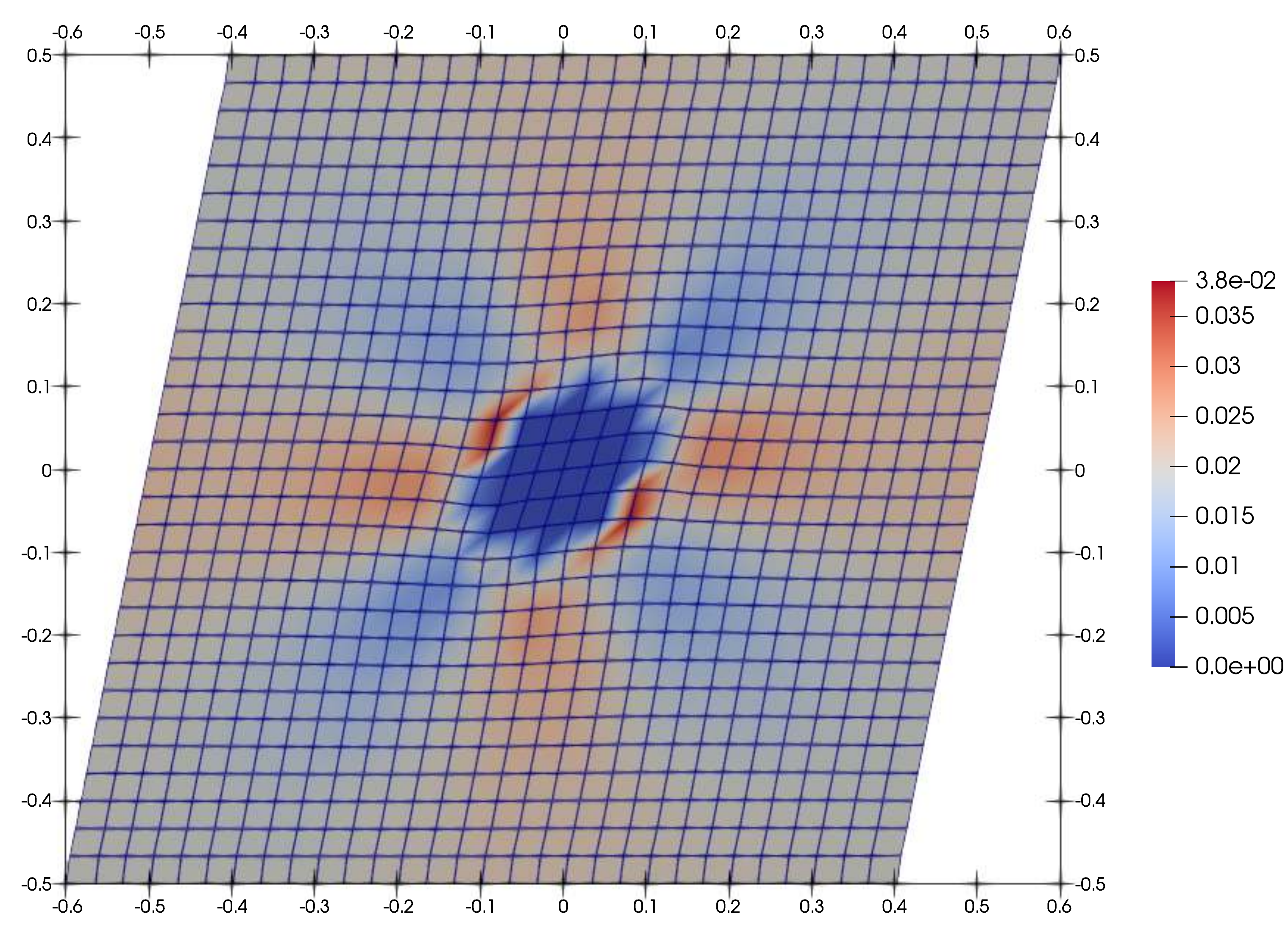}}
    \hfill
	\subfloat[Traction for deformation mode (b)]
	{\includegraphics[width=0.44\textwidth]{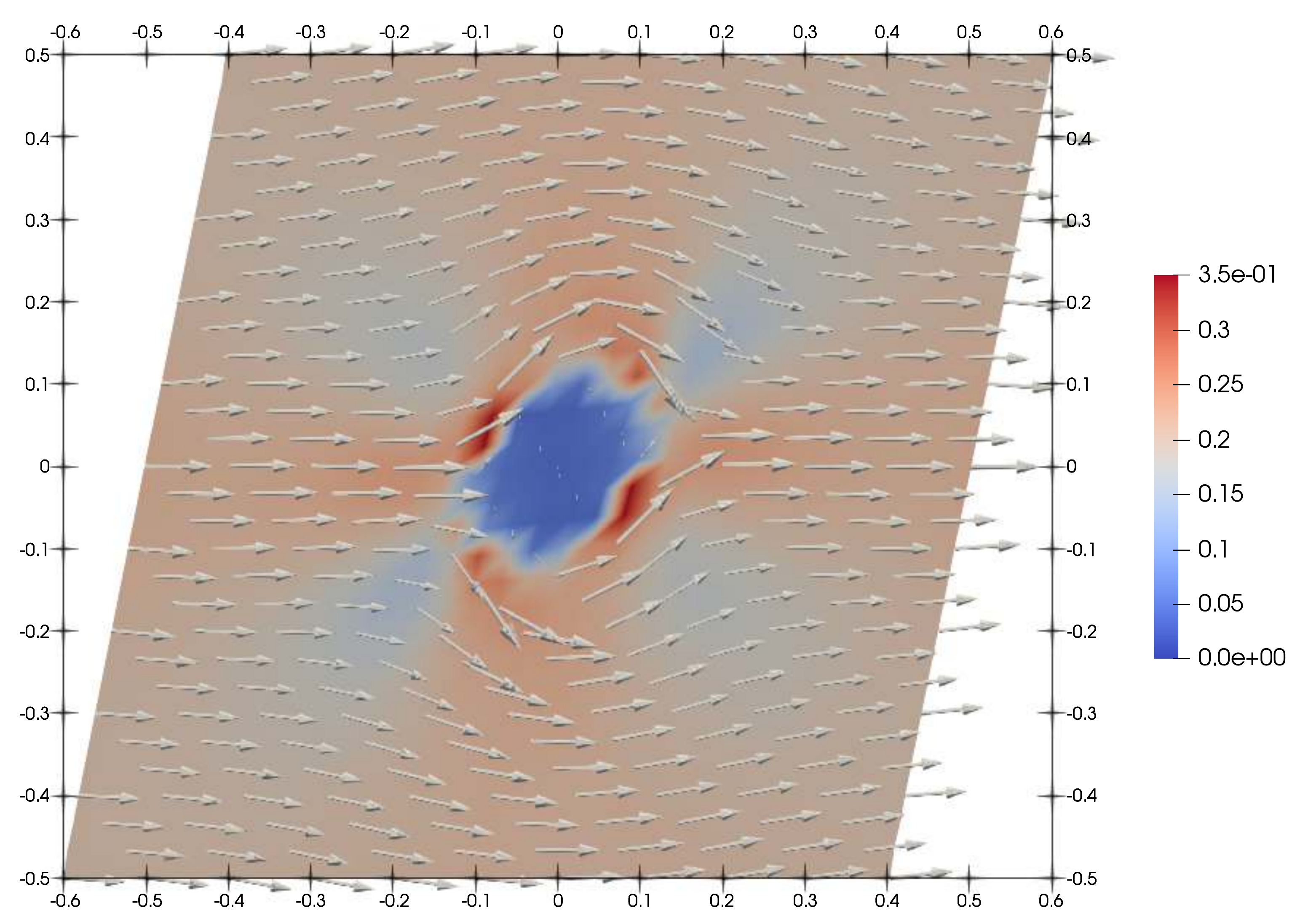}}
	\caption{Left column: elastic energy density (overlaid on a mesh that shows the deformation); right column: traction on the plane with normal $\bfe_2$ for the deformation modes that should have zero effective crack energy.}
	\label{fig:energy and traction at modes of zero energy}
\end{figure}

\begin{figure}[h!]
\centering
	\subfloat[Elastic energy, overlaid on a mesh showing deformation,  for deformation mode (c)]
	{\includegraphics[width=0.42\textwidth]{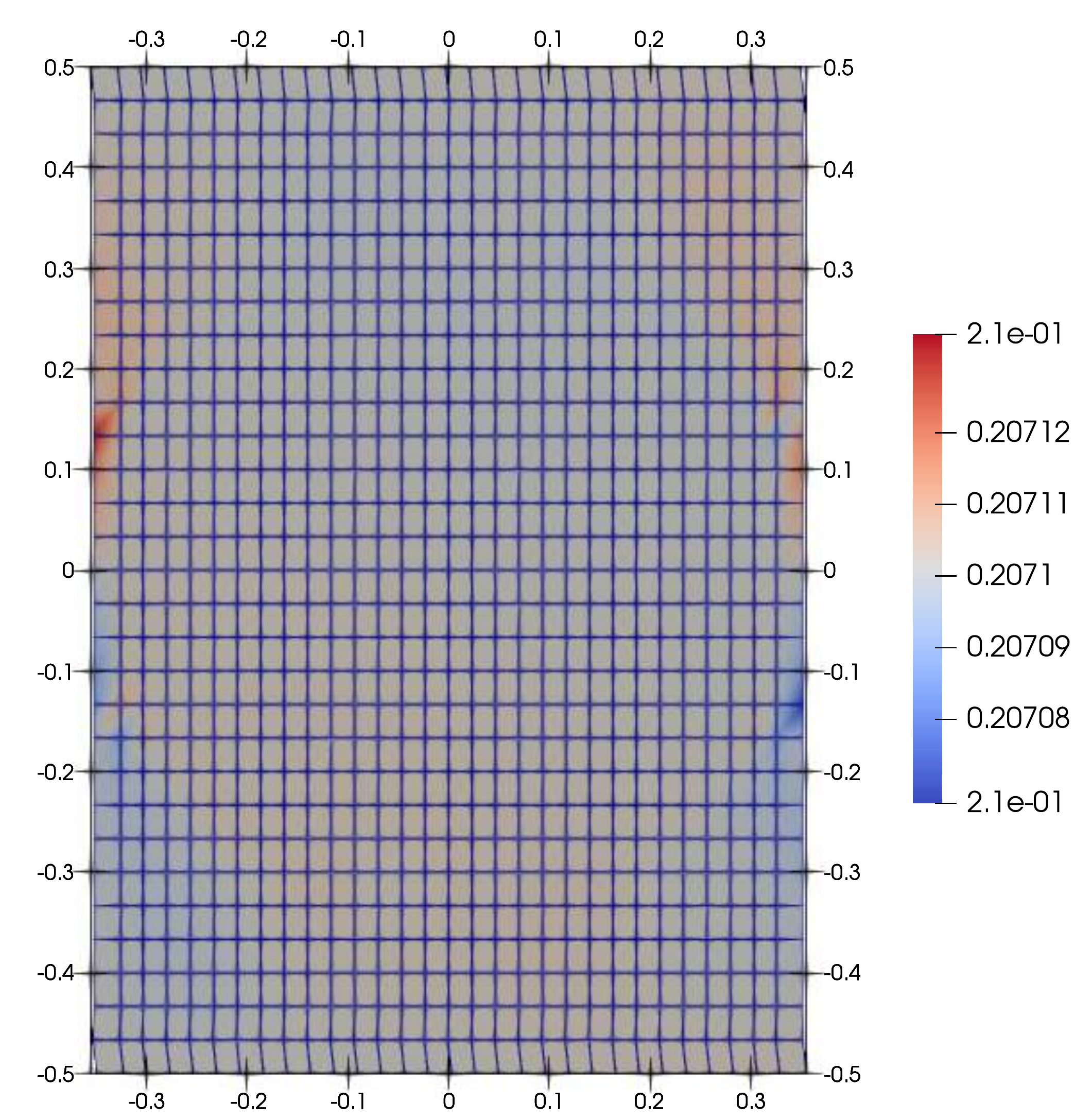}}
    \hfill
	\subfloat[Traction for deformation mode (c)]
	{\includegraphics[width=0.42\textwidth]{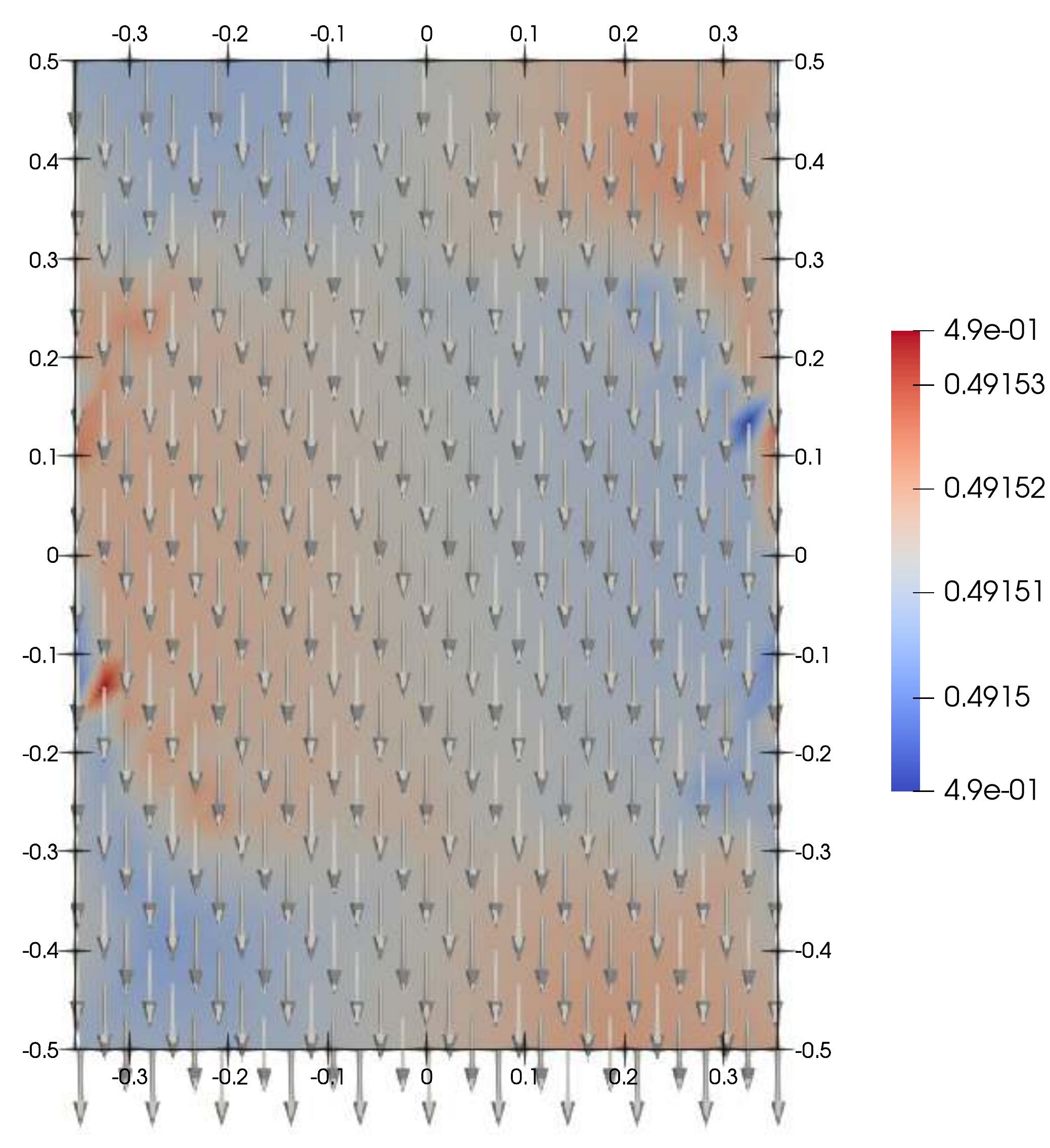}}
	\\
	\subfloat[Elastic energy, overlaid on a mesh showing deformation,  for deformation mode (d)]
	{\includegraphics[width=0.42\textwidth]{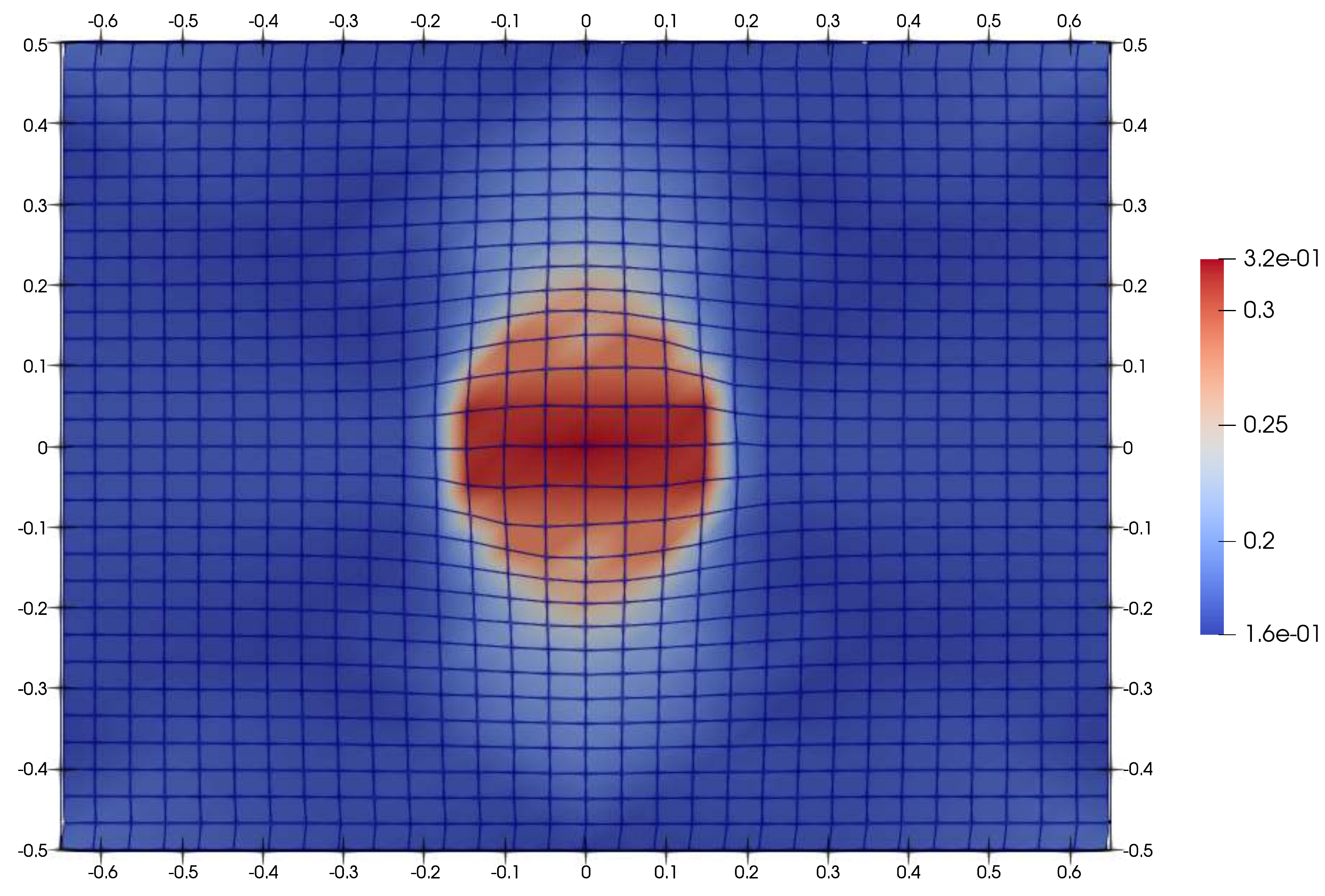}}
    \hfill
	\subfloat[Traction for deformation mode (d)]
	{\includegraphics[width=0.42\textwidth]{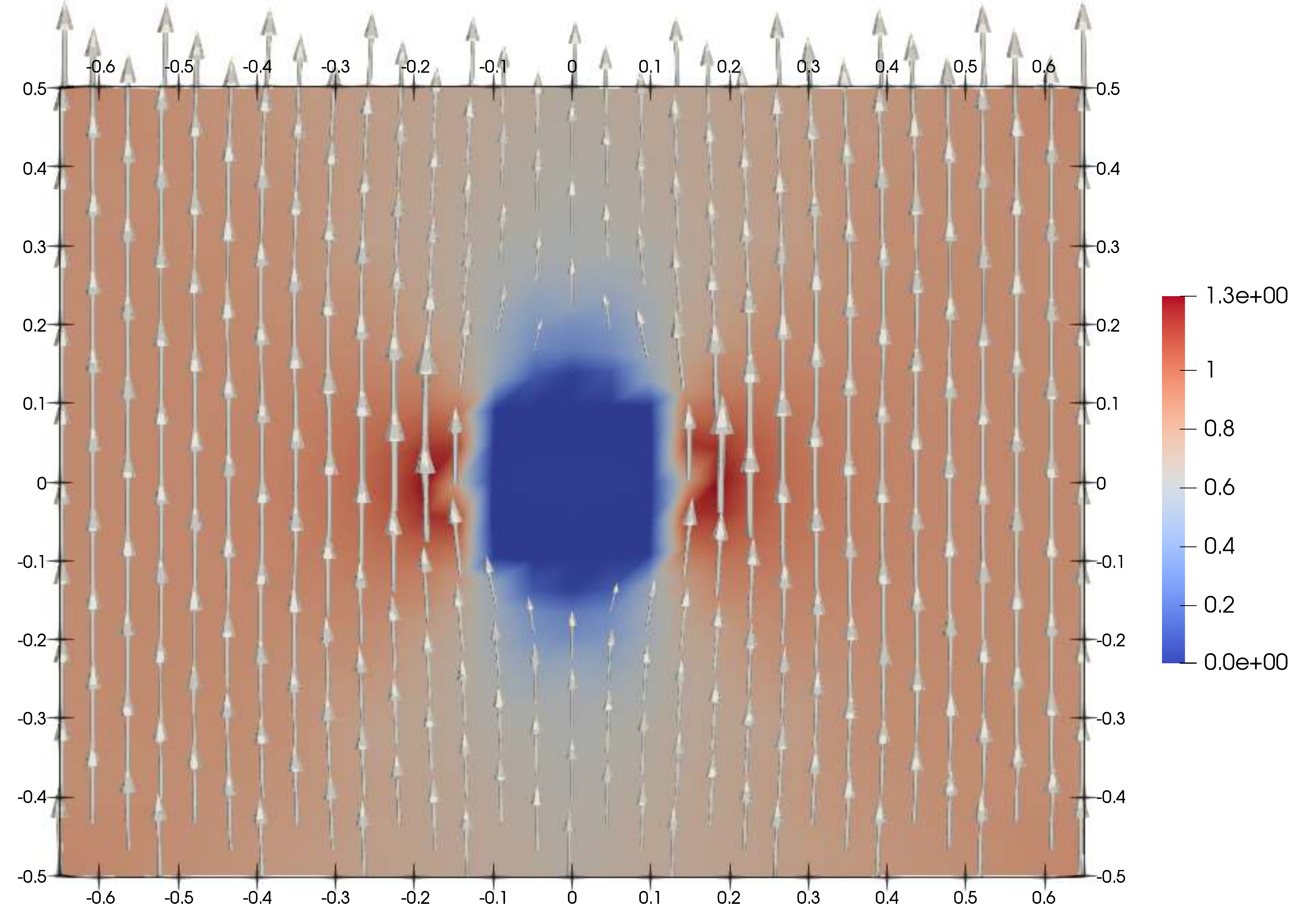}}
	\\
	\subfloat[Elastic energy, overlaid on a mesh showing deformation,  for deformation mode (e)]
	{\includegraphics[width=0.42\textwidth]{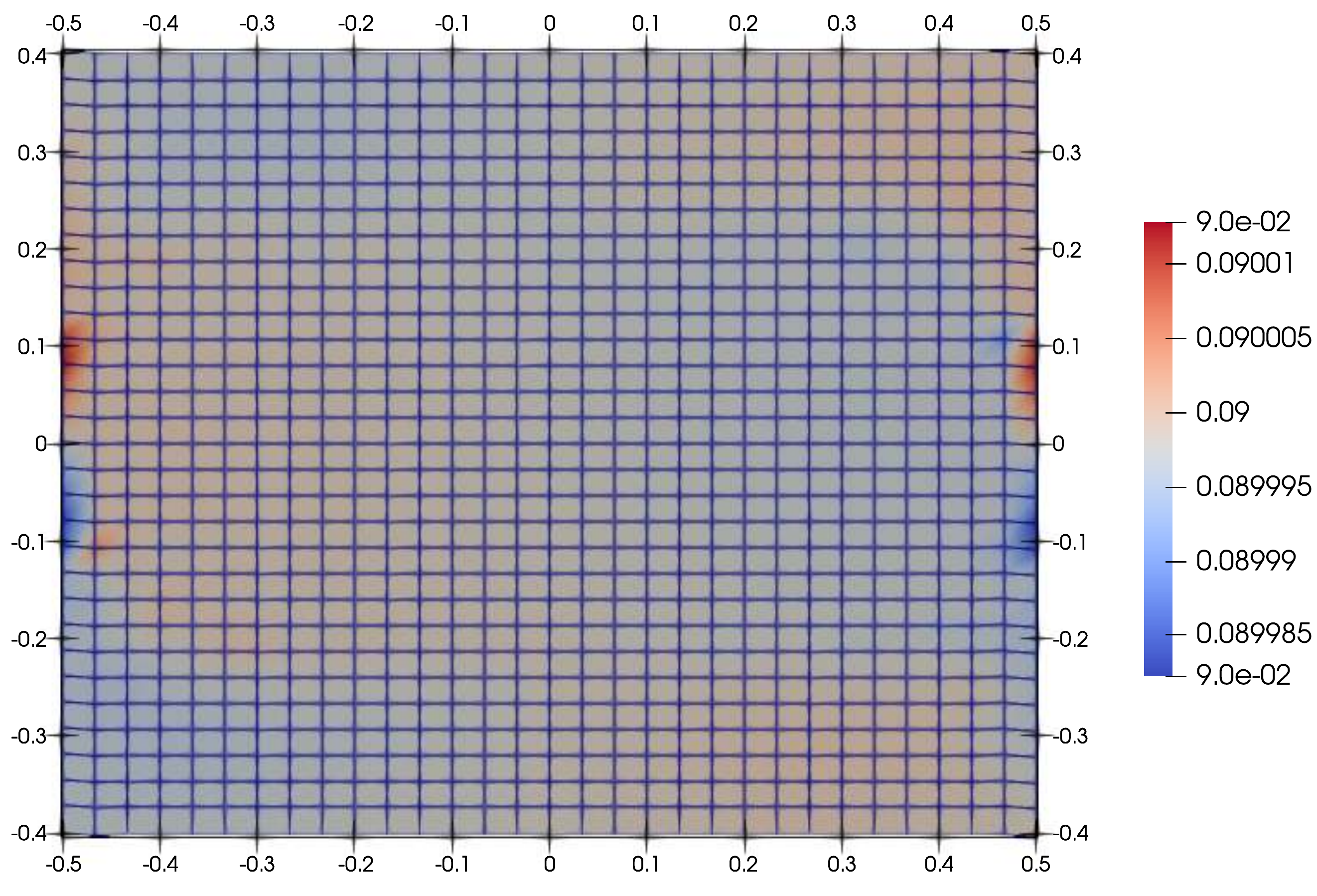}}
    \hfill
	\subfloat[Traction for deformation mode (e)]
	{\includegraphics[width=0.42\textwidth]{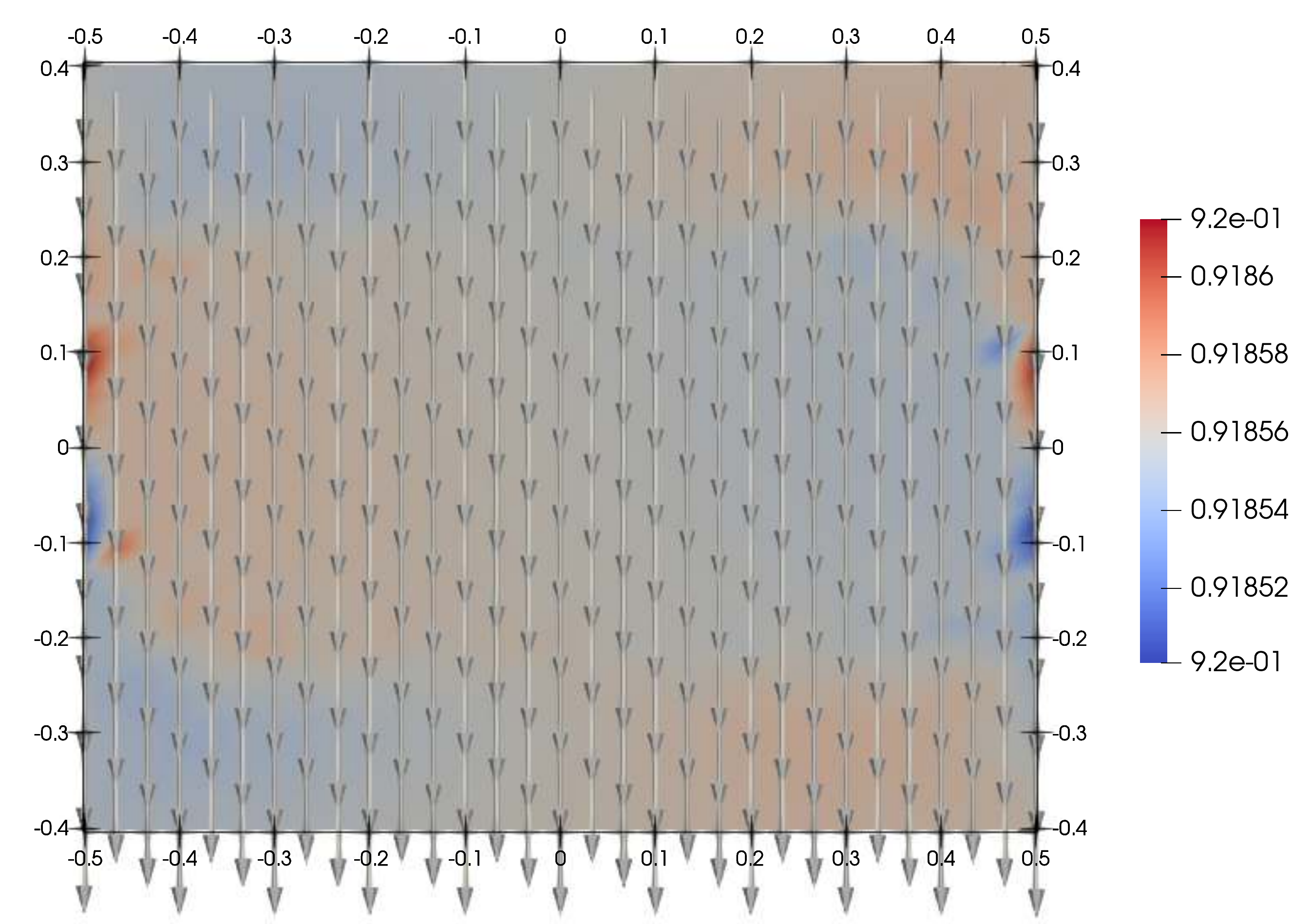}}
	\caption{Left column: elastic energy density (overlaid on a mesh that shows the deformation); right column: traction on the plane with normal $\bfe_2$ for the deformation modes that should have nonzero effective crack energy.}
	\label{fig:energy and traction at modes of positive energy}
\end{figure}

\begin{figure*}[ht!]
    \centering
	\includegraphics[width=0.67\textwidth]{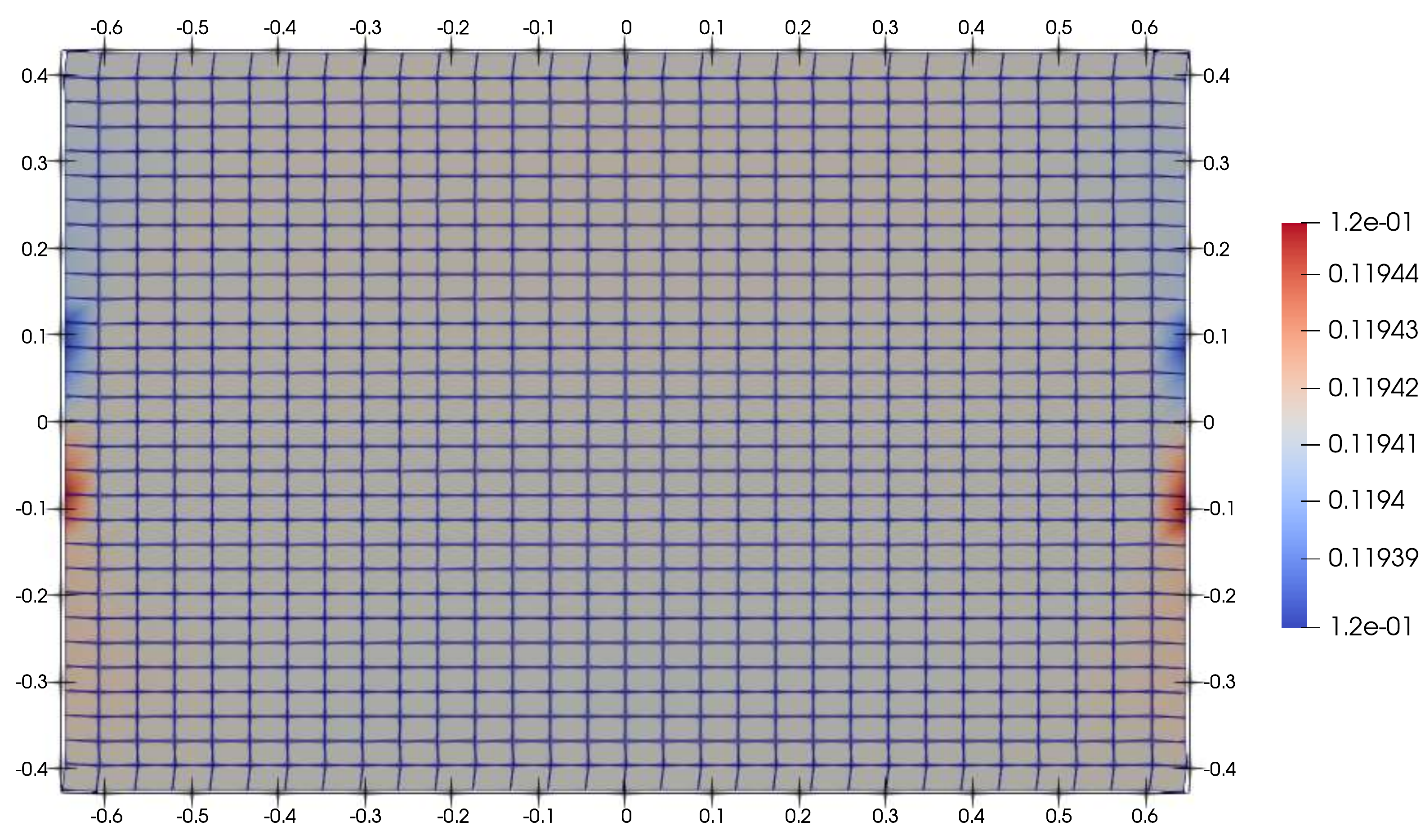}
	\caption{Mode (d), with the affine boundary conditions set up to effectively allow relaxation in the vertical direction. The elastic energy density and the deformation are both essentially uniform.}
	\label{fig:case d with lambda =0.85}
\end{figure*}

\subsection{Crack Growth under Cyclic Shear Loading with Crack Face Contact}

We now consider an example in which cracks change direction, close, and branch due to a cyclic shear loading.
We use a specimen as in Figure \ref{fig:crack growth}(a) that contains an initial crack. 
Our boundary conditions are as follows: on the bottom face, we fix the displacement to zero; on the left and right faces, we fix the traction to zero; and on the top face, we fix the vertical displacement to zero and the horizontal displacement to cause shearing.

First, we shear the specimen to the right, and that causes the initial crack to kink, i.e., change direction, and grow towards the bottom-right (Figs.\ \ref{fig:crack growth}(bc)).
This is consistent with the crack normal being aligned with the direction of maximum tension.
We notice 2 stress concentrations: one is from the crack tip, and the other is at the point that the crack kinks.
The stress concentration at the kink can be understood as, roughly, due to the re-entrant corner.

Next, we reverse the shear on the specimen to the left, and that causes a branch crack to nucleate at the kink -- driven by the high stresses there -- and grow towards the right (Figs.\ \ref{fig:crack growth}(de)).
The direction of maximum tension would suggest that the crack grow towards the top-right, but we impose constraints -- requiring that both components of $\bfd$ are non-negative  -- to drive the crack to grow horizontally after branching.
We do this to set up a situation without reentrant-corner-like geometries, enabling us to better understand the stresses at the branch point.

We highlight that the crack that grew under the shearing to the right now closes and the crack faces contact.
From Figure \ref{fig:crack growth}(e), we see that the elastic energy shows only a very small signature of the presence of the closed crack, suggesting that the model works well in capturing the response of closed cracks.

We also highlight that are higher stresses near the branch point, which would not occur in an idealized Y-configuration crack pattern with a straight upper crack branch.
These stresses occur because $\bfd$ is a smooth field and it rotates as it transitions between crack branches.
That is, the second argument of the effective crack energy density $\Wd(\nabla\bfy,\bfn)$ takes a range of values, and hence the response deviates from the idealized sharp crack response.

\begin{figure}[h!]
\centering
    \subfloat[The initial crack configuration before cyclic shear loading.]{\includegraphics[width=0.45\textwidth]{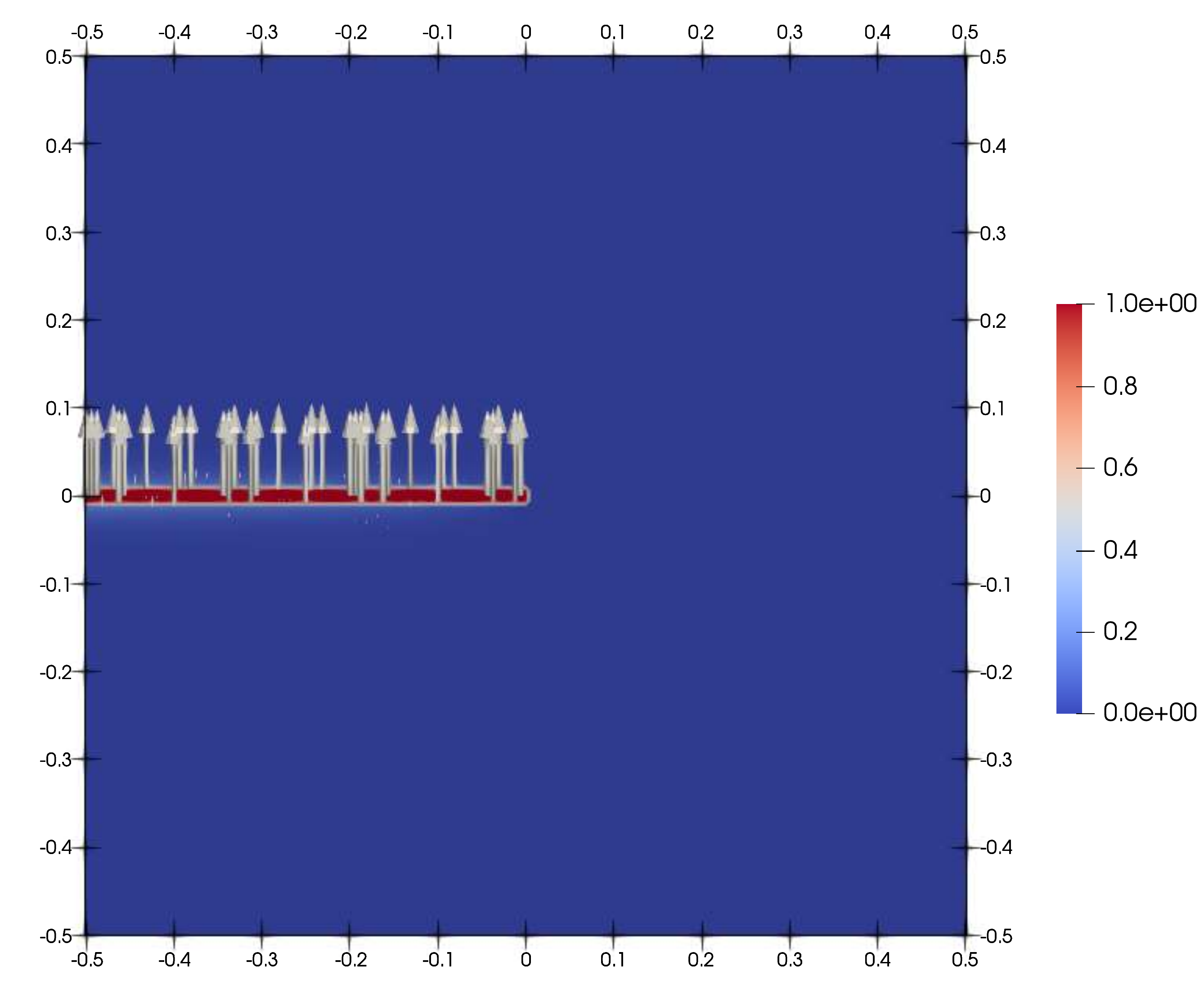}}
    \\
 	\subfloat[Crack changing direction under shear to the right.]
 	{\includegraphics[width=0.45\textwidth]{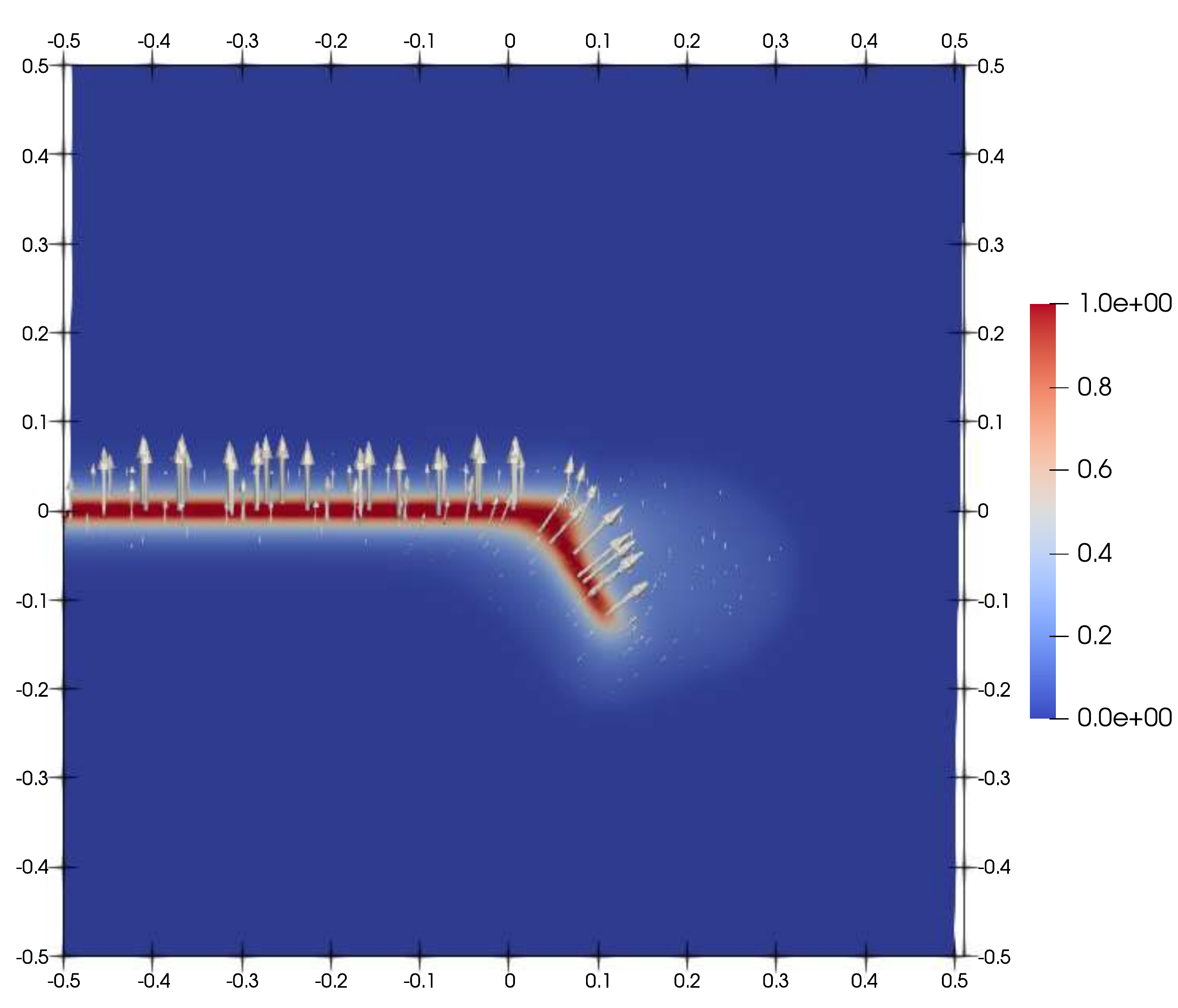}}
     \hfill
 	\subfloat[Elastic energy density under shear to the right.]
 	{\includegraphics[width=0.45\textwidth]{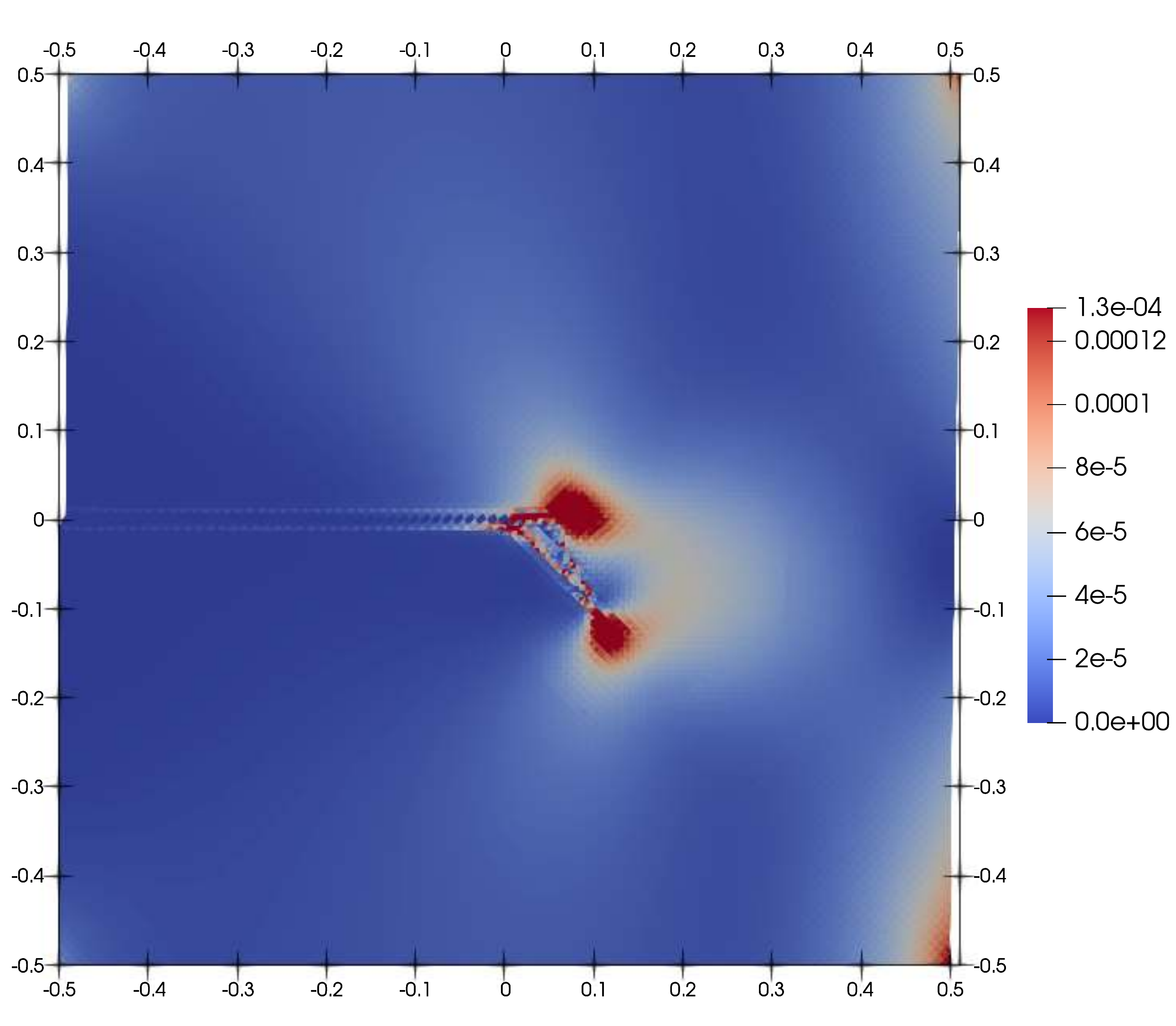}}
	\\
	\subfloat[Crack branching under subsequent shear to the left.]
	{\includegraphics[width=0.45\textwidth]{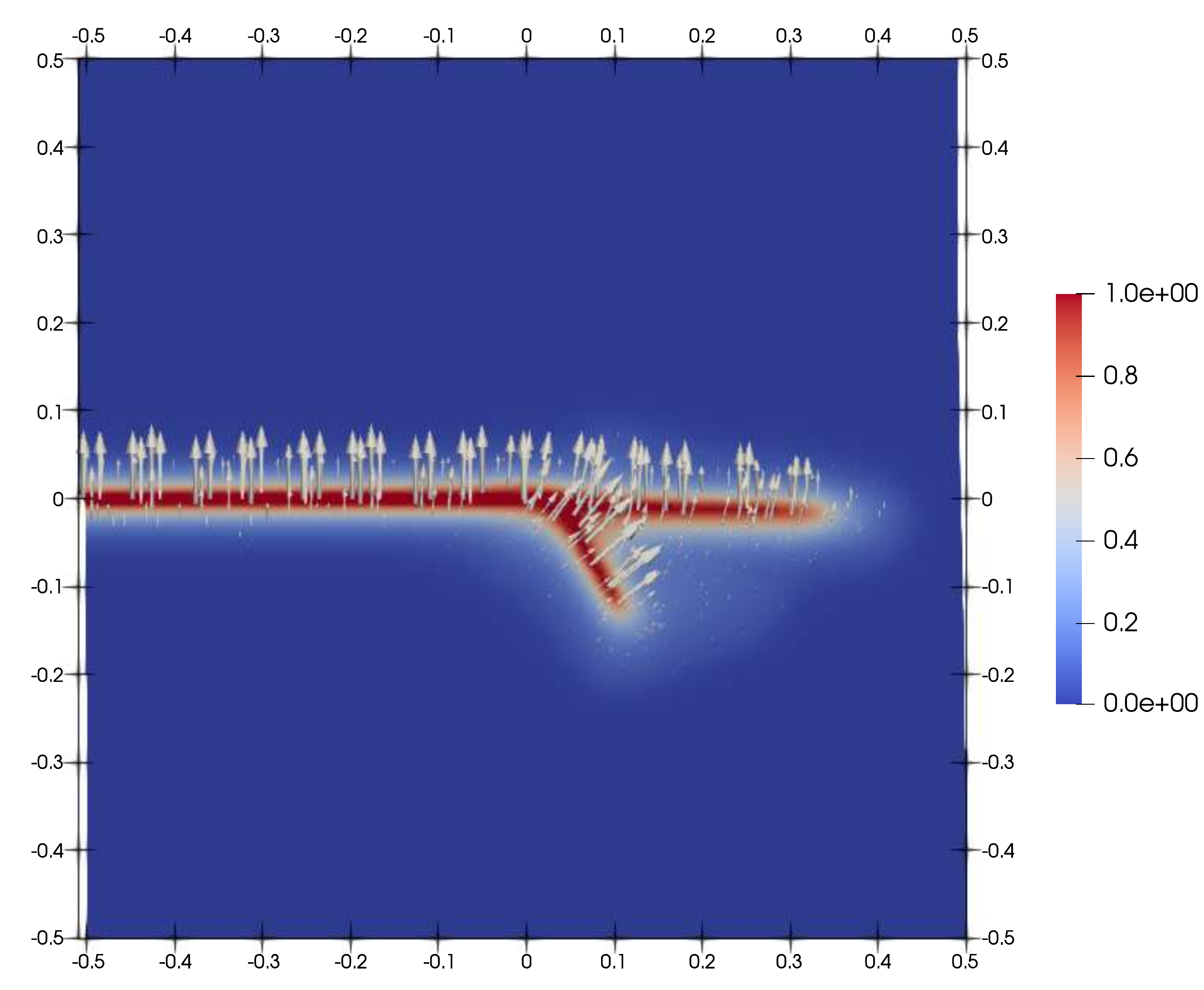}}
    \hfill
	\subfloat[Elastic energy density under subsequent shear to the left.]
	{\includegraphics[width=0.45\textwidth]{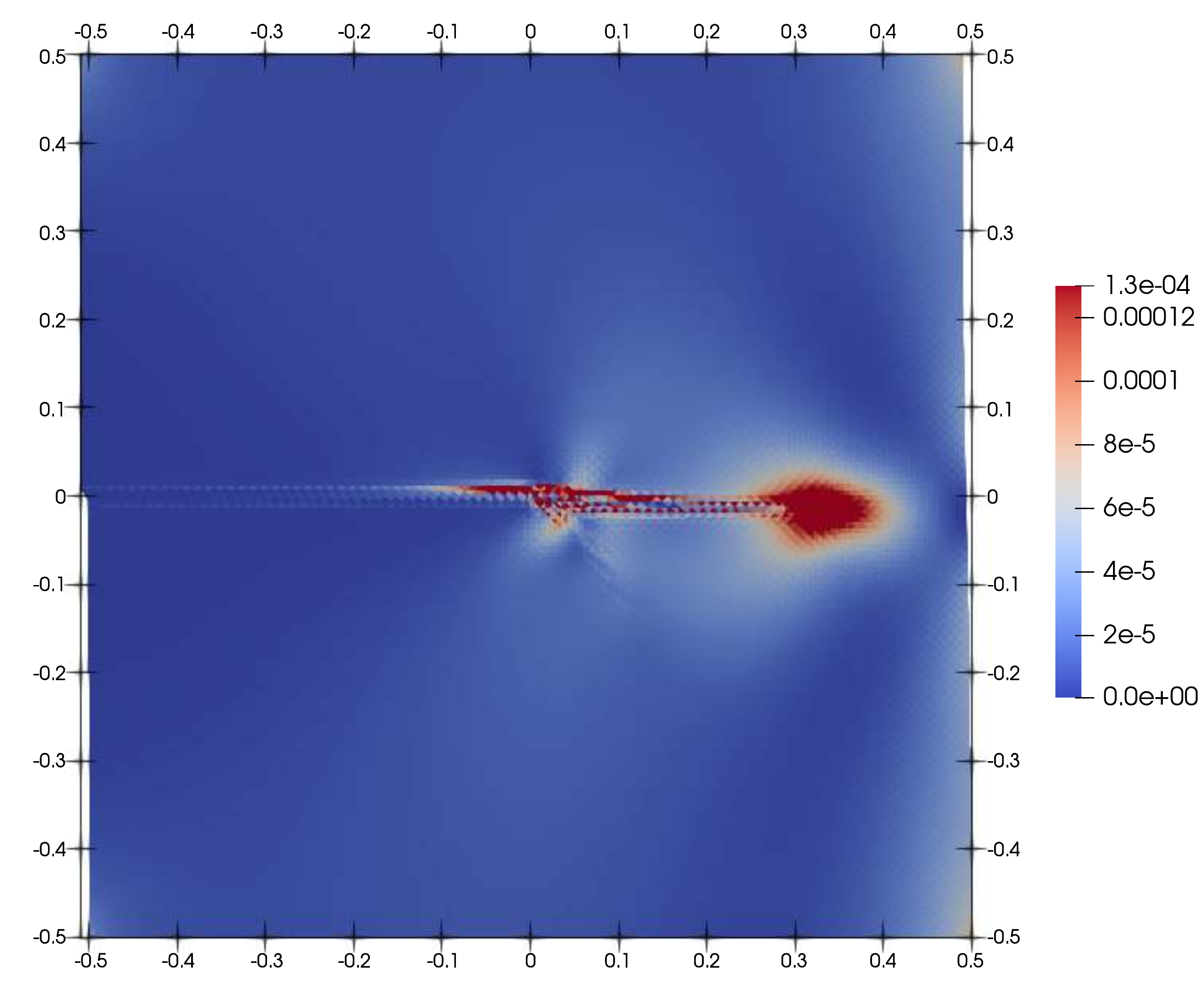}}
	\caption{Crack configuration and elastic energy density under cyclic shear loading.}
	\label{fig:crack growth}
\end{figure}


\subsection{Cracks in a Cavity under Cyclic Loading}

We consider the growth of fracture patterns in a cavity that is subject to a remote load.
While our work is motivated by recent carefully controlled experiments that show a range of interesting behavior \cite{mijailovic2021localized,milner2021dynamic,raayai2019intimate,kim2020extreme}, these experiments have additional physics such as inertia and cavitation that is not considered here.

We consider a square specimen with a circular cavity as in Figure \ref{fig:cavity-fracture}(a). 
We begin with 8 initial short pre-cracks.
We apply the remote loading by imposing affine boundary conditions, i.e., imposing $\bfy = \bfF^0 \bfx$ on the outer boundary, where $\bfF^0$ is proportional to the identity.
The surface of the cavity is traction free.

Figure \ref{fig:cavity-fracture}(b) shows the elastic energy density of the specimen initially under compression as a baseline.
That is, the elastic energy density is what we would have in an intact uncracked specimen.

Figures \ref{fig:cavity-fracture}(c,d) show the crack configuration and elastic energy density, respectively, when the specimen is then loaded under tension.
We find that 4 of the 8 cracks grow while the others do not, and the elastic energy shows the expected regularized stress concentration at the crack tips.
We highlight that the loading had to be imposed in small increments to capture the crack configuration shown in the figure, because the cracks would grow rapidly and reach the boundary soon after the state shown in the figures.

Figures \ref{fig:cavity-fracture}(e,f) show the crack configuration and elastic energy density, respectively, when the specimen is subsequently loaded under compression.
The crack configuration is largely the same, though the cracks have become visually narrower; this is due to our approach to healing, wherein we allow healing when $\bfd$ has magnitude below a critical value (Section \ref{se:numerical}).
An important highlight is that the elastic energy density in Figure \ref{fig:cavity-fracture}(f) is identical to the uncracked case in Figure \ref{fig:cavity-fracture}(b), despite the completely different crack configurations in these settings.
The elastic response of the cracked configuration under compression is identical to the uncracked configuration under compression because all the cracks have closed, suggesting that the model is working well in capturing the desired response.

\begin{figure}[h!]
	\subfloat[Initial unloaded configuration with pre-cracks.]
	{\includegraphics[width=0.42\textwidth]{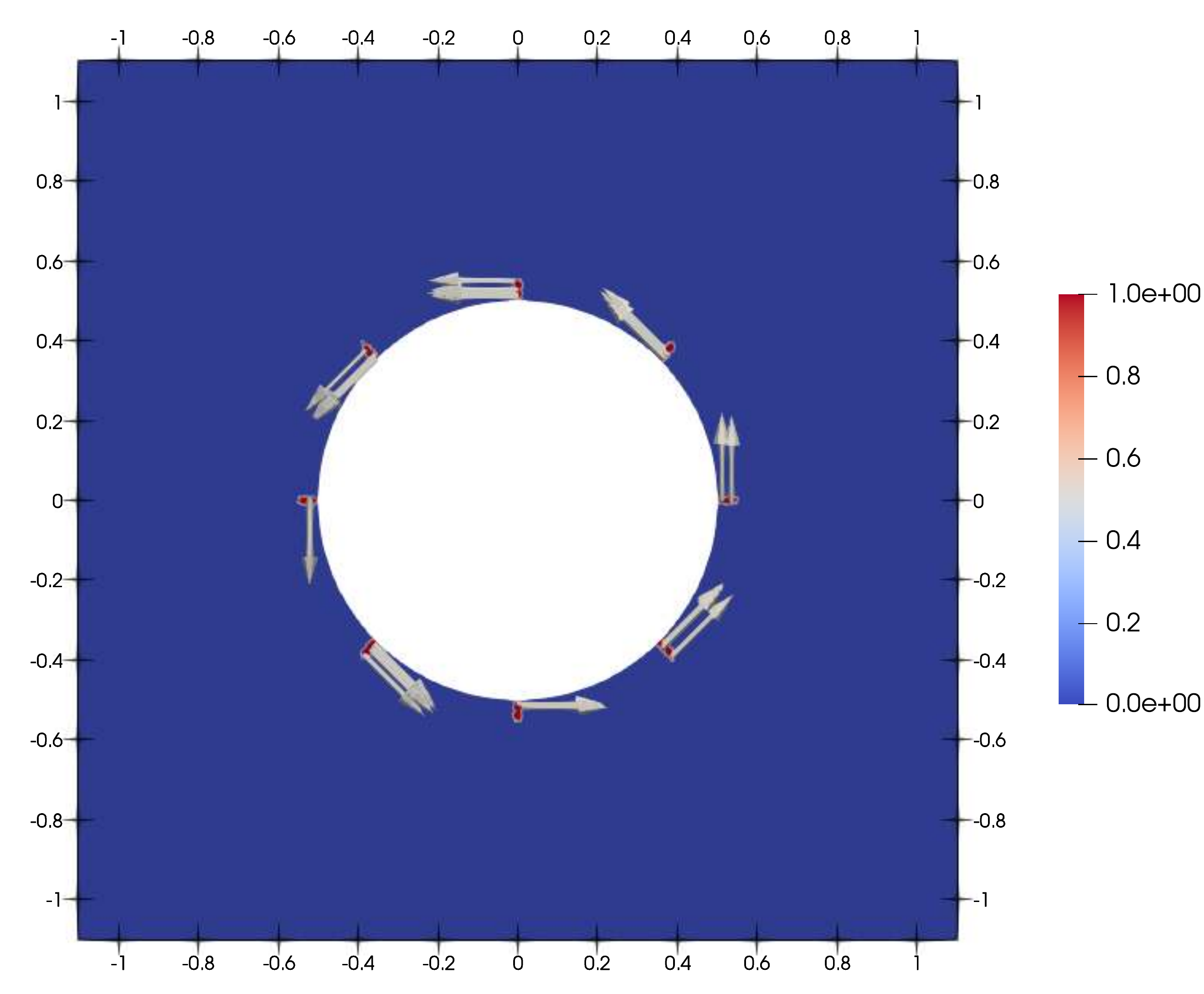}}
    \hfill
	\subfloat[Elastic energy density of the initial configuration under compression.]
	{\includegraphics[width=0.42\textwidth]{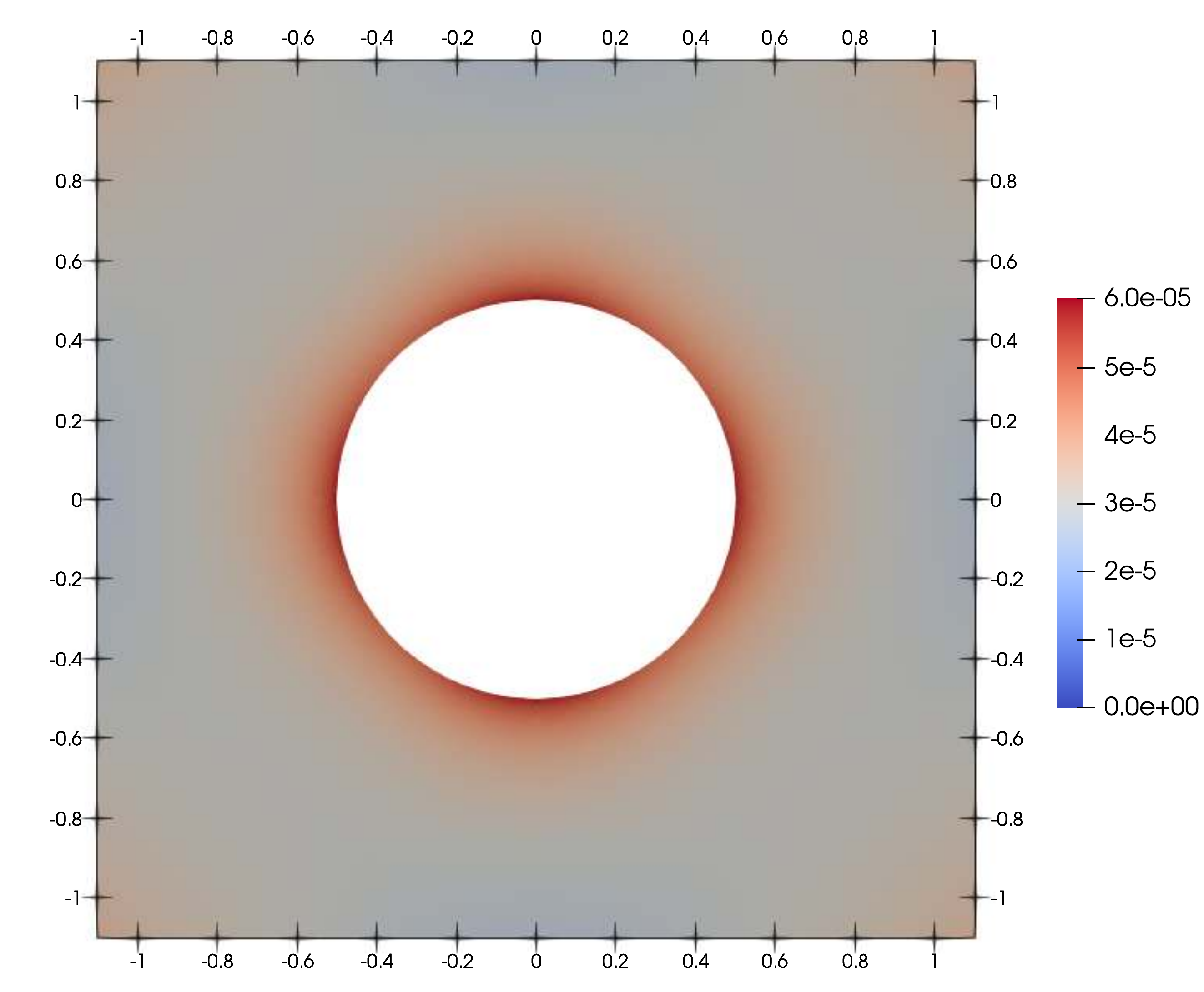}}
	\\
	\subfloat[Crack growth under tensile load.]
	{\includegraphics[width=0.43\textwidth]{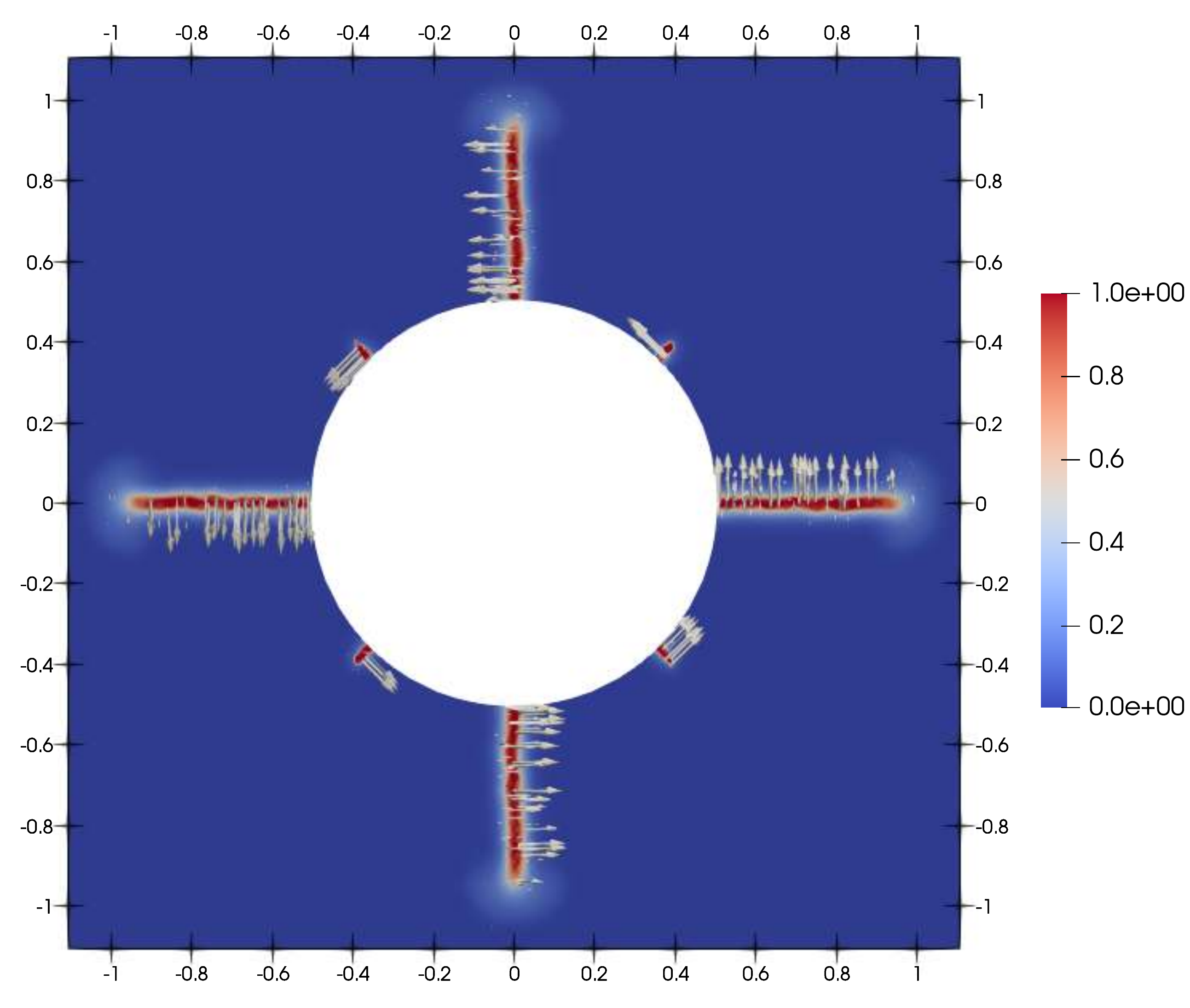}}
    \hfill
	\subfloat[Elastic energy density of the cracked configuration under tensile load.]
	{\includegraphics[width=0.43\textwidth]{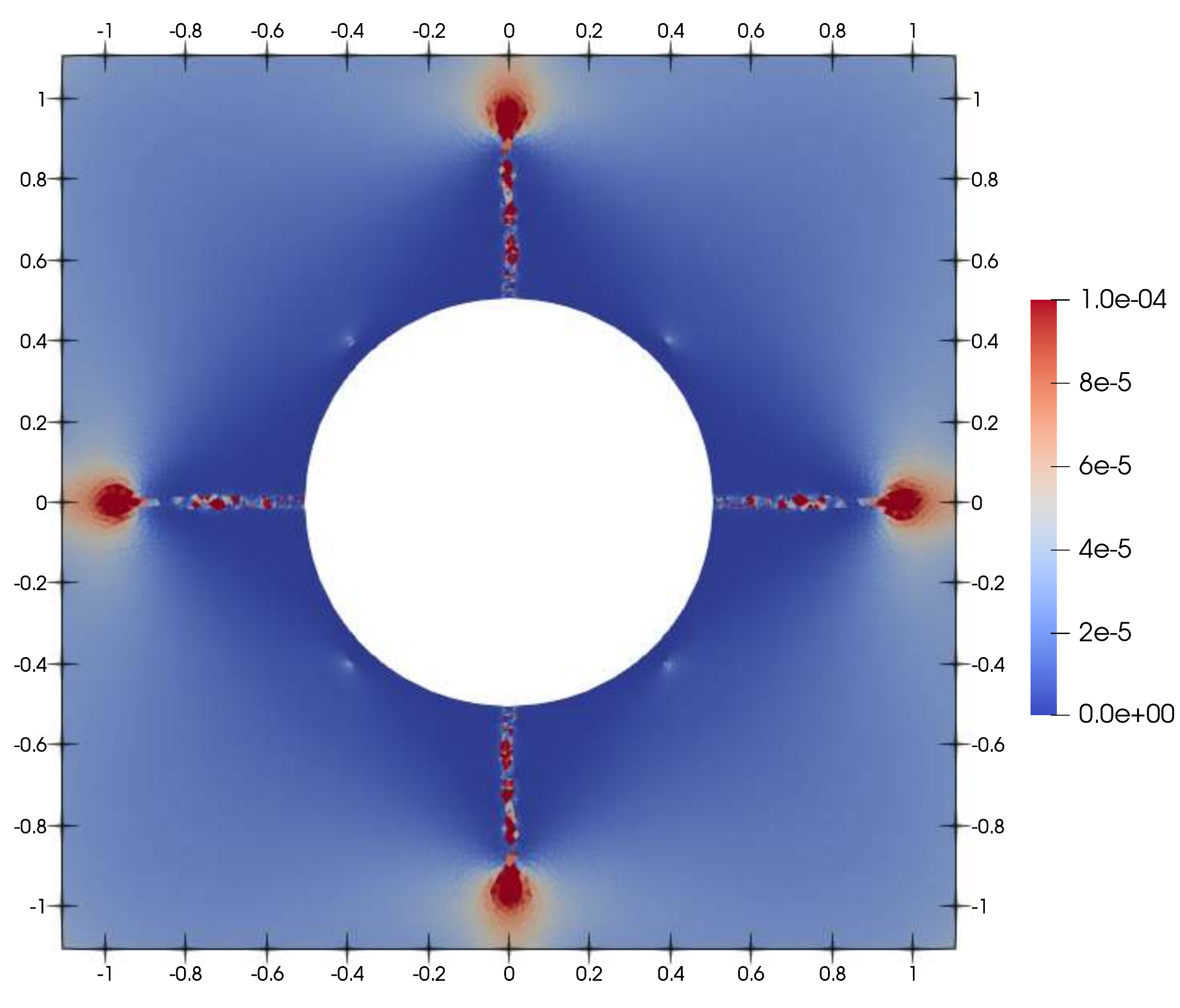}}
	\\
	\subfloat[Crack configuration under a second cycle of compressive load; crack patterns do not evolve under compression.]
	{\includegraphics[width=0.42\textwidth]{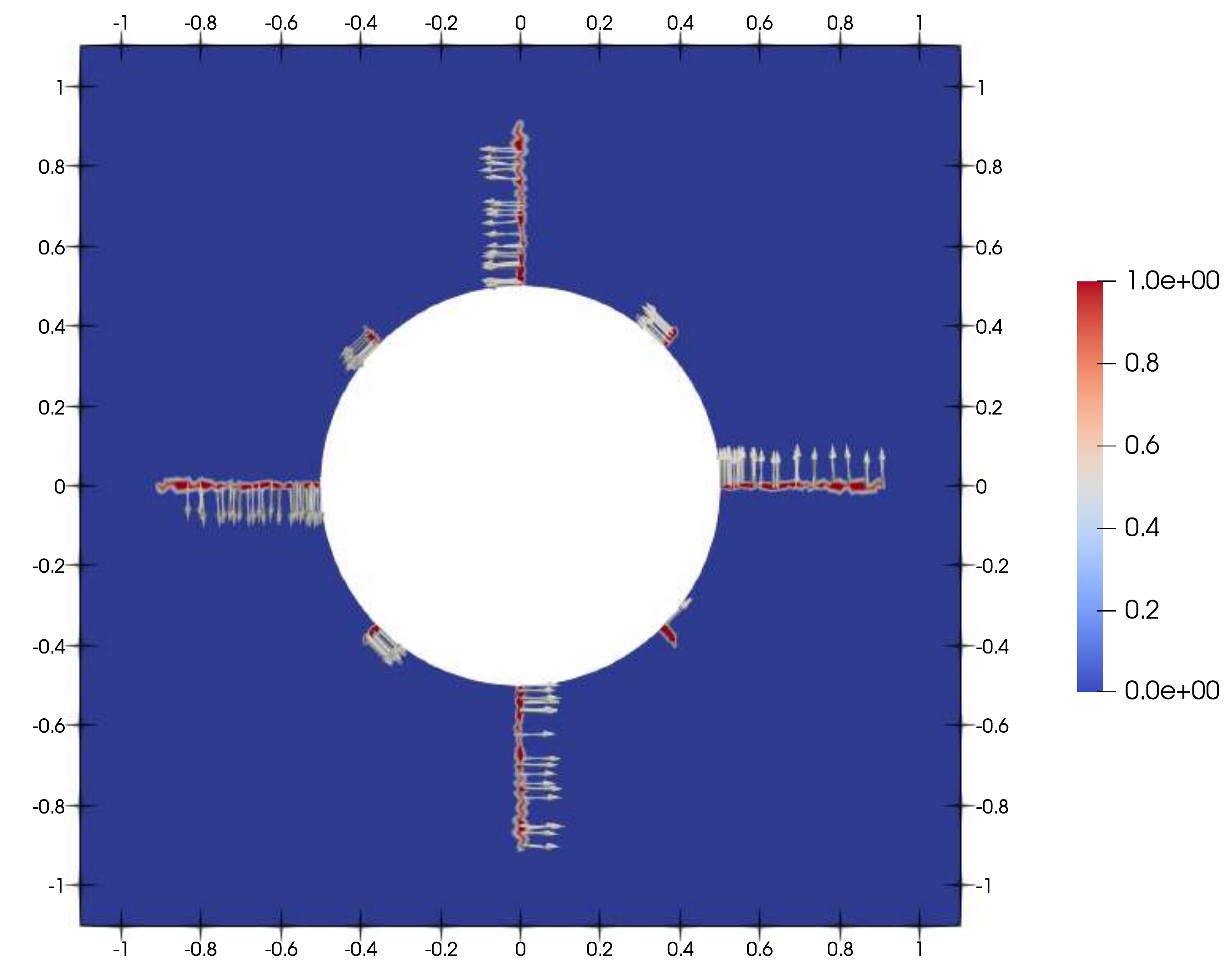}}
    \hfill
	\subfloat[Elastic energy density of the cracked configuration under compressive load; response is identical to the intact initial configuration.]
	{\includegraphics[width=0.42\textwidth]{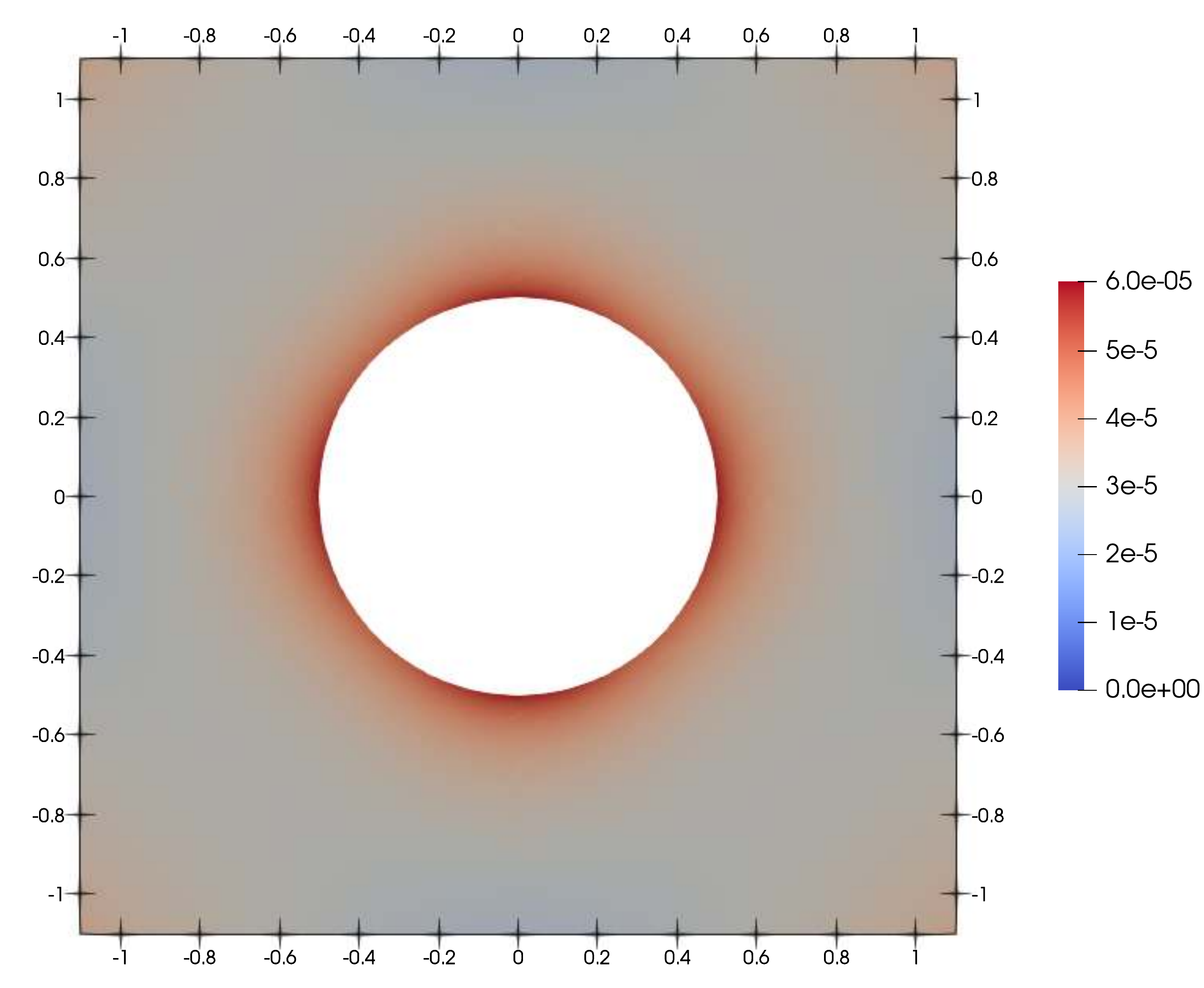}}
	\caption{Crack growth in a circular cavity under cyclic remote loading.}
	\label{fig:cavity-fracture}
\end{figure}


\section{Discussion}\label{se:discussion}

In this paper, we have presented an effective energy for phase-field cracks that provides the correct crack response when the crack is subject to complex loadings that cause contact across the crack faces.
Our approach is valid in the setting of finite deformations, which enables application to soft materials in complex configurations (e.g., to complement theoretical studies such as \cite{mo2022finite,geubelle1994finite,knowles1983large}), as well as to stiff materials where strains are small but fractured pieces can have large rotations or where compression driving crack closure is significant (e.g., \cite{huang2022cohesive}).

While the numerical results presented above are promising, there remain several important directions for future research:
\begin{itemize}
    \item Quasiconvexity is a central property for well-posed problems in finite elasticity, e.g., \cite{Braides98,antman2005problems,DaFrTo05}, and it remains to check that $\Wd$ is quasiconvex in the $\bfF$ variable.
    While lack of convexity typically leads to instabilities or microstructure, our numerical computations have not shown that behavior.
    In this context, we notice that an alternative definition of $\Wd$ could have been
    \[
     \Wd(\bfF, \bfn) = \min_{\bfy = \bfF \bfx \text{ on } \partial Q_{\bfn}} \int_{Q_{\bfn}} W(\nabla\bfy) \dm V_{\bfx} ,
    \]
    where: $Q_{\bfn}$ is the cube of volume one centered at the origin with edges parallel to the orthonormal basis $\{ \bft_1, \bft_2, \bfn \}$; $\bfy$ is Sobolev in the two half-cubes $Q_{\bfn}^{\pm} = \{ \bfx \in Q_{\bfn} : \pm \bfx \cdot \bfn > 0\}$, but can have a jump in the interface $\{ \bfx \in Q_{\bfn} : \bfx \cdot \bfn = 0\}$; $\bfy$ does not interpenetrate.
    This definition renders a quasiconvex $\Wd$ and, in addition, simulates the effective response of two blocks under the affine deformation $\bfF$ applied on the boundary (Fig. \ref{fig:effective-response-2-blocks}).
    \begin{figure*}[ht!]
    	\includegraphics[width=0.87\textwidth]{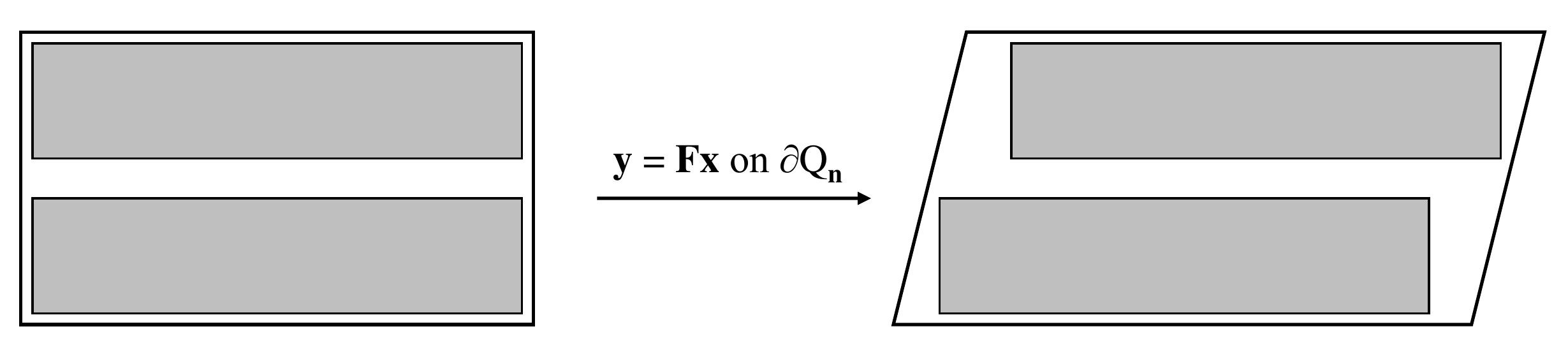}
    	\caption{Obtaining the effective crack energy through the homogenized effective response of two blocks.}
    	\label{fig:effective-response-2-blocks}
    \end{figure*}
    However, explicitly carrying out the minimization is challenging.
    
    \item The limiting sharp-interface model that appears as $\epsilon\to 0$, and the corresponding consequences on the crack face tractions, needs to be clarified. 

    \item The numerical studies on branched cracks show that the crack normal is not well defined at the branch point, which leads to undesired responses at the branch point when the crack is subject to complex loads. While it is not immediately clear to us as to how to eliminate this, possibly it is useful to borrow from damage theory in using a tensorial parameter \cite{lemaitre1994mechanics} that can potentially represent a richer kinematics such as multiple normals at a spatial location.

    \item We have considered the setting of idealized frictionless crack faces. A natural extension is to consider friction, which is relevant to a large class of materials. While relaxation of the energy is not applicable to this setting, the QR decomposition provides an approach to transparently separate the relative slip of the crack faces from other deformation modes. Therefore, we expect that the kinematics introduced in this paper will be useful in considering more realistic boundary conditions.
    
\end{itemize}

\section*{Software Availability}

A version of the code developed for this work is available at  \\
\url{https://github.com/maryhzd/Phase-field.git}

\begin{acknowledgments}
    We thank George Gazonas, Anthony Rollett, and Noel Walkington for useful discussions;
    Pradeep Sharma for pointing us to relevant prior work in continuum damage mechanics;
    Agencia Estatal de Investigaci\'on of the Spanish Ministry Research and Innovation (Project MTM2017-85934-C3-2-P), Army Research Office (MURI W911NF-19-1-0245), Office of Naval Research (N00014-18-1-2528), and National Science Foundation (CMMI MOMS 1635407, DMREF 2118945, DMS 2108784) for support; 
    Center for Nonlinear Analysis for hosting a visit by Carlos Mora-Corral;
    and National Science Foundation for XSEDE resources provided by Pittsburgh Supercomputing Center.
    This paper draws from the doctoral dissertation of Vaibhav Agrawal at Carnegie Mellon University.
\end{acknowledgments}

\appendix

\makeatletter
\renewcommand*{\thesection}{\Alph{section}}
\renewcommand*{\thesubsection}{\thesection.\arabic{subsection}}
\renewcommand*{\p@subsection}{}
\renewcommand*{\thesubsubsection}{\thesubsection.\arabic{subsubsection}}
\renewcommand*{\p@subsubsection}{}
\makeatother

\section{Shortcomings of the Energy Splitting Method}
\label{sec:Vaibhav-comparison}

The simplest approach to model the energy density of the damaged volume is to set the damaged energy to $0$.
However, this leads to unphysical behavior such as interpenetration of the crack faces and an incorrect mechanical response of a body with a crack.
The papers \cite{miehe-ijnme} and \cite{amor-jmps} proposed different ways of decomposing the energy, and associating only certain terms of the energy to the damaged region.
While the energy splitting method vastly improved the difficulties caused by simply using zero energy for the crack, it does not consider the crack orientation in the energetic decomposition.
Consequently, energy splitting leads to unphysical response in some settings; further, it is unclear how to extend it beyond the setting of linear isotropic elasticity.
We discuss below some simple examples where the splitting methods provide unexpected results.
We emphasize that the incorrect stress field at crack tips can affect the crack growth behavior significantly, even if the stresses are largely correct away from the crack tip.

In all examples, we consider a large specimen with a given far-field stress or strain, and a finite crack oriented with normal $\bfe_2$.
We use $\bfsigma_{split}$ to denote the response of the cracked material predicted by the energy splitting model. 

We define the intact energy and the corresponding stress by
\begin{equation}
     W_{intact}(\bfvareps) = \half\lambda (\tr \bfvareps)^2 + \mu |\bfvareps|^2
     \quad \Rightarrow \quad
     \bfsigma_{intact} = \lambda\left(\trace{\bfvareps}\right)\bfI + 2\mu\bfvareps .
\end{equation}
A value $\mu>0$ is required for convexity of $W_{intact}$.
For simplicity we consider $\lambda>0$, and consequently the Poisson ratio $\nu$ is positive.

\subsection{Splitting Based on the Principal Strain}\label{subse:splitting}

Following \cite{miehe-ijnme}, the compressive ($\psi_0^-$) and tensile ($\psi_0^+$) energies are defined by
\begin{equation*}
    \psi_0^{\pm} 
        := 
        \half\lambda\langle\eps_1 + \eps_2 +\eps_3 \rangle_\pm^2
        + \mu \langle\eps_1\rangle_\pm^2+\langle\eps_2\rangle_\pm^2+\langle\eps_3\rangle_\pm^2 ,
\end{equation*}
where $\eps_1, \eps_2, \eps_3$ denote the principal strains, and the corresponding principal directions are $\bfn_1, \bfn_2, \bfn_3$.
Also, $\langle x \rangle_{+} = \max \{0, x\}$ and $\langle x \rangle_{-} = \min \{0, x\}$.
The effective energy of the crack is assumed to consist of only the compressive part $\psi_0^-$.
The corresponding stress is
\begin{equation}
    \bfsigma_{split} 
    = 
        \sum_{a=1}^3 
            \left(\lambda\langle\eps_1+\eps_2+\eps_3\rangle_{-} + 2\mu\langle\eps_a\rangle_{-}\right)
            \bfn_a\otimes\bfn_a .
\end{equation}
We notice from this expression that $\bfsigma_{split}$ is always non-tensile in every direction.

\paragraph{Uniaxial tension parallel to the crack.}

Consider the specimen subject to a far-field stress $\sigma_0 \bfe_1 \otimes \bfe_1$ with $\sigma_0 > 0$, giving a far-field strain of:
\begin{equation}
\label{eqn:split-ex-1}
    \bfvareps = \frac{\sigma_0}{\mu(3\lambda+2\mu)}
        \begin{pmatrix}
            \lambda+\mu & 0 & 0
            \\
            0 & -\lambda/2 & 0
            \\
            0 & 0 & -\lambda/2
        \end{pmatrix} .
\end{equation}

The corresponding stress response in the crack in this model is $\bfsigma_{split} = -\frac{\sigma_0}{3\lambda+2\mu} \left(\bfe_2\otimes\bfe_2+\bfe_3\otimes\bfe_3\right)$.

We emphasize two aspects of this result.
First, there is an unexpected (since $\nu>0$) compressive stress along the $\bfe_2$ direction.
Second, the stress along the $\bfe_1$ direction is zero, while the expectation is that for far-field stress parallel to crack, the material should behave as an intact material and sustain the far-field stress as it is. 
The incorrect stress field can potentially cause spurious crack growth: the effective zero stiffness for tension parallel to cracks can lead to spurious stress concentration around the tip as the intact material ahead of the tip can sustain far-field stress in $\bfe_1$ direction but the cracked phase cannot. 
Further, this can have a significant spurious impact on the T-stress which is an important parameter for ductile crack growth \cite{ravindran2021fracture}.

\paragraph{Shear traction across the crack.}

Consider the specimen subject to a far-field strain given by
\begin{equation*}
    \bfvareps 
    = 
    \frac{\tau}{2\mu}
    \begin{pmatrix}
        0 & 1 & 0\\
        1 & 0 & 0 \\
        0 & 0 & 0
    \end{pmatrix}
\end{equation*}
with $\tau>0$.
The corresponding stress response in the crack in this model is
\begin{equation}
    \bfsigma_{split}
    = 
    \frac{\tau}{2}
    \begin{pmatrix}
        -1 & 1 & 0\\
        1  & -1 & 0 \\
        0 & 0 & 0
    \end{pmatrix}
\end{equation}
which predicts a shear traction across the crack faces, which violates the classical crack face traction conditions.

\subsection{The Hydrostatic-Deviatoric Split}

\cite{amor-jmps} proposes a splitting that allows the crack to resist compressive hydrostatic stress, but not tensile hydrostatic and deviatoric stresses.
One starts by writing the isotropic linear elastic energy as
\begin{equation*}
    W_{intact}(\bfvareps) = \half\kappa\langle\trace\bfvareps\rangle_-^2 + \half\kappa\langle\trace\bfvareps\rangle_+^2 + \mu|\bfvareps_D|^2
\end{equation*}
where $\kappa = \lambda+2\mu/3 > 0$ is the bulk modulus and $\bfvareps_D := \bfvareps - \frac{1}{3}\left(\trace{\bfvareps}\right)\bfI$ is the deviatoric component of the strain.
The effective energy of the crack is assumed to consist only of the term $\half\kappa\langle\trace\bfvareps\rangle_-^2$, giving the stress reponse
\begin{equation*}
    \bfsigma_{split} = \left(\kappa\langle\trace\bfvareps\rangle_-\right)\bfI .
\end{equation*}
We notice from this expression that the crack response is always non-tensile hydrostatic, regardless of the applied loading.

\paragraph{Uniaxial tension parallel to the crack.}

Consider the specimen subject to a far-field stress $\sigma_0 \bfe_1 \otimes \bfe_1$ with $\sigma_0 > 0$, and the corresponding far-field strain is in \eqref{eqn:split-ex-1}.

In this model, we find $\bfsigma_{split} = {\bf 0}$. 
As described in Section \ref{subse:splitting}, the expectation is that material sustains the far-field tensile stress parallel to crack faces as it is. 
Further, the result that $\bfsigma_{split} = \bf0$ has similar implications: spurious stress concentrations at the crack tip that affects the crack driving force and the T-stress.

\paragraph{Uniaxial compression normal to the crack.}

Consider the specimen subject to a far-field stress $- \sigma_0 \bfe_2 \otimes \bfe_2$ with $\sigma_0 > 0$.
We expect the crack faces to contact and the response to be identical to the intact material.

The far-field strain corresponding to this stress is
\begin{equation*}
    \bfvareps = 
    \frac{\sigma_0}{\mu(3\lambda+2\mu)}
    \begin{pmatrix}
        \lambda/2 & 0 & 0 \\
        0 & -(\lambda+\mu) & 0 \\
        0 & 0 & \lambda/2
    \end{pmatrix} ,
\end{equation*}
which should also be response of the crack.
However, this model predicts the crack response $\bfsigma_{split} = -\frac{\sigma_0}{3}\bfI$.
Similar results can be obtained in the general case of multiaxial compressive loading.


\section{Proofs}\label{ap:proofs}

In this appendix, we provide the proofs of various statements in the main body of the work.

\begin{proof}[Proof of Proposition \ref{prop:QR}]
This is actually a restatement of the well-known QR decomposition (\cite[Th.\ 2.6.1]{HoJo90} or \cite[Thms.\ 7.1 and 7.2]{TrBa97}), according to which given any  $\bfF \in \R^{3 \times 3}_+$ there exist unique $\bfR \in SO(3)$ and $\bfA$ upper triangular (with respect to the basis $\{ \bft_1, \bft_2, \bfn\}$) with positive diagonal elements such that $\bfF = \bfR \bfA$.
This $\bfA$ must be of the form $\bfA_{\bft_1, \bft_2, \bfn} (A_{\bfn \bfn}, A_{\bft_1 \bft_1}, A_{\bft_2 \bft_2}, A_{\bft_1 \bfn}, A_{\bft_2 \bfn}, A_{\bft_1 \bft_2})$ for some $A_{\bfn \bfn} , A_{\bft_1 \bft_1}, A_{\bft_2 \bft_2} > 0$ and $A_{\bft_1 \bfn}, A_{\bft_2 \bfn}, A_{\bft_1 \bft_2} \in \R$.
\end{proof}

\begin{proof}[Proof of Lemma \ref{le:Ann}]
We have, succesively,
\[
 \bfF = \bfR \bfA , \qquad \bfF^{-T} = \bfR \bfA^{-T} , \qquad \bfF^{-T} \bfn = \bfR \bfA^{-T} \bfn , \qquad \left| \bfF^{-T} \bfn \right| = \left| \bfA^{-T} \bfn \right| .
\]
Now, since
\[
 \bfA = A_{\bfn \bfn} \bfn \otimes \bfn + A_{\bft_1 \bft_1} \bft_1 \otimes \bft_1 + A_{\bft_2 \bft_2} \bft_2 \otimes \bft_2 + A_{\bft_1 \bfn} \bft_1 \otimes \bfn + A_{\bft_2 \bfn} \bft_2 \otimes \bfn + A_{\bft_1 \bft_2} \bft_1 \otimes \bft_2 ,
\]
we can easily calculate
\[
 \bfA^T = A_{\bfn \bfn} \bfn \otimes \bfn + A_{\bft_1 \bft_1} \bft_1 \otimes \bft_1 + A_{\bft_2 \bft_2} \bft_2 \otimes \bft_2 + A_{\bft_1 \bfn} \bfn \otimes \bft_1 + A_{\bft_2 \bfn} \bfn \otimes \bft_2 + A_{\bft_1 \bft_2} \bft_2 \otimes \bft_1 ,
\]
as well as
\begin{align*}
 \bfA^{-T} = & \frac{1}{A_{\bfn \bfn}} \bfn \otimes \bfn + \frac{1}{A_{\bft_1 \bft_1}} \bft_1 \otimes \bft_1 + \frac{1}{A_{\bft_2 \bft_2}} \bft_2 \otimes \bft_2 + \left( - \frac{A_{\bft_1 \bfn}}{A_{\bft_1 \bft_1} A_{\bfn \bfn}} + \frac{A_{\bft_1 \bft_2} A_{\bft_2 \bfn}}{A_{\bft_1 \bft_1} A_{\bft_2 \bft_2} A_{\bfn \bfn}} \right) \bfn \otimes \bft_1 \\
 & - \frac{A_{\bft_2 \bfn}}{A_{\bft_2 \bft_2} A_{\bfn \bfn}} \bfn \otimes \bft_2 - \frac{A_{\bft_1 \bft_2}}{A_{\bft_1 \bft_1} A_{\bft_2 \bft_2}}  \bft_2 \otimes \bft_1 ,
\end{align*}
so
\[
 \bfA^{-T} \bfn = \frac{1}{A_{\bfn \bfn}} \bfn \quad \text{and} \quad  \left| \bfA^{-T} \bfn \right| = \frac{1}{A_{\bfn \bfn}} ,
\]
which concludes the proof.
\end{proof}

\begin{proof}[Proof of Proposition \ref{pr:alpha*}]
We will only prove \ref{item:alpha*}, the rest of the claims being analogous.
Let $( A_{\bfn \bfn}^{(j)} , A_{\bft_1 \bfn}^{(j)} , A_{\bft_2 \bfn}^{(j)} )$ be a sequence (indexed by $j \in \N$) in $(0, \infty) \times \R \times \R$ satisfying
\begin{align*}
 & \inf_{\substack{A_{\bfn \bfn} >0 \\ A_{\bft_1 \bfn}, A_{\bft_2 \bfn} \in \R}} W \left( \bfA_{\bft_1, \bft_2, \bfn} (A_{\bfn \bfn}, A_{\bft_1 \bft_1}, A_{\bft_2 \bft_2}, A_{\bft_1 \bfn}, A_{\bft_2 \bfn}, A_{\bft_1 \bft_2}) \right) \\
 & = \lim_{j \to \infty} W \left( \bfA_{\bft_1, \bft_2, \bfn} (A_{\bfn \bfn}^{(j)}, A_{\bft_1 \bft_1}, A_{\bft_2 \bft_2}, A_{\bft_1 \bfn}^{(j)}, A_{\bft_2 \bfn}^{(j)}, A_{\bft_1 \bft_2}) \right) .
\end{align*}
We have that
\[
 \det \bfA_{\bft_1, \bft_2, \bfn} (A_{\bfn \bfn}^{(j)}, A_{\bft_1 \bft_1}, A_{\bft_2 \bft_2}, A_{\bft_1 \bfn}^{(j)}, A_{\bft_2 \bfn}^{(j)}, A_{\bft_1 \bft_2}) = A_{\bft_1 \bft_1} A_{\bft_2 \bft_2} A_{\bfn \bfn}^{(j)}
\]
and
\begin{align*}
 &\left| \bfA_{\bft_1, \bft_2, \bfn} (A_{\bfn \bfn}^{(j)}, A_{\bft_1 \bft_1}, A_{\bft_2 \bft_2}, A_{\bft_1 \bfn}^{(j)}, A_{\bft_2 \bfn}^{(j)}, A_{\bft_1 \bft_2}) \right| \\
 & = \sqrt{A_{\bft_1 \bft_1}^2 + A_{\bft_1 \bft_2}^2 + (A_{\bft_1 \bfn}^{(j)})^2 + A_{\bft_2 \bft_2}^2 + (A_{\bft_2 \bfn}^{(j)})^2 + (A_{\bfn \bfn}^{(j)})^2} .
\end{align*}
Condition \eqref{eq:Winfty} implies that there exist $m, M >0$ such that
\[
 m \leq A_{\bfn \bfn}^{(j)} \leq M , \quad -M \leq A_{\bft_1 \bfn}^{(j)} \leq M , \quad -M \leq A_{\bft_2 \bfn}^{(j)} \leq M , \qquad j \in \N .
\]
Therefore, there exist $A_{\bfn \bfn}^* , A_{\bft_1 \bfn}^* , A_{\bft_2 \bfn}^*$ with
\[
 m \leq A_{\bfn \bfn}^* \leq M , \quad -M \leq A_{\bft_1 \bfn}^* \leq M , \quad -M \leq A_{\bft_2 \bfn}^* \leq M
\]
such that, for a subsequence (not relabelled),
\[
 \lim_{j \to \infty} ( A_{\bfn \bfn}^{(j)} , A_{\bft_1 \bfn}^{(j)} , A_{\bft_2 \bfn}^{(j)} ) = ( A_{\bfn \bfn}^* , A_{\bft_1 \bfn}^* , A_{\bft_2 \bfn}^* ) .
\]
As $W$ is continuous,
\begin{align*}
 & \lim_{j \to \infty} W \left( \bfA_{\bft_1, \bft_2, \bfn} (A_{\bfn \bfn}^{(j)}, A_{\bft_1 \bft_1}, A_{\bft_2 \bft_2}, A_{\bft_1 \bfn}^{(j)}, A_{\bft_2 \bfn}^{(j)}, A_{\bft_1 \bft_2}) \right) \\
 & = W \left( \bfA_{\bft_1, \bft_2, \bfn} (A_{\bfn \bfn}^*, A_{\bft_1 \bft_1}, A_{\bft_2 \bft_2}, A_{\bft_1 \bfn}^*, A_{\bft_2 \bfn}^*, A_{\bft_1 \bft_2}) \right)
\end{align*}
and the proof is complete.
\end{proof}

\begin{proof}[Proof of Proposition \ref{pr:Wdalt}]
Following the notation of Definition \ref{de:Wd}, we define, additionally,
\[
 \bfA' = \bfA_{\bft_1, \bft_2, \bfn} (A'_{\bfn \bfn}, A_{\bft_1 \bft_1}, A_{\bft_2 \bft_2}, A'_{\bft_1 \bfn}, A'_{\bft_2 \bfn}, A_{\bft_1 \bft_2}) .
\]
By frame-indifference, $W (\bfA') = W (\bfR \bfA')$, and $\bfR \bfA' = \bfR \bfA \bfA^{-1} \bfA' = \bfF \bfA^{-1} \bfA'$.
We calculate
\[
 \bfA^{-1} = \bfA_{\bft_1, \bft_2, \bfn} \left( \frac{1}{A_{\bfn \bfn}}, \frac{1}{A_{\bft_1 \bft_1}}, \frac{1}{A_{\bft_2 \bft_2}}, \frac{-A_{\bft_1 \bfn}}{A_{\bft_1 \bft_1} A_{\bfn \bfn}} + \frac{A_{\bft_1 \bft_2} A_{\bft_2 \bfn}}{A_{\bft_1 \bft_1} A_{\bft_2 \bft_2} A_{\bfn \bfn}}, - \frac{A_{\bft_2 \bfn}}{A_{\bft_2 \bft_2} A_{\bfn \bfn}}, - \frac{A_{\bft_1 \bft_2}}{A_{\bft_1 \bft_1} A_{\bft_2 \bft_2}} \right) .
\]
and
\begin{multline*}
 \bfA^{-1} \bfA' = \\
 \bfA_{\bft_1, \bft_2, \bfn} \left( \frac{A'_{\bfn \bfn}}{A_{\bfn \bfn}}, 1, 1, \frac{A'_{\bft_1 \bfn}}{A_{\bft_1 \bft_1}} - \frac{A_{\bft_1 \bft_2} A'_{\bft_2 \bfn}}{A_{\bft_1 \bft_1} A_{\bft_2 \bft_2}} - \frac{A_{\bft_1 \bfn} A'_{\bfn \bfn}}{A_{\bft_1 \bft_1} A_{\bfn \bfn}} + \frac{A_{\bft_1 \bft_2} A_{\bft_2 \bfn} A'_{\bfn \bfn}}{A_{\bft_1 \bft_1} A_{\bft_2 \bft_2} A_{\bfn \bfn}}, \frac{A'_{\bft_2 \bfn}}{A_{\bft_2 \bft_2}} - \frac{A_{\bft_2 \bft_3} A'_{\bfn \bfn}}{A_{\bft_2 \bft_2} A_{\bfn \bfn}}, 0 \right) .
\end{multline*}
Performing the changes
\begin{align*}
 & A''_{\bft_1 \bfn} := \frac{A'_{\bft_1 \bfn}}{A_{\bft_1 \bft_1}} - \frac{A_{\bft_1 \bft_2} A'_{\bft_2 \bfn}}{A_{\bft_1 \bft_1} A_{\bft_2 \bft_2}} - \frac{A_{\bft_1 \bfn} A'_{\bfn \bfn}}{A_{\bft_1 \bft_1} A_{\bfn \bfn}} + \frac{A_{\bft_1 \bft_2} A_{\bft_2 \bfn} A'_{\bfn \bfn}}{A_{\bft_1 \bft_1} A_{\bft_2 \bft_2} A_{\bfn \bfn}} , \qquad A''_{\bft_1 \bfn} := \frac{A'_{\bft_2 \bfn}}{A_{\bft_2 \bft_2}} - \frac{A_{\bft_2 \bft_3} A'_{\bfn \bfn}}{A_{\bft_2 \bft_2} A_{\bfn \bfn}} , \\
 & A''_{\bfn \bfn} := \frac{A'_{\bfn \bfn}}{A_{\bfn \bfn}} ,
\end{align*}
it is immediate to see that
\begin{align*}
 & \Big\{ \bfA_{\bft_1, \bft_2, \bfn} \left( \frac{A'_{\bfn \bfn}}{A_{\bfn \bfn}}, 1, 1, \frac{A'_{\bft_1 \bfn}}{A_{\bft_1 \bft_1}} - \frac{A_{\bft_1 \bft_2} A'_{\bft_2 \bfn}}{A_{\bft_1 \bft_1} A_{\bft_2 \bft_2}} + \frac{(-A_{\bft_1 \bfn} A_{\bft_2 \bft_2} + A_{\bft_1 \bft_2} A_{\bft_2 \bfn}) A'_{\bfn \bfn}}{A_{\bft_1 \bft_1} A_{\bft_2 \bft_2} A_{\bfn \bfn}}, \frac{A'_{\bft_2 \bfn}}{A_{\bft_2 \bft_2}} - \frac{A_{\bft_2 \bft_3} A'_{\bfn \bfn}}{A_{\bft_2 \bft_2} A_{\bfn \bfn}}, 0 \right) \\
 & \hspace{30em} : \, A'_{\bfn \bfn} > 0, \, A'_{\bft_1 \bfn} , A'_{\bft_2 \bfn} \in \R \Big\} \\
 & =  \left\{ \bfA_{\bft_1, \bft_2, \bfn} \left( A''_{\bfn \bfn}, 1, 1, A''_{\bft_1 \bfn}, A''_{\bft_2 \bfn}, 0 \right) : \, A''_{\bfn \bfn} > 0, \, A''_{\bft_1 \bfn} , A''_{\bft_2 \bfn} \in \R \right\} .
\end{align*}
This shows that
\begin{align*}
 & \min_{\substack{A'_{\bfn \bfn} >0 \\ A'_{\bft_1 \bfn}, A'_{\bft_2 \bfn} \in \R}} W \left( \bfA_{\bft_1, \bft_2, \bfn} (A'_{\bfn \bfn}, A_{\bft_1 \bft_1}, A_{\bft_2 \bft_2}, A'_{\bft_1 \bfn}, A'_{\bft_2 \bfn}, A_{\bft_1 \bft_2}) \right) \\
 & = \min_{\substack{A''_{\bfn \bfn} >0 \\ A''_{\bft_1 \bfn}, A''_{\bft_2 \bfn} \in \R}} W \left( \bfF \bfA_{\bft_1, \bft_2, \bfn} \left( A''_{\bfn \bfn}, 1, 1, A''_{\bft_1 \bfn}, A''_{\bft_2 \bfn}, 0 \right) \right)
\end{align*}
and formula \eqref{eq:Anns*} holds.
Analogously, one can show that
\begin{align*}
 & \min_{A'_{\bft_1 \bfn}, A'_{\bft_2 \bfn} \in \R} W \left( \bfA_{\bft_1, \bft_2, \bfn} (A_{\bfn \bfn}, A_{\bft_1 \bft_1}, A_{\bft_2 \bft_2}, A'_{\bft_1 \bfn}, A'_{\bft_2 \bfn}, A_{\bft_1 \bft_2}) \right) \\
 & = \min_{A''_{\bft_1 \bfn}, A''_{\bft_2 \bfn} \in \R} W \left( \bfF \bfA_{\bft_1, \bft_2, \bfn} \left( A_{\bfn \bfn}, 1, 1, A''_{\bft_1 \bfn}, A''_{\bft_2 \bfn}, 0 \right) \right) .
\end{align*}
\end{proof}

\begin{proof}[Proof of Proposition \ref{pr:invariance}]
We will use Proposition \ref{pr:Wdalt} and trace the dependence of the quantities involved.

We start with \ref{item:Wcompa2}.
According to \eqref{eq:Ann}, it is clear that $A_{\bfn \bfn}$ does not depend on $\bft_1, \bft_2$.
Let us see that $A_{\bfn \bfn}^*$ does not depend either, so let  $\{ \bft'_1, \bft'_2 , \bfn \}$ be another orthonormal basis.
We first notice that
\[
 \bft_1 \otimes \bft_1 + \bft_2 \otimes \bft_2 = \bft'_1 \otimes \bft'_1 + \bft'_2 \otimes \bft'_2 ,
\]
since both terms act as the identity in $\spn \{ \bft_1, \bft_2 \} = \spn \{ \bft'_1, \bft'_2 \}$ and as zero in $\spn\{ \bfn \}$.
On the other hand, there exist an invertible matrix
\begin{equation}\label{eq:changebasis}
 \begin{pmatrix}
 a_{11} & a_{12} \\
 a_{21} & a_{22}
 \end{pmatrix}
\end{equation}
such that
\[
 \begin{cases}
 \bft_1 = a_{11} \bft'_1 + a_{12} \bft'_2 \\
 \bft_2 = a_{21} \bft'_1 + a_{22} \bft'_2 .
 \end{cases}
\]
With this we find that for any $A''_{\bfn \bfn} > 0$ and $A''_{\bft_1 \bfn}, A''_{\bft_2 \bfn} \in \R$ we have
\begin{align*}
 & A''_{\bfn \bfn} \bfn \otimes \bfn + \bft_1 \otimes \bft_1 + \bft_2 \otimes \bft_2 + A''_{\bft_1 \bfn} \bft_1 \otimes \bfn + A''_{\bft_2 \bfn} \bft_1 \otimes \bfn \\
 & = A''_{\bfn \bfn} \bfn \otimes \bfn + \bft'_1 \otimes \bft'_1 + \bft'_2 \otimes \bft'_2 + \left( a_{11} A''_{\bft_1 \bfn} + a_{21}  A''_{\bft_2 \bfn} \right) \bft_1 \otimes \bfn + \left( a_{12} A''_{\bft_1 \bfn} + a_{22}  A''_{\bft_2 \bfn} \right) A''_{\bft_2 \bfn} \bft_1 \otimes \bfn .
\end{align*}
Thus, performing the changes
\[
 \begin{cases}
 A''_{\bft'_1 \bfn} = a_{11} A''_{\bft_1 \bfn} + a_{21}  A''_{\bft_2 \bfn} \\
 A''_{\bft'_2 \bfn} = a_{12} A''_{\bft_1 \bfn} + a_{22}  A''_{\bft_2 \bfn} ,
 \end{cases}
\]
we have shown that
\[
 \bfA_{\bft_1, \bft_2, \bfn} \left( A''_{\bfn \bfn}, 1, 1, A''_{\bft_1 \bfn}, A''_{\bft_2 \bfn}, 0 \right) = \bfA_{\bft'_1, \bft'_2, \bfn} \left( A''_{\bfn \bfn}, 1, 1, A''_{\bft'_1 \bfn}, A''_{\bft'_2 \bfn}, 0 \right) .
\]
Having in mind that the matrix \eqref{eq:changebasis} is invertible, this shows that 
\begin{align*}
 & \left\{ \bfA_{\bft_1, \bft_2, \bfn} \left( A''_{\bfn \bfn}, 1, 1, A''_{\bft_1 \bfn}, A''_{\bft_2 \bfn}, 0 \right) : \, A''_{\bfn \bfn} > 0, \, A''_{\bft_1 \bfn} , A''_{\bft_2 \bfn} \in \R \right\} \\
 & = \left\{ \bfA_{\bft'_1, \bft'_2, \bfn} \left( A''_{\bfn \bfn}, 1, 1, A''_{\bft'_1 \bfn}, A''_{\bft'_2 \bfn}, 0 \right) : \, A''_{\bfn \bfn} > 0, \, A''_{\bft'_1 \bfn} , A''_{\bft'_2 \bfn} \in \R \right\}
\end{align*}
and, analogously,
\begin{align*}
 & \left\{ \bfA_{\bft_1, \bft_2, \bfn} \left( A''_{\bfn \bfn}, 1, 1, A''_{\bft_1 \bfn}, A''_{\bft_2 \bfn}, 0 \right) : \, A''_{\bft_1 \bfn} , A''_{\bft_2 \bfn} \in \R \right\} \\
 & = \left\{ \bfA_{\bft'_1, \bft'_2, \bfn} \left( A''_{\bfn \bfn}, 1, 1, A''_{\bft'_1 \bfn}, A''_{\bft'_2 \bfn}, 0 \right) : \, A''_{\bft'_1 \bfn} , A''_{\bft'_2 \bfn} \in \R \right\} .
\end{align*}
This shows that $A_{\bfn \bfn}^*$ and $\Wd$ are the same for both basis, so proving \ref{item:Wcompa2}.

Now we show \ref{item:Wcompa3}.
Changing $\bfn$ with $-\bfn$ does not alter $A_{\bfn \bfn}$, as can be shown from \eqref{eq:Ann}.
In order to show that the rest of the quantities remain equal under this change, we first notice that $\{ -\bft_1, -\bft_2 , -\bfn \}$ is also an orthonormal basis and, for all $A_{\bfn \bfn}, A_{\bft_1 \bft_1}, A_{\bft_2 \bft_2} > 0$ and $A_{\bft_1 \bfn}, A_{\bft_2 \bfn}, A_{\bft_1 \bft_2} \in \R$,
\[
 \bfA_{\bft_1, \bft_2, \bfn} (A_{\bfn \bfn}, A_{\bft_1 \bft_1}, A_{\bft_2 \bft_2}, A_{\bft_1 \bfn}, A_{\bft_2 \bfn}, A_{\bft_1 \bft_2}) = \bfA_{-\bft_1, -\bft_2, -\bfn} (A_{\bfn \bfn}, A_{\bft_1 \bft_1}, A_{\bft_2 \bft_2}, A_{\bft_1 \bfn}, A_{\bft_2 \bfn}, A_{\bft_1 \bft_2}) .
\]
This formula, together with \eqref{eq:Anns*}, show that $A_{\bfn \bfn}^*$ is the same for $\bfn$ and $- \bfn$.
In fact, with Proposition \ref{pr:invariance} it also implies the conclusion of \ref{item:Wcompa3}.

We now show \ref{item:Wcompa4}, so let $\bfQ \in SO(3)$.
From \eqref{eq:Ann} we see immediately that $A_{\bfn \bfn}$ does not change when $\bfF$ is replaced with $\bfQ \bfF$, since $|(\bfQ \bfF)^{-T} \, \bfn | = |\bfQ^{-T} \bfF^{-T} \, \bfn | = |\bfF^{-T} \, \bfn |$.
Now, in view of \eqref{eq:Anns*} and Proposition \ref{pr:invariance}, the frame-indifference of $W$ readily implies that $A_{\bfn \bfn}^*$ and the whole $\Wd$ is the same for $\bfF$ and $\bfQ \bfF$.

In order to show \ref{item:Wcompa45}, so let $\bfQ \in \mathcal{S}$.
When the pair $(\bfF, \bfn)$ is replaced with $(\bfF \bfQ^T, \bfQ \bfn)$, the last expression of \eqref{eq:Ann} does not change, since, $|(\bfF \bfQ^T)^{-T}  \bfQ \bfn | = |\bfF^{-T} \bfQ^{-1} \bfQ \bfn | = |\bfF^{-T} \, \bfn |$.
This shows that $A_{\bfQ \bfn \, \bfQ \bfn} = A_{\bfn \bfn}$.
In order to show that the rest of the quantities remain equal under this change, we first notice that $\{ \bfQ \bft_1, \bfQ \bft_2 , \bfQ \bfn \}$ is also an orthonormal basis and, for all $A_{\bfn \bfn}, A_{\bft_1 \bft_1}, A_{\bft_2 \bft_2} > 0$ and $A_{\bft_1 \bfn}, A_{\bft_2 \bfn}, A_{\bft_1 \bft_2} \in \R$,
\begin{align*}
 & \bfF \bfA_{\bft_1, \bft_2, \bfn} (A_{\bfn \bfn}, A_{\bft_1 \bft_1}, A_{\bft_2 \bft_2}, A_{\bft_1 \bfn}, A_{\bft_2 \bfn}, A_{\bft_1 \bft_2}) \\
 & = \bfF \bfQ^T \bfA_{\bfQ \bft_1, \bfQ \bft_2, \bfQ \bfn} (A_{\bfn \bfn}, A_{\bft_1 \bft_1}, A_{\bft_2 \bft_2}, A_{\bft_1 \bfn}, A_{\bft_2 \bfn}, A_{\bft_1 \bft_2}) \bfQ
\end{align*}
and, hence, since $\bfQ$ is a symmetry for $W$,
\begin{align*}
 & W \left( \bfF \bfA_{\bft_1, \bft_2, \bfn} (A_{\bfn \bfn}, A_{\bft_1 \bft_1}, A_{\bft_2 \bft_2}, A_{\bft_1 \bfn}, A_{\bft_2 \bfn}, A_{\bft_1 \bft_2}) \right) \\
 & = W \left( \bfF  \bfQ^T \bfA_{\bfQ \bft_1, \bfQ \bft_2, \bfQ \bfn} (A_{\bfn \bfn}, A_{\bft_1 \bft_1}, A_{\bft_2 \bft_2}, A_{\bft_1 \bfn}, A_{\bft_2 \bfn}, A_{\bft_1 \bft_2}) \right) .
\end{align*}
This shows that $A_{\bfQ \bfn \, \bfQ \bfn}^* = A_{\bfn \bfn}^*$ and $\Wd (\bfF \bfQ^T, \bfQ \bfn) = \Wd (\bfF, \bfn)$, which completes the proof of \ref{item:Wcompa45}.

Finally, \ref{item:Wcompa5} is a particular case of \ref{item:Wcompa45} when $\mathcal{S} = SO(3)$.
\end{proof}

\begin{proof}[Proof of Lemma \ref{le:posdef}]
Taking $\bfvareps$ as
\[
 \vareps_{11} = 0 , \qquad \vareps_{12} = \frac{1}{2} , \qquad \vareps_{22} = 0
\]
and using $\C \bfvareps : \bfvareps >0$, we find that $c_{1212} > 0$.
Analogously, taking
\[
 \vareps_{11} = 0 , \qquad \vareps_{12} = -  \frac{c_{1222}}{2\sqrt{c_{1212}}} , \qquad \vareps_{22} = \sqrt{c_{1212}} ,
\]
we find that $c_{1212} c_{2222} - c_{1222}^2 > 0$.
\end{proof}



\end{document}